\documentclass[12pt]{article}
\usepackage{csquotes}

\usepackage{bbm}
\usepackage[english]{babel}
\usepackage{amssymb}
\usepackage{mathtools}
\usepackage{graphicx}
\usepackage{caption}
\usepackage{subcaption}
\usepackage{epsfig}
\usepackage{amsmath}
\allowdisplaybreaks
\usepackage{amsthm}
\usepackage{setspace}
\usepackage{mathrsfs}
\usepackage{enumitem}
\usepackage{authblk}
\usepackage{color}
\usepackage{tcolorbox}
\usepackage[ruled,vlined]{algorithm2e}
\SetKwProg{initialize}{Initialize:}{}{}
\SetKw{KwBy}{by}
\usepackage{gensymb}
\usepackage{multirow}
\usepackage{adjustbox}
\usepackage[graphicx]{realboxes}
\usepackage{longtable}
\usepackage{hhline}
\usepackage{geometry}
\usepackage{lipsum}                     % Dummytext
\usepackage{xargs}                      % Use more than one optional parameter in a new commands

\usepackage[colorinlistoftodos,prependcaption,textsize=tiny]{todonotes}
\newcommandx{\unsure}[2][1=]{\todo[linecolor=red,backgroundcolor=red!25,bordercolor=red,#1]{#2}}
\newcommandx{\change}[2][1=]{\todo[linecolor=blue,backgroundcolor=blue!25,bordercolor=blue,#1]{#2}}
\newcommandx{\info}[2][1=]{\todo[linecolor=OliveGreen,backgroundcolor=OliveGreen!25,bordercolor=OliveGreen,#1]{#2}}
\newcommandx{\improvement}[2][1=]{\todo[linecolor=Plum,backgroundcolor=Plum!25,bordercolor=Plum,#1]{#2}}
\newcommandx{\thiswillnotshow}[2][1=]{\todo[disable,#1]{#2}}

\theoremstyle{definition}

\newtheorem{thm}{Theorem}

\newtheorem{lem}{Lemma}

\definecolor{ForestGreen}{RGB}{34,139,34}
 
\usepackage{float}
\restylefloat{table}
\usepackage[backend=bibtex]{biblatex}
\usepackage{booktabs}

\newcounter{example}[section]

\usepackage{nomencl}
\makenomenclature

\usepackage{multicol}
\usepackage{color}

\usepackage{comment}

\title{
Optimal strategies for the growth of dual-seeded lattice structures
}
\date{\today}

\begin{document}
\newcommand{\edt}[1]{{\vbox{ \hbox{#1} \vskip-0.3em \hrule}}}
\newcommand{\RR}{{\mathbb R}}
\newcommand{\QQ}{{\mathbb Q}}
\newcommand{\NN}{{\mathbb N}}
\newcommand{\cf}[1]{{\bf{#1}}}
\newcommand{\EE}{{\mathbb E}}
\newcommand{\PP}{{\mathbb P}}
\newcommand{\cA}{{\mathcal A}}
\newcommand{\cX}{{\mathcal X}}
\newcommand{\cH}{{\mathcal H}}

\newcommand{\expect}{\operatorname{\EE}\expectarg}
\DeclarePairedDelimiterX{\expectarg}[1]{[}{]}{%
  \ifnum\currentgrouptype=16 \else\begingroup\fi
  \activatebar#1
  \ifnum\currentgrouptype=16 \else\endgroup\fi
}

\newcommand{\innermid}{\nonscript\;\delimsize\vert\nonscript\;}
\newcommand{\activatebar}{%
  \begingroup\lccode`\~=`\|
  \lowercase{\endgroup\let~}\innermid 
  \mathcode`|=\string"8000
}

\hyphenation{high-di-men-sio-nal}

\author{
\begin{tabular}{ccc}
Maike C.\ de Jongh$^{1}$ 
Cristian Spitoni$^{2}$ 
Emilio N.\ M.\ Cirillo$^{3}$
\end{tabular}
}

\date{}

\maketitle

\begin{center}
\small
$^{1}$Department of Applied Mathematics, University of Twente,\\
P.O. Box 217, NL--7500 AE Enschede, The Netherlands\\[0.5em]

$^{2}$Mathematics Department, Utrecht University,\\
Budapestlaan 6, 3584~CD Utrecht, The Netherlands\\[0.5em]

$^{3}$Dipartimento SBAI, Sapienza Università di Roma,\\
Via A.\ Scarpa 16, I--00161 Rome, Italy
\end{center}

\begin{abstract}
Optimal growth of structures governed by spatially
stochastic dynamics arises in many scientific settings,
for example in processes such as solution-based
crystallization and the formation of microbial biofilms
on patterned substrates or microfluidic networks. In
this work, we investigate lattice growth using a
two-dimensional, zero-temperature stochastic model of
short-range spin interactions. Our goal is to determine
how external perturbations can be optimized to steer the
system efficiently toward the uniformly positive state,
starting from two initial clusters of positive sites. To
achieve this, we cast the problem as a Markov decision
process adapted for a two-dimensional Ising model with zero-temperature dynamics. Within this framework, we compare
alternative growth geometries and identify the structure
of optimal strategies across three representative
regimes.


    %Optimal growth of structures subject to spatially stochastic dynamics arises naturally in many scientific contexts, including solution-based crystallization and microbial biofilms on structured substrates or microfluidic grids. In this work, we model the growth of lattice structures using the zero-temperature 2D stochastic Ising model on the square lattice with periodic boundary conditions and Metropolis dynamics. We study the nucleation trajectory to the all-plus configuration when the system is initially seeded with two small droplets of spin $+1$ sites and is subject to external perturbations. To optimize this trajectory, we formulate the problem as a Markov decision process (MDP). Using this MDP framework, we compare the performance of different growth geometries and characterize the structure of optimal growth strategies in three representative regimes. 
\end{abstract}

\section{Introduction}
\label{Introduction}
The optimization of structure growth in spatially
constrained domains is a compelling challenge across
many scientific fields. In such systems, a lattice or
spatial grid defines a substrate on which a structure
evolves by its intrinsic growth dynamics, yet the
outcome may be significantly improved by external
interventions—such as seeding nuclei, applying
external fields, or modulating resource fluxes—to
steer the growth process toward desired morphologies,
sizes, or uniformities. The fundamental questions
concern how internal kinetics (nucleation, diffusion,
aggregation) couple with externally applied controls,
and how one can design minimal control policies to
achieve targeted structural objectives under cost or
resource constraints. Approaches from statistical
physics, materials science, and bioengineering
converge in tackling these problems, and recent
studies emphasize the importance of spatiotemporal
control in addition to steady‐state modulation.

The growth of crystals via controlled nucleation
provides a clear illustration of this paradigm. In
solution-based crystallization, for example,
externally introduced nuclei or local heating may
allow one to reduce excessive nucleation and promote
growth of fewer, larger crystals with improved
quality and functionality. In a recent study, the
authors used near-infrared laser heating to locally
modulate supersaturation, thereby directing the
nucleation and growth of calcium carbonate crystals
into user-defined patterns \cite{Bistervels2023}.
Such interventions highlight how the interplay
between internal kinetics and external actuation can
be exploited to optimize performance. 

A second example arises in the context of microbial
biofilms grown on structured substrates or
microfluidic grids. Even when the microbial community
exhibits its intrinsic growth, nutrient consumption,
and structural dynamics, external seeding of access
points, nutrient injection, or substrate patterning
may considerably influence the ultimate size,
vertical height, and uniformity of the biofilm.
By explicitly controlling the spatial distribution of bacterial
cells during the initial inoculation, it is possible to steer
the development and architecture of a biofilm. 
In \cite{Jin2024} this principle is 
demonstrated using an optogenetic toolkit,
termed ``Multipattern Biofilm Lithography,'' which enables
precise, orthogonal patterning of multi-strain biofilms
\cite{Jin2024}. By temporally controlled light signals, the authors were able
to shape both the structure and functional properties of the
resulting biofilms. 

A mathematical framework in which this problem can be
approached is that of Markov Decision Processes (MDP)
\cite{Puterman}. Given
a system with its own state space, whose 
evolution is governed by a prescribed
dynamics, we assume that external \textit{actions} are
performed on it, controlled in both time and
space. The manner in which the external effects are
applied to the natural dynamics of the system is
called a \textit{policy} and is characterized by a
deterministic temporal schedule and a stochastic
selection among possible actions in different spatial
locations. The times at which external actions are
taken are called \textit{epochs}. At each epoch,
depending on the state of the system, a \textit{reward}
is assigned to the process. Its average, weighted with 
a suitable \textit{discount factor}, is
called the \textit{value function}. An \textit{optimal
policy} is a policy that maximizes the value function.

Markov Decision Processes were introduced in the context of
dynamic programming in the 1950s. We refer, for instance,
to the book \cite{Bellman1957}, in which the idea of optimal
policies is clearly presented, and to \cite{Howard1960} for the
introduction of the concept of these processes in relation
to optimal decision theory.
Since then, MDP theory has been applied to a variety
of contexts \cite{Puterman}. Here, we mention a few studies
that are in the same spirit as our investigation.

MDPs have proven to be a powerful and versatile modelling framework when one must make sequential decisions under uncertainty, particularly in systems characterized by stochastic propagation, spatial or networked structure, and limited intervention resources. 
For example, in wildfire management the challenge of allocating scarce firefighting assets across a spatial substrate subject to random ignition and spread has motivated the formulation of both MDP and partially observable MDP (POMDP) models: recent work uses POMDPs for resource deployment when fire state is incompletely observed~\cite{Diao2020,Altamimi2022}. 
Similarly, large-scale forest management has been cast as a spatial MDP for optimization of thinning, harvesting or suppression policies under budget and risk constraints~\cite{Haksar2019,Haksar2020}. 
In the domain of networked contagion, MDP-based strategies have been developed to determine optimal intervention timing (e.g.,  as vaccination, awareness, treatment or quarantine) under constrained budgets and uncertain transition dynamics; moreover, robust or distributionally robust MDPs extend this to cope with ambiguity in disease transmission parameters~\cite{Nasir2023,Song2023}. 
In materials science, the process of steering colloidal particles to assemble into defect-free crystalline structures has been controlled via MDPs, where the time-evolving configuration of the ensemble is regarded as a Markov chain and external field inputs are the control actions~\cite{Tang2016,Tang2017}. 
More abstractly, the question of where to allocate limited resources (samples, computational budget, sensors) within a decision-making model itself has been addressed via an MDP-centric approach, in which exploration and exploitation are balanced to reduce policy uncertainty or approximation error under a limited exploration budget~\cite{Munos2001}. 
In the social and information sciences, Ni~\cite{Ni2017} formulated sequential influence diffusion as an MDP, demonstrating how adaptive, stage-wise seeding can maximize spread efficiency under limited marketing capital. 

These domains, though diverse, share a common structure: a stochastic process propagating over a spatial or networked substrate, coupled with interventions that must be selected optimally in time and space.
In the same spirit, our present work examines a seeded lattice-growth problem that can naturally be formulated as a Markov decision process. 
Here, droplets of plus spins correspond to scarce intervention resources (seed insertions) that are deployed sequentially in time, while the lattice configuration evolves stochastically according to the zero-temperature Ising dynamics. 
The goal is to drive the system toward the absorbing all-plus configuration while minimizing the cost or duration of the intervention. 
This analogy connects directly with wildfire suppression (where retardant is deployed to contain a spreading front), epidemic control (where vaccines or awareness campaigns are administered in stages), and colloidal self-assembly (where external fields are tuned to drive structural order). 
In each case, the optimal policy must balance immediate resource expenditure against the future stochastic benefit of accelerating the desired transition. 
Thus, the MDP framework provides a unifying formalism and computational tool-set for sequential intervention in stochastic growth or spreading processes with limited resources and naturally motivates our approach to lattice growth optimization.

Focusing to our  problem of the growth of lattice
structures, building on \cite{deJongh2025}, we use as a modeling
tool the simplest possible dynamics, namely the 
zero-temperature 2D stochastic Ising model on the square lattice
with periodic boundary conditions and Metropolis dynamics.

The question we address can be formulated as follows: starting
the system from the fully minus state, in which a single small
square droplet of plus spins is inserted, and assuming that the
magnetic field is positive and small, we aim to describe the
transition to the all-plus state. At zero temperature, the
natural dynamics cannot induce an exit from the initial
configuration; therefore, we act externally by adding plus spins
according to a prescribed rule. After each external insertion of
a plus spin, the system evolves following the Metropolis dynamics,
quickly reaching a local minimum of the Hamiltonian. There it will
remain stuck, waiting for the next external update.

In \cite{deJongh2025}, for the scenario described above,
it has been proven that, by choosing as reward function one
if the configuration is all-plus and zero otherwise,
the optimal policy resembles the manner in which
Metropolis dynamics would trigger the transition at low
temperature, although diagonal growth is preferred to growth
orthogonal to the rectangle sides.
Such a problem, indeed,
has been widely studied in the context of metastability theory.
We refer, for instance, to \cite{neves1991critical} for an early
study in the case of the 2D Ising model, and to the papers
\cite{cirillo1998metastability, cirillo2022review, boviermanzo2002metastability}
where it was further investigated.

The fact that similar optimal trajectories are observed when
approaching the problem from these different points of view is
not entirely obvious. Indeed, while in the metastability setup
the Metropolis dynamics updates the state based on knowledge of
the current configuration, in the MDP
approach the external plus feeding is carefully tailored based
on knowledge of the entire process and aimed at optimizing the
full trajectory.

The problem addressed in the present paper concerns the
optimization of the all-plus configuration growth when the system
is initially seeded with two separate small droplets. This question
is no longer directly related to the metastability problem. In that
setup, the small droplets immediately disappear due to thermal
fluctuations, and the trajectory leading to the all-plus configuration
is essentially a stochastic trajectory starting from the all-minus
state.

In our MDP approach, on the other hand,
considering the zero-temperature Metropolis dynamics, the initial
seeds are not destroyed by the dynamics. %The question therefore
%becomes whether it is more convenient to grow only one droplet,
%waiting for the second to be essentially phagocytized by the first,
%or to grow both droplets simultaneously until they coalesce.
Interesting questions in this setting concern the details of the growth
mechanisms: the droplets could expand toward each other, in the
same direction, or in orthogonal directions. All these possibilities
must be carefully analyzed, and the MDP provides
a systematic tool to select the optimal strategy in view of the
chosen reward function.

We investigate the structure of the
optimal policy in three representative regimes of the two-seed
problem: stripe--stripe, stripe--droplet, and droplet--droplet.
For each case we construct an auxiliary MDP that drastically
reduces the configuration space and allows a controlled
comparison among a small set of candidate policies.

In the stripe--stripe regime, our computations show that the
two main policy classes, acting at distance~1 or distance~2,
achieve very similar values. Nevertheless, by combining the
numerical evidence with the analytical results of Section~4,
we identify a sharp transition at the critical discount factor $\lambda_{c}=15/17$. For $\lambda>\lambda_{c}$ the distance--1
policy is optimal, as it minimizes the hitting time to the
all--plus configuration; for $\lambda<\lambda_{c}$ the
distance--2 policy becomes preferable, since it generates very
fast trajectories toward absorption.

In the stripe--droplet regime, the auxiliary MDPs reveal a
richer competition among growth geometries. Simulations of
four candidate policies indicate that the picture
observed in the stripe--stripe case persists, namely,
rapid front expansion is advantageous.
In the droplet--droplet regime, the results show that
policies prioritizing diagonal growth in the region separating
the two clusters most effectively reduce the hitting time,
while policies acting on wider regions are less efficient.

The paper is organized as follows. In Section~\ref{Preliminaries},
we define the model and outline our strategy. In
Section~\ref{s:numerical_study}, we present our numerical
results for the case in which both initial seeds are striped,
as well as for those in which at least one of the two initial
droplets is not a stripe. 
In Section~\ref{s:rigorous}, we rigorously analyze the
two-stripe case and clarify some of the points discussed in
the preceding section. Finally, Section~\ref{s:conclusion}
provides brief concluding remarks.

\section{The model}
\label{Preliminaries}
We first briefly recall the definition of the two-dimensional Ising model to
fix the notation and parameters, and then we introduce the MDP.

\subsection{The stochastic two-dimensional Ising model at zero temperature}
\label{s:ising}

We consider the two-dimensional Ising model on the finite square lattice $\Lambda = \{0, \ldots, N-~1\}^2$ with periodic boundary conditions. 
We equip the lattice with a distance measure $\delta: \Lambda \times \Lambda \rightarrow \mathbb{N}_0$ given by
\begin{equation*}
    \delta(i,j) = \min\{|j_1-i_1|, N - |j_1-i_1|\} + \min\{|j_2-i_2|, N-|j_2-i_2|\}, 
\end{equation*}
where $i=(i_1,i_2)$ and $j=(j_1,j_2)$ are sites of $\Lambda$. Note that the definition incorporates a torus edge correction. Let $P(\Lambda)$ denote the power set of $\Lambda$. By misusing the notation, we define a distance measure $\delta: P(\Lambda) \times P(\Lambda) \rightarrow \mathbb{N}_0$ between subsets of the lattice as
\begin{equation*}
    \delta(W_1, W_2) = \min_{(x,y) \in W_1, (x',y') \in W_2} \delta((x,y), (x',y')), \quad W_1, W_2 \subseteq \Lambda.
\end{equation*}
%In a similar way, we define the horizontal  distance measures $\delta_h$  between two 
%sites of the lattice, but we set it equal to infinity if the two sites are not on the same horizontal line, namely, we set 

%$\delta_h: \Lambda \times \Lambda \rightarrow \mathbb{N}^{\infty}_0$ and $\tilde{\delta}_h: P(\Lambda) \times P(\Lambda) \rightarrow \mathbb{N}^{\infty}_0$ as
%\begin{align*}
%   &\delta_h(i,j) = \begin{cases}
%            \min\{|j_1-i_1|, N-|j_1-i_1|\}, &\text{if } i_2 = j_2\\           \infty, &\text{otherwise.}
 %     \end{cases} \\
% &\tilde{\delta}_h((x,y), W) = \min_{(x',y') \in V}\delta_h((x,y), (x',y')), \quad (x,y) \in \Lambda, \quad W \subseteq \Lambda.   
%\end{align*}
%The vertical distance measure $\delta_v$ 
%is defined similarly. The extensions to subsets of the lattice are straightforward.
%: \Lambda \times \Lambda \rightarrow \mathbb{N}^{\infty}_0$ and $\tilde{\delta}_v: P(\Lambda) \times P(\Lambda)$ are defined analogously. 

We say that $i,j\in\Lambda$ are \textit{neighbors} or 
\textit{nearest neighbors} if and only if $\delta(i,j)=1$.
Let $N_h(i) \subseteq \Lambda$ and $N_v(i) \subseteq \Lambda$ denote the sets of horizontal and vertical neighbors of a spin $i \in \Lambda$; moreover,
$N(i)=N_h(i)\cup N_v(i)$.
We say that $W\subset\Lambda$ is 
\textit{connected} if and only if for any $i,j\in W$
there exists a sequence $i_1,\dots,i_n$ of sites
in $W$ such that $i_1=i$, $i_n=j$, and $i_k$ is  a
nearest neighbor of $i_{k+1}$ for all 
$k=1,\dots,n-1$.

To each site, we associate the spin variable $\sigma(i) \in \{-1, +1\}$. 
We denote the configuration space by $S = \{-1, +1\}^{\Lambda}$ and 
the Hamiltonian  by
\begin{equation*}
    H(\sigma) = -\sum\limits_{\substack{i \in \Lambda \\ j \in N(i)}} \sigma(i)\sigma(j) - h \sum\limits_{i \in \Lambda} \sigma(i),
\end{equation*}
where $N(i)$ denotes the set of horizontal and vertical neighbours of a site $i \in \Lambda$ and $h \in (0,1)$ denotes the external magnetic field. We denote by $\sigma^i$ the configuration obtained by flipping the
spin at site $i \in \Lambda$ starting from a configuration
$\sigma \in S$. Similarly, the configuration resulting from flipping
all spins in a set $W \subseteq \Lambda$ is denoted by $\sigma^W$.

We assume that the model evolves according to the zero-temperature 
Metropolis dynamics, i.e., as a discrete-time Markov chain $\{X_t\}_{t \geq 0}$ on $S$ with transition probabilities
%\begin{equation*}
%    p_{\beta}(\sigma, \sigma^i) = \begin{cases}
%        1/N^2, & \text{if } H(\sigma^i) \leq H(\sigma),\\
%        (1/N^2)\exp(-\beta(H(\sigma^i) - H(\sigma))), & \text{otherwise,}
%    \end{cases}
%\end{equation*}
%and $p_{\beta}(\sigma, \sigma) = 1-\sum\limits_{i \in \Lambda} p_{\beta}(\sigma, \sigma^i)$, $\sigma \in S$, where $\beta > 0$ represents %the inverse temperature. The Metropolis dynamics is reversible with respect to the Gibbs measure on $S$, defined as
%\begin{equation*}
%    \mu_{\beta}(\sigma) = Z^{-1}_{\beta}\exp(-\beta H(\sigma)), \quad \sigma \in S,
%\end{equation*}
%with partition function 
%\begin{equation*}
%    Z_{\beta} = \sum\limits_{\sigma \in S} \exp(-\beta H(\sigma)).
%\end{equation*}
%We study the model in the zero-temperature regime. Let the zero-temperature Metropolis dynamics be given by
\begin{equation*}
    p(\sigma, \sigma^i) = \begin{cases}
        1/N^2, & \text{if } H(\sigma^i) \leq H(\sigma),\\
        0, & \text{otherwise,}
    \end{cases}, 
\end{equation*}
and $p(\sigma, \sigma) = 1- \sum\limits_{i \in \Lambda} \tilde{p}(\sigma, \sigma^i)$, $\sigma \in S$.

Given a sequence of configurations $\omega = (\sigma_0, \sigma_1, \ldots, \sigma_{\ell})$, $\ell \in \mathbb{N}$, 
let $p(\omega)$ denote the probability that this sequence occurs under the zero-temperature Metropolis dynamics, i.e.,
\begin{equation*}
    p(\omega) = \prod\limits_{k = 1}^{\ell-1} p(\sigma_k, \sigma_{k+1}).
\end{equation*}

A \textit{path} is a sequence of configurations 
$(\sigma_0, \sigma_1, \ldots, \sigma_{\ell})$, for some $\ell \in \mathbb{N}$, such that
for all $k = 0,1, \ldots, \ell-1$ it holds
$\sum_{i \in V} |\sigma_{k+1}(i) -~\sigma_k(i)| \leq 1$,
namely $\sigma_{k+1}$ is obtained by flipping one spin 
in $\sigma_k$.
Given a configuration $\sigma \in S$, a configuration $\sigma' \in S$ is called a \textit{downhill configuration} of $\sigma$ if there exists a path $\omega = (\sigma_0, \sigma_1, \ldots, \sigma_{\ell})$ for some $\ell \in \mathbb{N}$ such that $\sigma_0 = \sigma$, $\sigma_{\ell} = \sigma'$, and $H(\sigma_{k+1}) \leq H(\sigma_k)$ for all $k = 1, \ldots, \ell-1$. Note that such a sequence satisfies $p(\omega) > 0$. We denote by $\Gamma(\sigma) \subseteq S$ the set of all downhill configurations of a configuration $\sigma \in S$. Furthermore, we write $\Omega(\sigma, \eta)$ for the set of all downhill paths leading from a configuration $\sigma \in S$ to a configuration $\eta \in S$. 

A site $i \in \Lambda$ is called \textit{susceptible} in a configuration $\sigma \in S$ if $H(\sigma^i) \leq H(\sigma)$, i.e., if $p(\sigma, \sigma^i) > 0$. Since we assumed the external magnetic field to be small and positive, that is, $h \in (0,1)$, a site with spin $+1$ is susceptible if and only if at least three of its neighbors have spin $-1$ and a site with spin $-1$ is susceptible if and only if at least two of its neighbors have spin $+1$. Given a configuration $\sigma \in S$, we let the set of susceptible sites in $\sigma$ be denoted by $\Delta(\sigma)$. 

%An example of a configuration with a susceptible site is provided in Figure \ref{susceptible_spin} (de Jongh et al., 2025). 
%\begin{figure}
%\centering
%\includegraphics[width=0.2\linewidth]{Susceptible.pdf}
%\caption{Configuration depicted on the dual lattice. Each white square corresponds to a site with spin $-1$, each colored square to a site with spin $+1$. The site with spin $+1$ that is represented by the dark orange square is susceptible, since it is surrounded by three neighboring sites with spin $-1$. (First appeared in de Jongh et al.\ 2025).}
%\label{susceptible_spin}
%\end{figure} 

A configuration $\sigma \in S$ is called \textit{fragile} if it has at least one susceptible site, i.e., if $\Delta(\sigma) \neq \emptyset$. If a configuration has no susceptible sites, we call it a \textit{robust} configuration or a \textit{local minimum} of the 
Hamiltonian. Figure \ref{fragile_vs_robust} shows some examples of fragile and robust configurations.

\begin{figure}[t]
\centering
\includegraphics[width=0.7\linewidth]{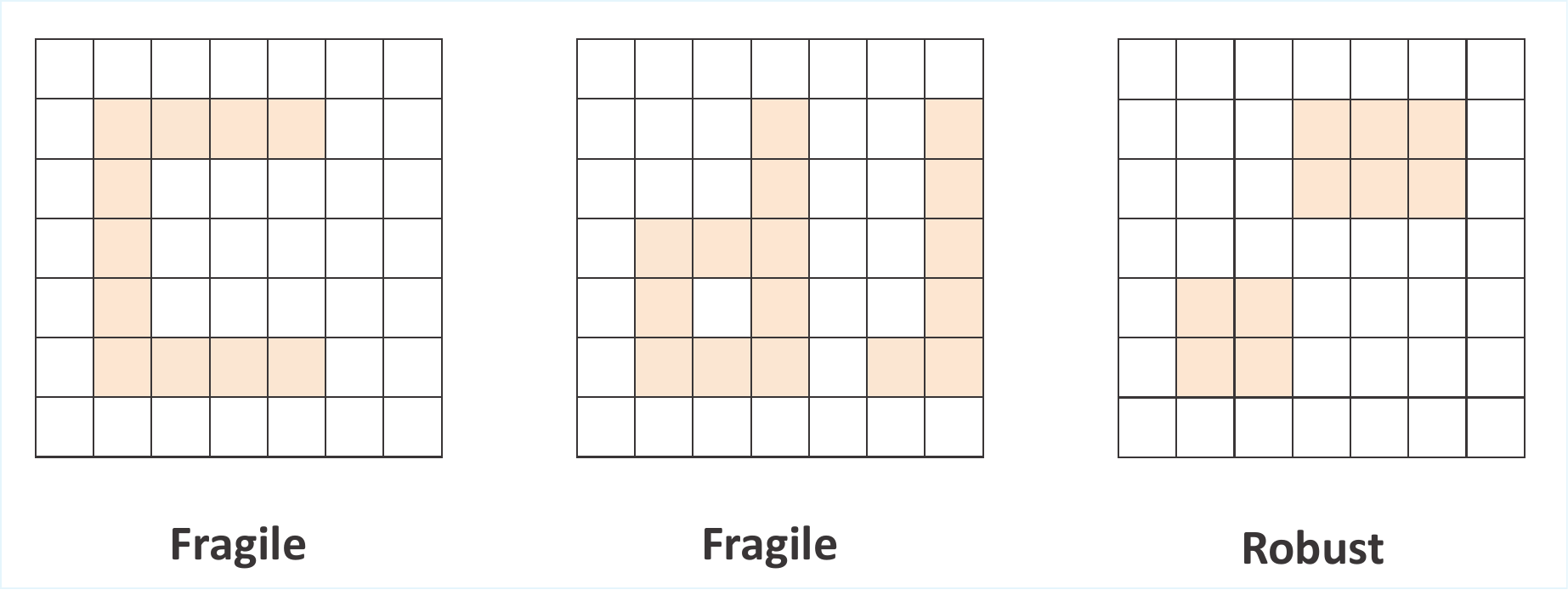}
\caption{Illustration of fragile versus robust configurations.}
\label{fragile_vs_robust}
\end{figure} 

We define $U \subseteq S$ as the set of all robust configurations. Given a configuration $\sigma \in S$, let $U(\sigma) \subset U$ denote the set of configurations in $U$ that are downhill configurations of $\sigma \in S$. Let $U^{(k)} \subset U$, $k \in 0, 1, \ldots$ denote the set of robust configurations in which the sites with plus spin form $k$ maximal connected components.

The set $U^{(1)}$ is made of configurations in which the sole plus component has the shape of a rectangle such that its longest side has a number of sites (side length) belonging to 
$\{2, 3, \ldots, N-3, N-2\}\cup\{ N\}$
and the smallest side lengths 
is arbitrary, if the longest is equal to $N$, and greater than or equal to $2$, otherwise.
For the proof of this statement see, e.g.,
\cite{neves1991critical,deJongh2025}.
On the other hand, see, e.g., 
\cite{olivieri2005large,neves1991critical,cirillo2022review},
a configuration  is in  $U^{(k)}$, with $k\ge2$, if and only if 
the sites with spin $+1$ form $k$ maximal connected components, each of which has the shape of a rectangle characterized as above and such that any two components 
$W_1, W_2$ satisfy $\delta(W_1, W_2) > 2$.
In the sequel, we shall often refer to the plus components
as \textit{droplets}, and we shall call them \emph{stripes}
when one of the two side lengths is equal to $N$.

\subsection{The Markov decision process}
\label{The Markov decision process}

An MDP is described by a tuple $(S, A, P, r)$, consisting of the following elements \cite{Puterman}:
\begin{itemize}
    \item[i)] a state space $S$ containing all possible states that the system can be in. We assume that $S$ is finite.
    \item[ii)] An action space $A$ containing the possible actions that the decision maker can select. We assume that the action space is finite. 
    \item[iii)] A \textit{transition probability kernel} $P: S \times A \times S$ describing the dynamics of the MDP. We denote by $P(s'|s, a)$ the probability that the system transitions to state $s' \in S$, given that its current state is $s \in S$ and action $a \in A$ was chosen.
    \item[iv)] A \textit{reward function} $r: S \rightarrow \mathbb{R}$ specifying the immediate reward $r(s)$ collected in state $s \in S$. We assume the reward function to be bounded. In the general MDP setup, the reward function may depend on both the state and the action selected in this state.    
\end{itemize}

Let $T = \{0, 1, 2, \ldots\}$ denote the set of points in time at which the decision maker can take an action, called the \textit{decision epochs}. The decision maker's behavior is described by a \textit{policy}. We restrict ourselves to policies that are \textit{stationary} and \textit{deterministic}. This type of policy repeatedly applies the same \textit{deterministic decision rule} $d: S \rightarrow A$ at each decision epoch, which prescribes an action $d(s) \in A$ for each state $s \in S$.
%This decision rule can be either \textit{deterministic} or \textit{randomized}. A \textit{deterministic decision rule} $d: S \rightarrow A$ prescribes an action $d(s) \in A$ for each state $s \in S$. A \textit{randomized decision rule} specifies a probability distribution $q_{d}(\cdot)$ on the action space $A$ for each $s \in S$. The stationary policy obtained from repeatedly applying a decision rule $d$ is written as $\pi = d^{\infty}$. 

We denote by $\Pi$ the set of stationary policies
and 
by $\mathbb{P}_s^\pi$ and $\mathbb{E}_s^\pi$, respectively,
the process probability with policy $\pi$ and the corresponding expectation
when the dynamics is started at the state $s\in S$.

The quality of a policy $\pi \in \Pi$ is measured by means of its expected total discounted reward, or the value function $v^{\pi}_{\lambda}: S \rightarrow \mathbb{R}$, given by
\begin{equation}\label{value_function}
    v^{\pi}_{\lambda}(s) = \mathbb{E}_{s}^{\pi}\left[\sum\limits_{t=0}^{\infty} \lambda^t r(X^\pi_t)\right], 
\end{equation}
where $X^\pi_t$  denotes
the state at decision 
epoch $t$ and $\lambda \in (0,1)$ 
is called \textit{discount factor},
and $s\in S$ is the initial state.
The value function is the unique solution, see, e.g.,
\cite[p. 151, Thm. 6.2.5]{Puterman}, 
of the 
following equations:
\begin{equation}
\label{eq:fp}
    v^{\pi}_{\lambda}(s) = r(s) + \lambda \sum\limits_{s' \in S} P(s'|s, d(s))v^{\pi}_{\lambda}(s'), \quad s \in S,
\end{equation}
which follow by conditioning on the state at the next decision epoch. 

From a stochastic--process viewpoint, \eqref{eq:fp} admits a natural interpretation in terms of renewal equations and resolvent potentials. Iterating \eqref{eq:fp} yields the classical resolvent representation
\begin{equation}\label{eq:resolvent}
v_\lambda
=
\sum_{t \ge 0} \lambda^t P^t r,
\end{equation}
which expresses the value function as a discrete convolution of the reward function with the powers of the transition kernel. This is the discrete analogue of the renewal equation 
\[
u = f + k * u
\]
and its solution via the renewal (resolvent) series, see~\cite{FellerVol2,Gripenberg}. Equation~\eqref{eq:fp} may therefore be regarded as a \emph{discrete Volterra equation of the second kind}, with $P_\pi$ acting as the renewal kernel and the discount factor~$\lambda$ playing the role of an exponential kernel in continuous-time formulations. This viewpoint is consistent with the potential theory of Markov chains: the operator
\[
R_\lambda = \sum_{t\ge0} \lambda^t P^t
\]
is the $\lambda$--resolvent of the kernel $P$, and \eqref{eq:resolvent} states simply that $v_\lambda = R_\lambda r$; see~\cite{KemenySnell,RevuzMarkovChains}.
An analogous identity holds in continuous time. If $L_\pi$ denotes the infinitesimal generator of a controlled Markov process, then the Laplace-transformed Kolmogorov backward equation reads
\[
(\rho I - L_\pi) u_\rho = r,
\qquad
u_\rho = \int_{0}^{\infty} e^{-\rho t} e^{t L_\pi} r \, dt.
\]
Thus, the correspondences $e^{-\rho t} \longleftrightarrow \lambda^t$, and $e^{t L_\pi} \longleftrightarrow P_\pi^t$,
identifies \eqref{eq:fp}--\eqref{eq:resolvent} as the discrete-time counterpart of the resolvent equation in continuous-time stochastic control; see ~\cite{Puterman, HLL}. In this sense, the value function $v_\lambda$ is  interpreted as the \emph{discounted potential} of the controlled Markov chain, solving a renewal/Volterra equation and coinciding with the resolvent of its transition dynamics.

%We let the function $\tilde{v}^{\pi}_{\lambda}: S \times A \rightarrow \mathbb{R}$, $\pi \in \Pi$, $\lambda \in (0,1)$ be defined as
%\begin{equation*}
%    \tilde{v}^{\pi}_{\lambda}(s, a) = r(s, a) + \lambda \sum\limits_{s' \in S} P(s'|s, a)v^{\pi}_{\lambda}(s'), \quad s \in S, \quad a \in A.
%\end{equation*}
%Hence, $\tilde{v}^{\pi}_{\lambda}(s, a)$, $s \in S$, $a \in A$, gives the expected total discounted reward that results from taking action $a$ at the first decision epoch and following policy $\pi \in \Pi$ afterwards. 

A stationary deterministic policy $\pi^*$ is \textit{optimal} if it satisfies
\begin{equation*}
    v^{\pi^*}_{\lambda}(s) \geq v^{\pi}_{\lambda}(s), \quad s \in S,
\end{equation*}
for each $\pi \in \Pi$. %The optimal value function of the MDP is given by 
%\begin{equation*}
%    v^*_{\lambda}(s) = \sup_{\pi \in \Pi} v^{\pi}_{\lambda}(s), \quad s \in S.
%\end{equation*}
We write the value function of such an optimal policy $\pi^*$ simply as $v^*_{\lambda}(s)$, $s \in S$.
It is possible to prove, see, 
\cite{Puterman}, 
that there exists an optimal stationary deterministic policy under the assumptions we made on the configuration space, the action space and the reward function. Furthermore, a policy $\pi^* \in \Pi$ is optimal if and only if its value function $v^{\pi^*}_\lambda$ is a solution to the \textit{optimality equations} or \textit{Bellman equations} \cite[p.~152]{Puterman}:
\begin{equation}\label{Bellman_equations}
    v_{\lambda}(s) = \sup_{a \in A}\{r(s) + \lambda \sum\limits_{s' \in S} p(s'|s, a) v_{\lambda}(s')\}, \quad s \in S.
\end{equation}
Observe that the listed assumptions on the configuration space, the action space and the reward function ensure the attainment of the supremum. For this reason, we will write a maximum instead in the remainder of the paper.

The Bellman optimality equations admit a precise interpretation as the
discrete--time counterpart of the Hamilton--Jacobi--Bellman (HJB)
equation in continuous--time stochastic control.  Consider a controlled
diffusion $(X_t)_{t\ge 0}$ on a domain $E\subset \mathbb{R}^d$ with
dynamics
\[
dX_t = b(X_t,a_t)\,dt + \sigma(X_t,a_t)\,dW_t,
\]
and let $\rho>0$ be the discount rate.  The value function of the
continuous--time control problem,
\[
v(x)=\sup_{a_\cdot}\,
\mathbb{E}_x\!\left[\int_0^\infty e^{-\rho t}\, r(X_t,a_t)\,dt\right],
\]
is known to satisfy the stationary HJB equation (see, e.g.,
\cite{FlemingSoner2006, OksendalSulem2007})
\begin{equation}\label{eq:HJB}
\rho\,v(x)
=
\sup_{a\in A(x)}
\Bigl\{
r(x,a) + (\mathcal{L}^a v)(x)
\Bigr\},
\end{equation}
where $\mathcal{L}^a$ is the infinitesimal generator of the diffusion,
\[
(\mathcal{L}^a v)(x)
=
b(x,a)\cdot \nabla v(x)
+\tfrac12 \mathrm{Tr}\!\left(\sigma(x,a)\sigma(x,a)^{\!\top} D^2 v(x)\right).
\]
The optimal feedback control selects an action $a^*(x)$ that maximizes
the expression on the right--hand side of~\eqref{eq:HJB}.

\medskip

In discrete time, for a controlled Markov chain with transition kernel
$P(\cdot\,|s,a)$ and discount factor $\lambda\in(0,1)$, the optimal value
function satisfies the Bellman equation (see
\cite{Puterman,BertsekasShreve1978})
\begin{equation}\label{eq:Bellman}
v^*(s)
=
\max_{a\in A(s)}
\left[
r(s,a)
+ \lambda \sum_{s'} P(s'|s,a)\,v^*(s')
\right].
\end{equation}
The connection with the HJB equation becomes explicit when the discrete
model is interpreted as a time discretization of the continuous one with
time step $\Delta t$.  In this case one has the classical approximations
\[
\lambda = e^{-\rho \Delta t}
= 1 - \rho\,\Delta t + O((\Delta t)^2),
\qquad
P(s'|s,a)
= \delta_{s'}(s)
+ \Delta t\,\mathcal{L}^a_{s\to s'} + o(\Delta t),
\]
where $\mathcal{L}^a$ appears as the first--order term in the expansion
of the discrete transition operator.  Substituting these relations into
\eqref{eq:Bellman} and letting $\Delta t\to 0$ yields the continuous--time
HJB equation~\eqref{eq:HJB}.  Thus, the Bellman equation is precisely a
backward--Euler time discretization of the HJB equation, with the
transition probabilities $(P(s'|s,a))_{s'}$ playing the role of the
exponential of the generator and $\lambda$ corresponding to
$e^{-\rho\Delta t}$.

\subsection{Remarks on the value function}
\label{s:valuemeaning}
In this section, we examine the meaning of the value
function to clarify the nature of the quantity being
optimized when an appropriate policy is selected for
the MDP. Clearly, the physical interpretation of the
value function depends crucially on the chosen reward
function. We shall discuss a case particularly relevant
to our application.

In several applications, e.g., in the Ising model case that we will discuss in the sequel, 
the MDP is introduced with the goal to optimize the path to some specific target state $\bar{s}\in S$.
Is these cases a typical choice for the the reward function is 
\begin{equation}\label{target_reward}
r(s) = 
    1 \text{ if } s = \bar{s}
    \;\textup{ and }\;
    r(s)=0 \text{ otherwise.}
\end{equation}
Given states $s, s' \in S$ and policy $\pi \in \Pi$, let $\tau^{s, \pi}_{s'}$ denote the first hitting time from state $s$ to state $s'$ under policy $\pi$, i.e.,
\begin{equation*}
    \tau^{s,\pi}(s') 
    = 
    \inf\{t \in \mathbb{N}:\,X^{\pi}_t = s'\},
\end{equation*}
where, we recall, $X^\pi_t$  denotes
the state at decision 
epoch $t$ 
and $s\in S$ is the initial state.

In \cite{deJongh2025}, the following relation between the value function and the first hitting time to the target state was established:
%\begin{equation*}
%v^{\pi}_{\lambda}(s) = \begin{cases}
%1 + \dfrac{\mathbb{E}[\lambda^{\tau^{\bar{s}, \pi}(\bar{s})}]}{1- \mathbb{E}[\lambda^{\tau^{\bar{s}, \pi}(\bar{s})}]}, 
%&
%\text{ if } s = \bar{s}, \\
%\dfrac{\mathbb{E}[\lambda^{\tau^{s, \pi}(\bar{s})}]}{1- \mathbb{E}[\lambda^{\tau^{s, \pi}(\bar{s})}]}, &\text{otherwise}.
%\end{cases}
%\end{equation*}
\begin{equation*}
v^{\pi}_{\lambda}(s) = 
\dfrac{\mathbb{E}_s^\pi[\lambda^{\tau^{s, \pi}(\bar{s})}]}{1- \mathbb{E}_s^\pi[\lambda^{\tau^{s, \pi}(\bar{s})}]}, 
\end{equation*}
for all $s\in S\setminus\{\bar{s}\}$. Moreover, 
if $\bar{s}$ is an absorbing state, this expression reduces to
\begin{equation}\label{exp_1}
    v^{\pi}_{\lambda}(s) = \dfrac{\mathbb{E}_s^\pi[\lambda^{{\tau}^{s, \pi}(\bar{s})}]}{1-\lambda}, \quad s \in S \setminus \{\bar{s}\},
\end{equation} 
which implies that the function
$G^{s,\pi}(\lambda)=(1-\lambda)v^\pi_\lambda(s)$ is the probability generating function of the first hitting time to $\bar{s}$.
%\begin{equation}
%    (1-\lambda)v^{\pi}_{\lambda}(s) = G_{\tau^{s, %\pi}_{s^*}}(\lambda), \quad s \in S.
%\end{equation}

In the following theorem, we identify the meaning of the
value function for $\lambda$ close to $1$. In fact, this
result provides a clear physical interpretation of the
value function in that regime: it shows that, in such a
case, the optimal policy, namely, the one maximizing the
value function, minimizes the first hitting time to
$\bar{s}$.

\begin{thm}
\label{t:value010}
Given a policy $\pi$ and $s,\bar{s}\in S$ 
such that $s\neq\bar{s}$ and $\bar{s}$ is an absorbing state. 
Then, 
\begin{equation}
\label{e:value000}
\lim_{\lambda\uparrow1}
\Big[
 \frac{1}
      {1-\lambda}  
      -v^\pi_\lambda(s)
\Big]
=
\mathbb{E}_s^\pi[{\tau}^{s, \pi}(\bar{s})].
\end{equation}
   \end{thm}

\begin{proof}
%\cs{\textbf{Cristian}:I guess we should define explicitly $G$ saying that it is the pgf of the first hitting time of $\bar{s}$}\\
Recalling the properties of the probability generating function of a positive discrete random variable, we have that,
\begin{displaymath}
   G^{s,\pi}(1)=1
   \;\textup{ and }\;
   \lim_{\lambda\to 1^-}
   \frac{\textup{d}}{\textup{d}\lambda}
   G^{s,\pi}(\lambda)
   =
   \mathbb{E}_s^\pi[{\tau}^{s, \pi}(\bar{s})].
   \end{displaymath}
Thus, we have
\begin{displaymath}
\frac{G^{s,\pi}(1)-G^{s,\pi}(\lambda)}
      {1-\lambda} 
      =
   \frac{1-G^{s,\pi}(\lambda)}
      {1-\lambda}    
      =
      \frac{1}
      {1-\lambda}  
      -v^\pi_\lambda(s),
    \end{displaymath}
    which implies 
    \eqref{e:value000}.
\end{proof}

A more refined computation, based on an asymptotic expansion,
can relate the value function to the higher moments of the
hitting time. We start from the expression
\begin{equation}
\label{e:value010}
v_\lambda^\pi(s)
= \frac{G^{s,\pi}(\lambda)}{1-\lambda}
\end{equation}
and construct an expansion as the discount factor
$\lambda \uparrow 1$. Since $G^{s,\pi}(1)=1$, the value
function $v_\lambda^\pi(s)$ diverges as $1/(1-\lambda)$.
Because $\lambda=1$ is a simple pole, the function is not
analytic at $\lambda=1$, and no Taylor expansion exists in
the usual sense. Nevertheless, one can develop a
\emph{formal asymptotic expansion} in powers of
$\varepsilon = 1-\lambda$ as $\varepsilon \downarrow 0$.

A function $f(\lambda)$ is said to admit an \emph{asymptotic expansion}
$f(\lambda) \sim \sum_{k=0}^{n} a_k (1-\lambda)^{\alpha_k}$
as $\lambda \uparrow 1$ if
$$
\lim_{\lambda \uparrow 1} 
\frac{f(\lambda) - \sum_{k=0}^{m} a_k (1-\lambda)^{\alpha_k}}
{(1-\lambda)^{\alpha_m}} = 0
\quad \text{for all } m \in\mathbb{N},
$$
where $a_k\in\mathbb{R}$ and $\alpha_k$ is 
an increasing sequence of reals.
The equality \(f(\lambda) \sim g(\lambda)\) indicates that
$f(\lambda)/g(\lambda) \to 1$ as $\lambda \uparrow 1$.

Let $\varepsilon = 1-\lambda$. 
For integer-valued nonnegative $\tau$, one has
$$
(1-\varepsilon)^{\tau}
= 1 - \varepsilon\tau
+ \frac{\varepsilon^2}{2}\tau(\tau-1)
- \frac{\varepsilon^3}{6}\tau^\pi(\tau-1)(\tau-2)
+ O(\varepsilon^4),
$$
and, therefore,
\begin{align*}
\mathbb{E}_s^\pi[(1-\varepsilon)^{{\tau}^{s, \pi}(\bar{s})}]
 = 1
&- \varepsilon\,\mathbb{E}_s^\pi[{\tau}^{s, \pi}(\bar{s})]
+ \frac{\varepsilon^2}{2}\,\mathbb{E}_s^\pi[{\tau}^{s, \pi}(\bar{s})({\tau}^{s, \pi}(\bar{s})-1)]
\\
&
- \frac{\varepsilon^3}{6}\,\mathbb{E}_s^\pi[{\tau}^{s, \pi}(\bar{s})({\tau}^{s, \pi}(\bar{s})-1)({\tau}^{s, \pi}(\bar{s})-2)]
+ O(\varepsilon^4).
\end{align*}
Substituting into the expression for $v_\lambda^\pi(s)$ gives, as
$\varepsilon \downarrow 0$,
\begin{align}
\label{eq:divergence}
v_{1-\varepsilon}^\pi(s)
\sim
\frac{1}{\varepsilon}
&
- \mathbb{E}_s^\pi[{\tau}^{s, \pi}(\bar{s})]
+ \frac{\varepsilon}{2}\,\mathbb{E}_s^\pi[{\tau}^{s, \pi}(\bar{s})({\tau}^{s, \pi}(\bar{s})-1)]
\\
&
- \frac{\varepsilon^2}{6}\,\mathbb{E}_s^\pi[{\tau}^{s, \pi}(\bar{s})({\tau}^{s, \pi}(\bar{s})-1)({\tau}^{s, \pi}(\bar{s})-2)]
+ O(\varepsilon^3).
\end{align}
Rewriting the expression by isolating the divergent part yields
\begin{align*}
\frac{1}{1-\lambda} - v_\lambda^\pi(s)
\sim
\mathbb{E}_s^\pi[{\tau}^{s, \pi}(\bar{s})]
&
- \frac{1-\lambda}{2}\,\mathbb{E}_s^\pi[{\tau}^{s, \pi}(\bar{s})({\tau}^{s, \pi}(\bar{s})-1)]
\\
&
+ \frac{(1-\lambda)^2}{6}\,\mathbb{E}_s^\pi[{\tau}^{s, \pi}(\bar{s})({\tau}^{s, \pi}(\bar{s})-1)({\tau}^{s, \pi}(\bar{s})-2)]
+ O((1-\lambda)^3).
\end{align*}

Taking the limit $\lambda\uparrow1$ we find 
again \eqref{e:value000}. But, we can also provide 
an interpretation in terms of the higher moment 
of the hitting time, for instance, for the 
second moment we get
$$
\lim_{\lambda \uparrow 1} 
\frac{2}{1-\lambda}
\Big(
-
\frac{1}{1-\lambda} 
+ v_\lambda^\pi(s)
+
\mathbb{E}_s^\pi[{\tau}^{s, \pi}(\bar{s})]
\Big)
= 
\mathbb{E}_s^\pi[{\tau}^{s, \pi}(\bar{s})({\tau}^{s, \pi}(\bar{s})-1)]
.
$$

We conclude this section discussing the physical meaning of the
value function for small values of $\lambda$. On a
heuristic basis, we may argue that, when $\lambda$ is small,
only trajectories reaching the all-plus configuration within a
short time significantly contribute to the value function.
Therefore, the optimal policy is expected to be the policy capable of
selecting those trajectories that accomplish a short flight to
the target configuration.

In order to make this heuristic idea precise, we denote by
$t^{s,\pi}(\bar{s})$ the deterministic number providing 
the shortest number of epochs required
for the MDP to reach the target configuration $\bar{s}$.
This notion is known in the literature and is sometimes
referred to as the \emph{minimal
path length} on the graph induced by the Markov chain. 
More precisely, we set
\begin{equation}
    \label{eq:valu050}
t^\pi_{s,\bar{s}}=\min\{{t\in T}:\,\mathbb{P}_s^\pi(X_s^\pi(t)=\bar{s})>0\}.
\end{equation}
The following theorem shows that, for small values of
$\lambda$, the optimal policy is the one with minimal
path length.

\begin{thm}
 \label{t:value050}
Let $s,\bar{s}\in S$ 
such that $s\neq\bar{s}$ and $\bar{s}$ is an absorbing state. 
Consider two policies $\pi$ and $\pi'$ such that 
$t^\pi_{s,\bar{s}}
<
t^{\pi'}_{s,\bar{s}}$.
If 
\begin{equation}
    \label{eq:value050}
\lambda
\le
\Big[
\mathbb{P}_s^\pi(\tau^{s, \pi}(\bar{s}) = t^\pi_{s,\bar{s}})
\Big]^{1/(t^{\pi'}_{s,\bar{s}}-t^\pi_{s,\bar{s}})},
\end{equation}
then $v^{\pi}_{\lambda}(s) \geq v^{\pi'}_{\lambda}(s)$.
\end{thm}
\begin{proof}
To achieve the proof we consider the following simple lower and upper bounds
to the value function:
     using \eqref{exp_1} and \eqref{eq:valu050} we obtain
     \begin{equation}
     \label{small_lambda_exp1}
         v^{\pi}_{\lambda}(s) 
         = 
         \dfrac{1}{1-\lambda}
         \sum\limits_{t = t^\pi_{s,\bar{s}}}^{\infty} 
         \lambda^t
          \mathbb{P}_s^\pi(\tau^{s, \pi}(\bar{s}) = t)
     \geq 
     \dfrac{\lambda^{t^\pi_{s,\bar{s}}}}{1-\lambda} 
     \mathbb{P}_s^\pi(\tau^{s, \pi}(\bar{s}) =t^\pi_{s,\bar{s}} ) 
     \end{equation}
     and
     \begin{equation}\label{small_lambda_exp2}
         v^{\pi'}_{\lambda}(s) 
         = 
         \dfrac{1}{1-\lambda}
         \sum\limits_{t = t^{\pi'}_{s,\bar{s}}}^{\infty} 
         \lambda^t
         \mathbb{P}(\tau^{s, \pi'}(\bar{s}) = t)
         \leq \dfrac{\lambda^{t^{\pi'}_{s,\bar{s}}}}{1-\lambda} .
     \end{equation}
     Combining expressions \eqref{small_lambda_exp1} and \eqref{small_lambda_exp2}, 
    recalling $t^\pi_{s,\bar{s}}<t^{\pi'}_{s,\bar{s}}$,
     we see that if $\lambda$ is so small that
     $
     \lambda^{t^\pi_{s,\bar{s}}}
     \mathbb{P}_s^\pi(\tau^{s, \pi}(\bar{s}) = t^\pi_{s,\bar{s}}) 
     \ge
     \lambda^{t^{\pi'}_{s,\bar{s}}}
     $
     then
      $v^{\pi}_{\lambda}(s) \geq v^{\pi'}_{\lambda}(s)$.
      And this proves the theorem.
     \end{proof}

\subsection{Definition of the Ising MDP}
\label{s:IsingMDP}
In order to control the Ising model at zero temperature, inspired by \cite{deJongh2025},
we formulate a MDP ranging only over the robust configurations, or the local minima of the Hamiltonian. 

Hence, the state space of the MDP is the set $U$. The decision maker has the power to flip any chosen spin from $\Lambda$, after which the system evolves according to the Metropolis dynamics for a certain period of time. Specifically, we let the process evolve until it reaches a next robust configuration. It is the goal of the decision maker to reach the all-plus configuration $\sigma^+$. Letting the action $a = 0$ represent the choice of not flipping any spins, the action space of the Ising MDP is given by $A = A(\sigma)$, where $A(\sigma) = \Lambda \cup \{0\}$ for all $\sigma \in U$. 
Coherently with the notation introduced in 
Section~\ref{s:ising},
we denote
the configuration obtained from $\sigma \in U$ after taking action $a \in A$, or the \textit{post-decision configuration}, by $\sigma^{(a)}$.

We define the transition probability kernel $P: U \times A \times U$ of the MDP as 
\begin{equation*}
    P(\sigma'|\sigma, a) = \sum\limits_{\omega \in \Omega(\sigma^{(a)}, \sigma')} p(\omega), \quad \sigma, \sigma' \in U, \quad a \in A.
\end{equation*} 
The objective of the decision maker, to drive the system towards the all-plus configuration, is captured by a reward function $r: U \rightarrow \mathbb{R}$ defined as
\begin{equation}\label{reward_function_isi}
    r(\sigma) = \begin{cases}
        1, &\text{if } \sigma = \sigma^+, \\
        0, &\text{otherwise,}
    \end{cases} \quad \sigma \in U, \quad a \in A.
\end{equation}
From the fact that the reward function is bounded, it follows that the value function is finite.
Note that \eqref{reward_function_isi} is simply the adaptation of \eqref{target_reward} to the Ising case.

\section{Numerical study of the double seeded Ising MDP}
\label{s:numerical_study}
This section presents a numerical investigation of the structure of the optimal policy in the two-droplet regime. First, we study the control problem for the case in which the two droplets form stripes that are wrapped around the torus. Then, we extend our analysis to the case in which only one of the droplets forms a stripe. Finally, we consider the scenario in which neither of the droplets forms a stripe.

In this numerical investigation, we do not claim to
determine the optimal policy. Rather, we consider a few
promising candidates and compute their associated value
functions numerically, so as to obtain a well-informed
conjecture about optimality. In the case of the two-stripe
initial seeds, our numerical results will be compared with
the rigorous analytical findings that will be established in
the following section.

The choice of the three initial configurations
listed above, the stripe–stripe pair, the
stripe–droplet pair, and the droplet–droplet
pair, is motivated by several considerations.
These settings provide a controlled framework
in which specific mechanisms of growth and
interaction can be isolated and examined with
precision.

The first objective is to model the expansion of
a front, represented by the initial stripe, and
the evolution of a small nucleus, represented by
a droplet of limited size. The mixed case serves
to illustrate how these two distinct initial
conditions interact, thereby offering insight
into intermediate regimes where different growth
dynamics coexist.

A further motivation concerns the double-stripe
configuration. Because the number of microscopic
situations to be analysed is inherently limited,
one can obtain rigorous control of the feeding
policies governing the system. This level of
control is sufficient to identify, in some
instances, an optimal policy. Such a result is of
considerable relevance in our context, as it
helps assess the reliability and interpretive
value of numerical simulations.

\subsection{The two-stripe case}
We study the control problem for the configurations in the set $U^{(2)}$ in which the two droplets with spin $+1$ form stripes that are wrapped around the torus. Let the set of such configurations be denoted by $U^{2, x}$. Analogously, let $U^{1, x}$ denote the set of configurations in which the sites with spin $+1$ form a single stripe that is wrapped around the torus. 

\subsubsection{The auxiliary MDP}
\label{The auxiliary MDP_strip_strip}
In order to find the optimal policy for configurations in the set $U^{2,x}$, we construct an auxiliary MDP denoted by $(S^x, A^x, P^x, r^x)$. Let the state space \(S^x\) be defined as
\begin{equation*}
    S^x = \{(i,j) \mid i,j \in \{0,2,3,\ldots,N\}\}.
\end{equation*}
Each element \((i,j)\in S^x\) should be interpreted as an \emph{equivalence class} of lattice
configurations in which the two stripes of \(+1\)-spins are separated by horizontal gaps of
lengths \(i\) and \(j\), respectively.  In other words, we consider on the set of stripe--stripe
configurations an equivalence relation identifying all configurations that share the same pair
of separating distances, irrespective of their absolute position on the torus.  The set
$S^x$ may therefore be regarded as a quotient space obtained by collapsing each such
equivalence class to a single representative;  Figure~\ref{Aux_state_space_two_strips} (left panel) illustrates this
representation. 

Here, a state $(i,j)$ in which either $i = 0$ or $j = 0$ corresponds to a configuration with a single stripe of spins in state $+1$ and the state $(0, 0)$ corresponds to the all-plus configuration. 

\begin{figure}
\centering
\includegraphics[width=0.4\linewidth]
{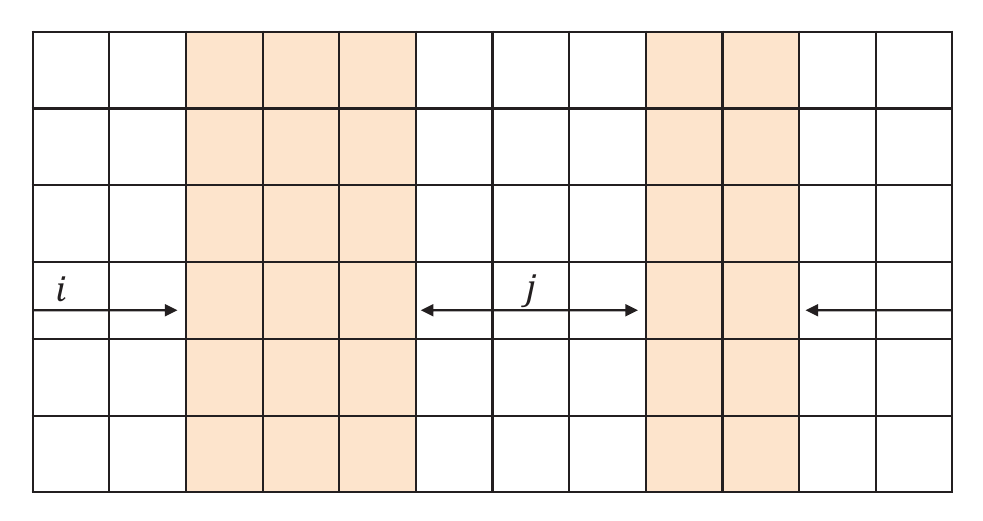}
\raisebox{-0.5mm}{
\includegraphics[width=0.435\linewidth]
{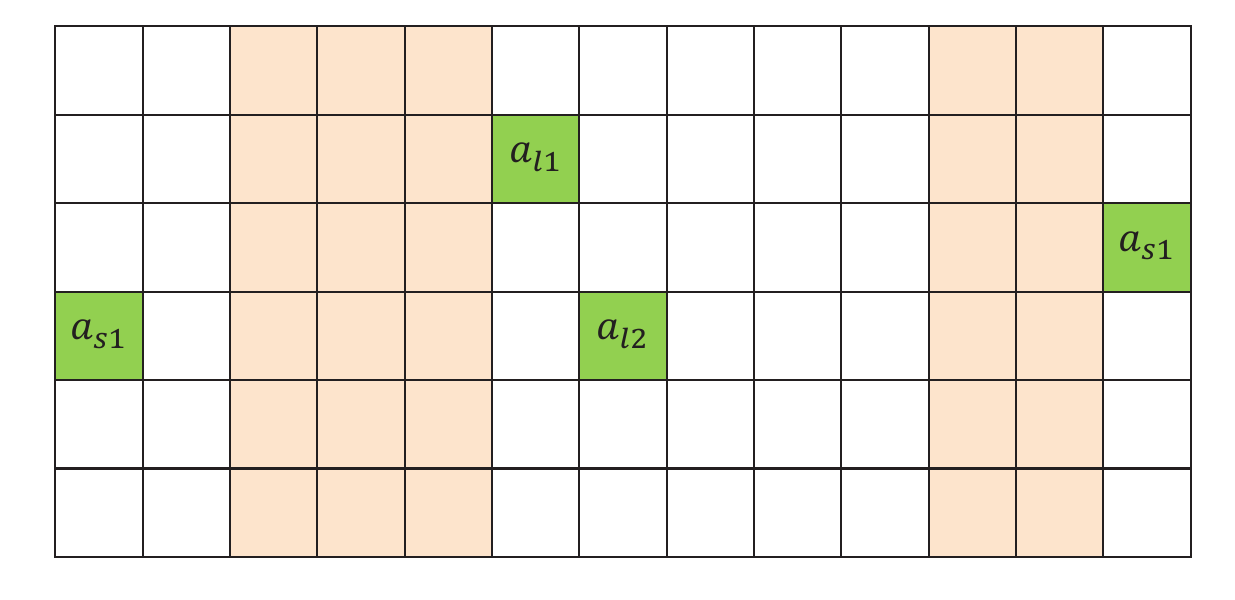}
}
\caption{Illustration of the state space $S^x$ (left) and of the action space $A^x$ (right).}
\label{Aux_state_space_two_strips}
\end{figure} 

%We define the action space $A^x(i,j)$ corresponding to state $(i,j) \in S^x$ as
%\begin{equation*}
%    A^x(i,j) = \begin{cases}
%        \{0\}, &\text{if } i = j = 0, \text{ or } i = N, \text{ or } j = N, \\
%        \{a_{\ell 1}, 0\}, &\text{if } i = 0, j = 2, \text{ or } i = 2, j = 0, \text{ or } i = j = 2\\
%        \{a_{\ell 1}, a_{\ell 2}, 0\}, &\text{if } i = 0,\quad 2 < j < N, \text{ or } 2 < i < N,\quad j = 0, \text{ or } 2 < i, j < N, \quad i = j,\\
%        \{a_{\ell 1}, a_{\ell 2}, a_{s1}, 0\}, &\text{if } 3 < i < N,\quad j = 2, \text{ or } i = 2, \quad 3 < j < N, \\
%        \{a_{\ell 1}, a_{\ell 2}, a_{s1}, a_{s2}, 0\}, &\text{if } 3 < i < N, \quad 3 < j < N, \quad i \neq j.
%    \end{cases}
%\end{equation*}
We define the action space as $A^x = \{a_{\ell 1}, a_{\ell 2}, a_{s1}, a_{s2}, 0\}$.
Here, $a_{\ell 1}$ represents the set of sites at distance 1 from either of the stripes at the side of the longest gap and $a_{\ell 2}$ represents the set of sites at distance 2 from either of the stripes at the side of the longest gap. Analogously, $a_{s1}$ and $a_{s2}$ represent the set of sites at distance 1 and distance 2 from either of the stripes at the side of the shortest gap. The action space $A^x$ is visualized in Figure \ref{Aux_state_space_two_strips}. By taking an action $a \in A^x$, we mean flipping the spin at one of the sites in the corresponding set. In our experiments, we choose this site uniformly at random. The transition probability kernel $P^x$ and the reward function $r^x$ follow in a natural way from their counterparts in the original MDP. Their formal definitions are given in Section \ref{s:rigorous}, where we provide a rigorous treatment of the two-stripe case.

%\enmc{[For Maike: here there is something that must be clarified. For instance: $a_{\ell 1}$ must be 
%a specific site, since the policy is deterministic. So we should not say that it is the set of sites at distance one, but we have to identify which site. Is it correct saying that it is one of the site at distance one prescribe dwith some deterministic rule (e.g., one of those closer to the mid point along the vertical direction). Another thing that should be clarified is the 
%following: when the action set $A_{1,2}(i,j)$ has more than one elemnt, how is made the choice? I think you explained me 
%in Rome, but I do not remember.]}

%\begin{figure}
%\centering
%\includegraphics[width=0.4\linewidth]{Aux_action_space_two_strips.pdf}
%\caption{Illustration of the action space $A^x$.}
%\label{Aux_action_space_two_strips}
%\end{figure} 

\subsubsection{Candidates for optimality}
We compare the performance of two distinct policies $\pi_1 = (d_1)^{\infty}$ and $\pi_2 = (d_2)^{\infty}$ defined in the auxiliary MDP, where $d_1(i,j) \in A_1(i,j)$ and $d_2(i,j) \in A_2(i,j)$. Here, letting $P(A^x)$ denote the power set of the action space $A^x$, the functions $A_k: S^x \rightarrow P(A^x)$, $k = 1,2$, specify a set of actions for each state $s \in S^x$. Note that the functions $A_1$ and $A_2$ define two families of policies: a policy in the family that corresponds to function $A_k$, $k = 1,2$, prescribes an action $a \in A_k(i,j)$ for each $(i,j) \in S^x$. 

In words, a policy from the first class is constructed by flipping, at
the decision epochs, the minus spins located at sites at
distance one from the growing cluster. Conversely, in
a policy from the second class, spins located at distance two are
flipped.

More precisely, the functions $A_k: S^x \rightarrow P(A^x)$, $k = 1,2$, for states $(i,j)$, $i \geq j$, are defined as follows and 
visualized in 
Figs.~\ref{opt_pol_two_strips_fig-ssc}. The cases in which the two policies share the same actions are
\begin{equation}
\label{eq:action-ssc}
\begin{array}{llll}
    A_k(0, 0) = \{0\}, 
    &
    A_k(2, 2) = \{a_{\ell 1}, a_{s1}\}, 
    &
    A_k(3,2) = \{a_{\ell 2}, a_{s1}\}, 
    &
    A_k(4,2) = \{a_{\ell 1}, a_{s1}\}, 
    \\
    A_k(3, 3) = \{a_{\ell 2}, a_{s2}\},
    &
    A_k(4, 3) = \{a_{\ell 1},a_{s2}\}, 
    &
    A_k(4, 4) = \{a_{\ell 1}, a_{s1}\},
    & 
    \end{array}
\end{equation}
    for $k=1,2$. On the contrary, in the following cases the two policies have different actions
    \begin{equation}
\label{eq:action-ssnc}
        \begin{array}{llll}
    A_1(i,j)=
        \{a_{\ell 1}, a_{s1}\} 
        &
        \text{ and }
        &
        A_2(i,j)=
        \{a_{\ell 2}, a_{s1}\}, 
        &
        \text{ for }
        i \geq 5, \quad j = 2,4,
        \\
         A_1(i,3)=
        \{a_{\ell 1}, a_{s2}\} 
        &
        \text{ and }
        &
        A_2(i,3)=
        \{a_{\ell 2}, a_{s2}\}, 
        &
        \text{ for }
        i \geq 5, 
        \\
        A_1(i,j)=
        \{a_{\ell 1},a_{s1}\} 
        &
        \text{ and }
        &
        A_2(i,j)=
        \{a_{\ell 2},a_{s2}\}, 
        &
        \text{ for }
        i,j \geq 5.
    \end{array}
    \end{equation}
In Section \ref{s:rigorous}, we rigorously prove that a policy defined by the function $A_1$ is optimal for $\lambda \in~[\lambda_c, 1)$, whereas a policy specified by the function $A_2$ is optimal for $\lambda \in (0, \lambda_c]$, where $\lambda_c = 15/17$. In simulating the two types of policies, we use randomized versions of the decision rules $d_1$ and $d_2$. That is, in each state $(i,j)$, an action is chosen uniformly at random from the set $A_k(i,j)$, for $k = 1,2$. 

\begin{figure}
\centering
\includegraphics[width=0.2\linewidth]{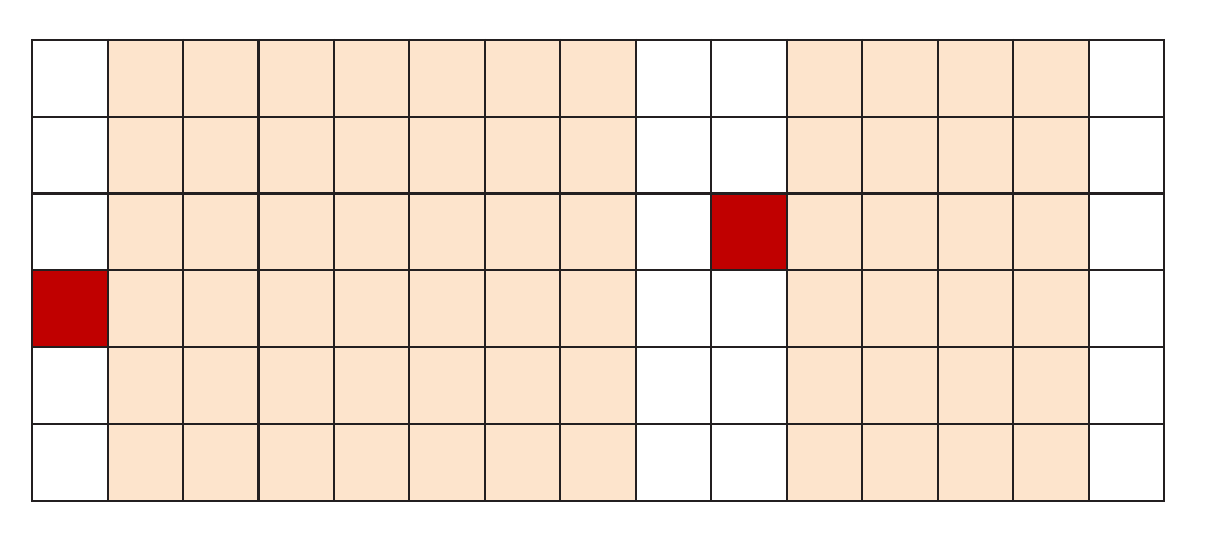}
 \includegraphics[width=0.2\linewidth]{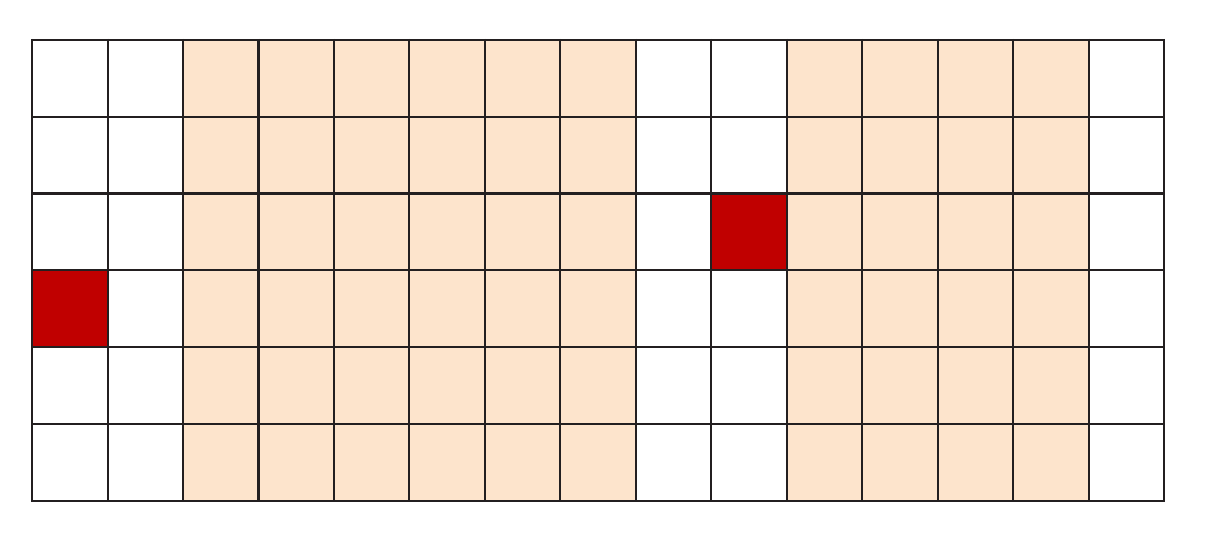}
   \includegraphics[width=0.2\linewidth]{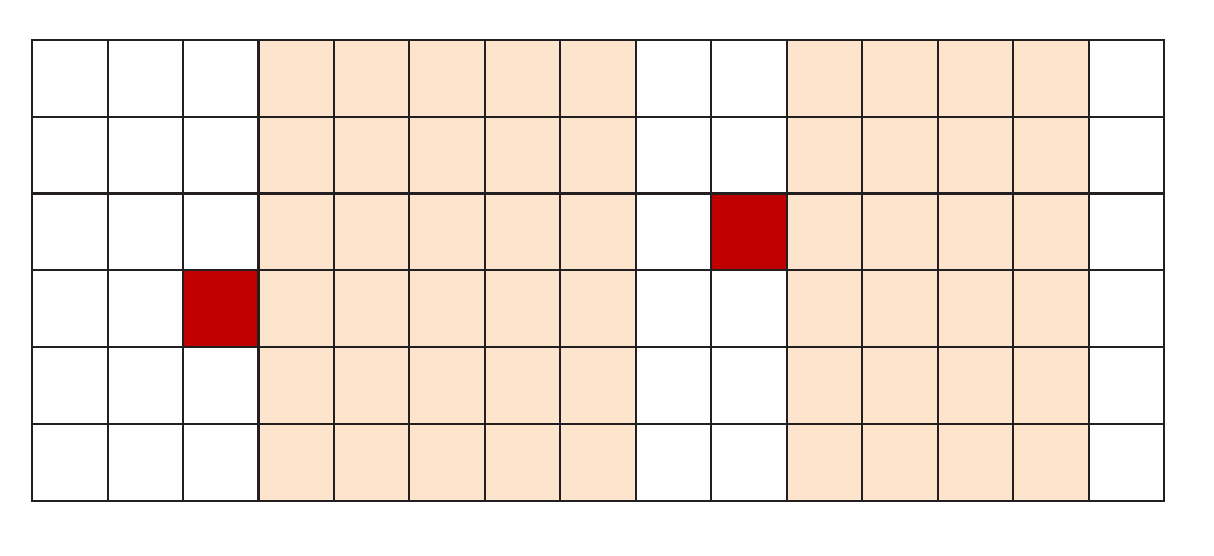}
   \\
   \includegraphics[width=0.2\linewidth]{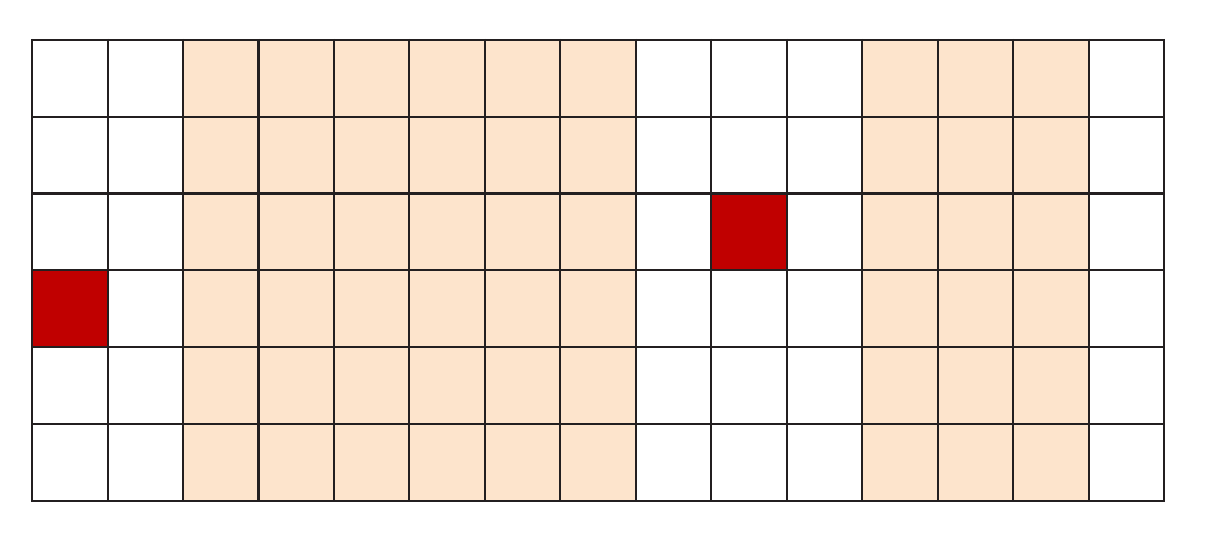}
    \includegraphics[width=0.2\linewidth]{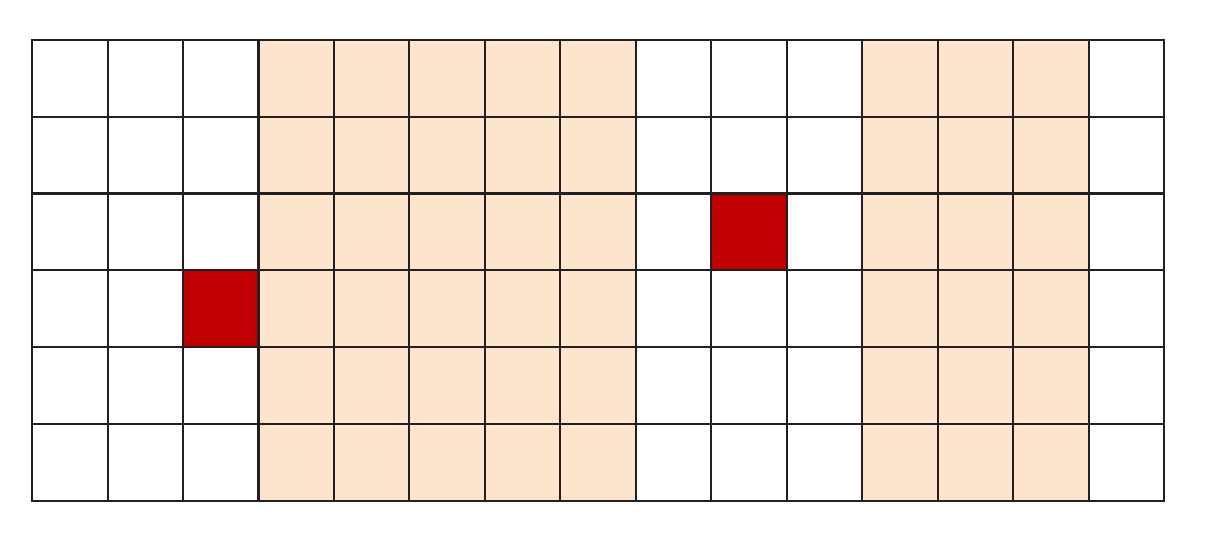}
     \includegraphics[width=0.2\linewidth]{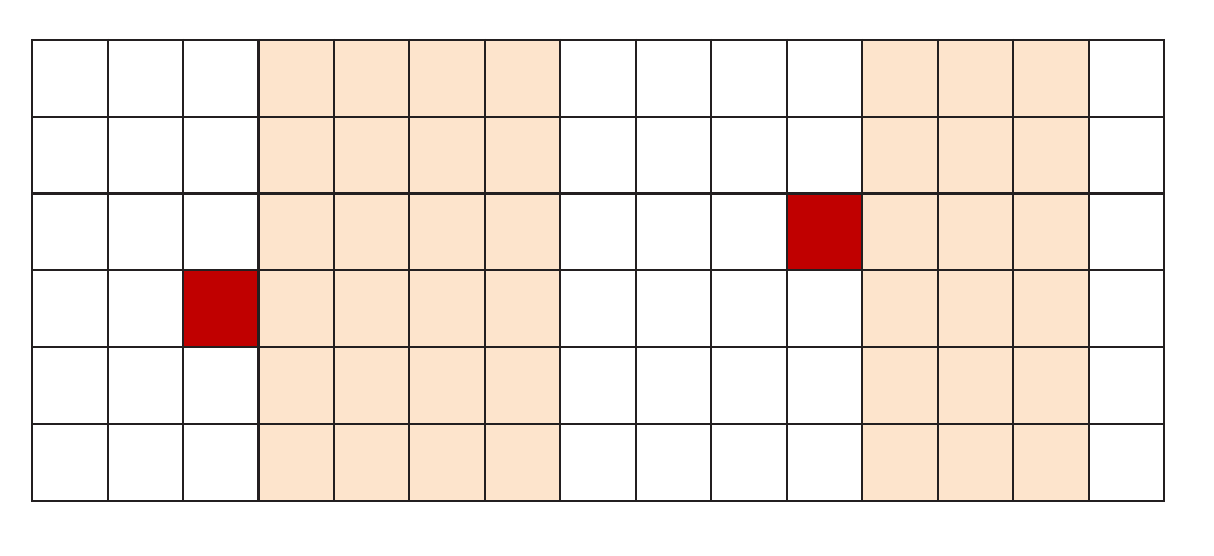}
\caption
    {Visualization of the action sets defined in 
    \eqref{eq:action-ssc}.Top: $(i,j)=(2,2),(3,2),(4,2)$.
    Bottom: $(i,j)=(3,3),(3,4),(4,4)$}
   \label{opt_pol_two_strips_fig-ssc}
\end{figure}

\begin{figure}
\centering
\includegraphics[width=0.2\linewidth]{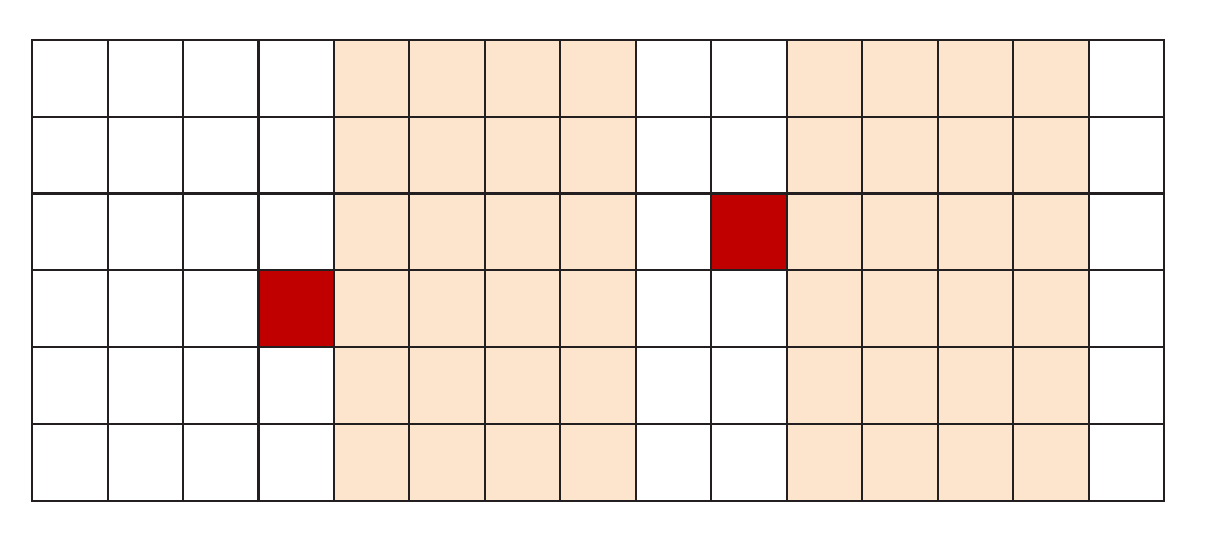}
 \includegraphics[width=0.2\linewidth]{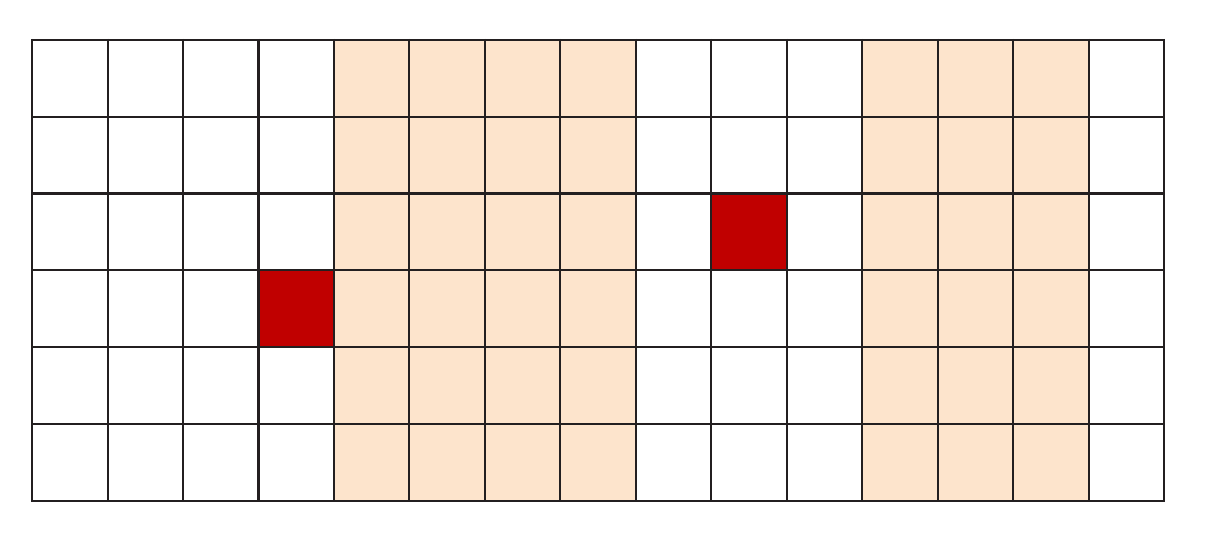}
   \includegraphics[width=0.2\linewidth]{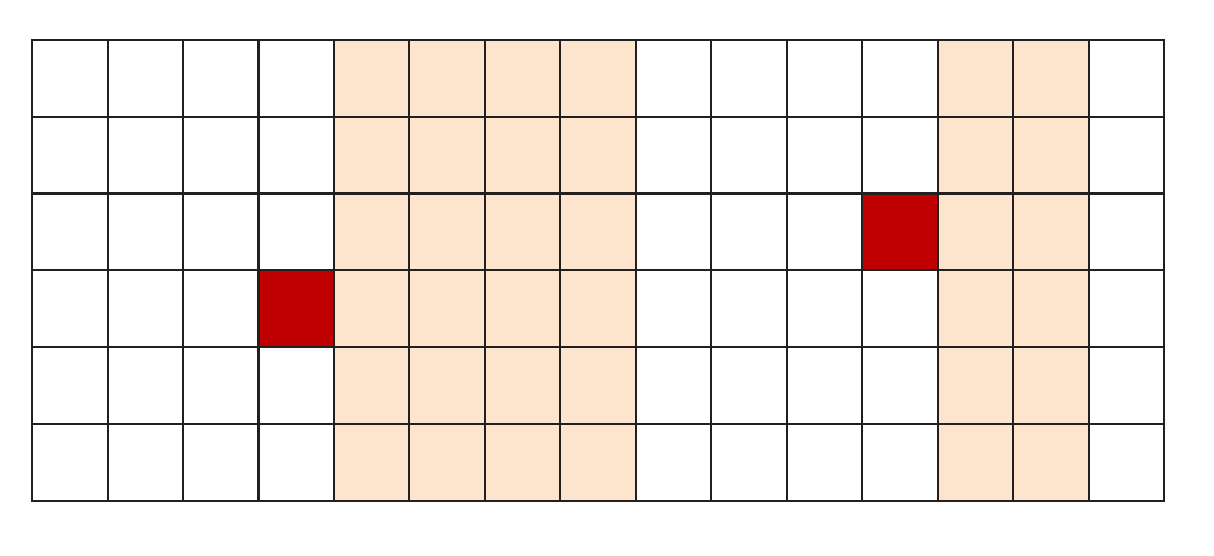}
    \includegraphics[width=0.2\linewidth]{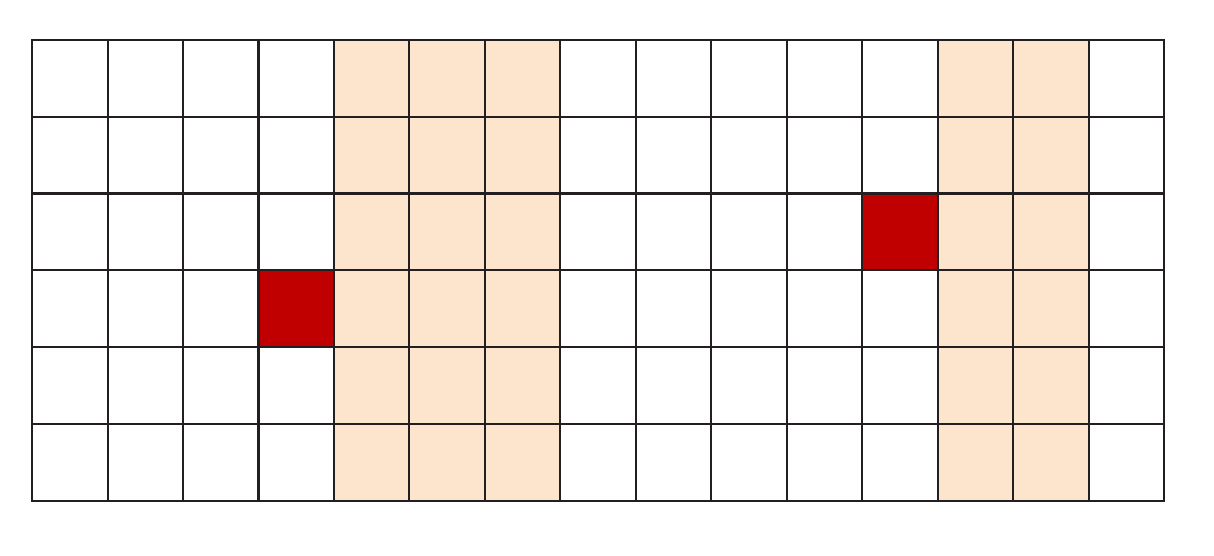}
   \\
   \includegraphics[width=0.2\linewidth]{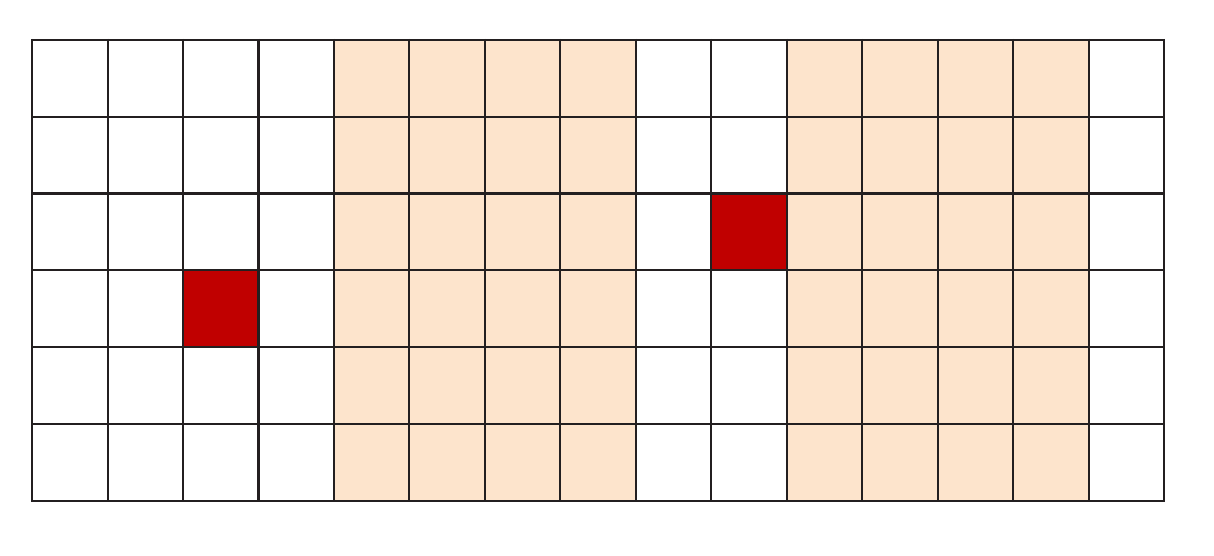}
 \includegraphics[width=0.2\linewidth]{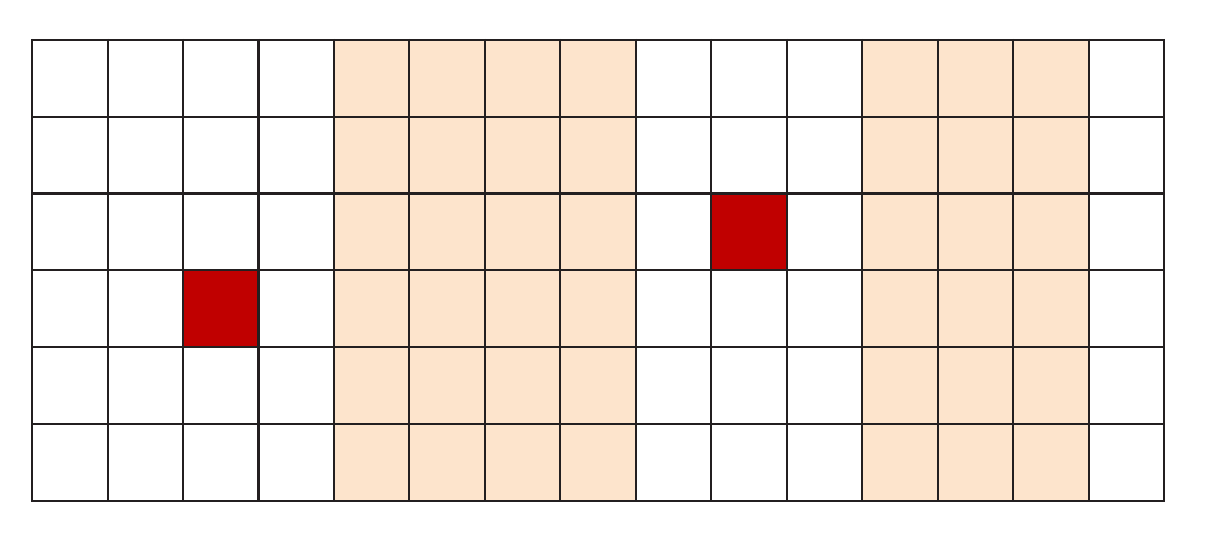}
   \includegraphics[width=0.2\linewidth]{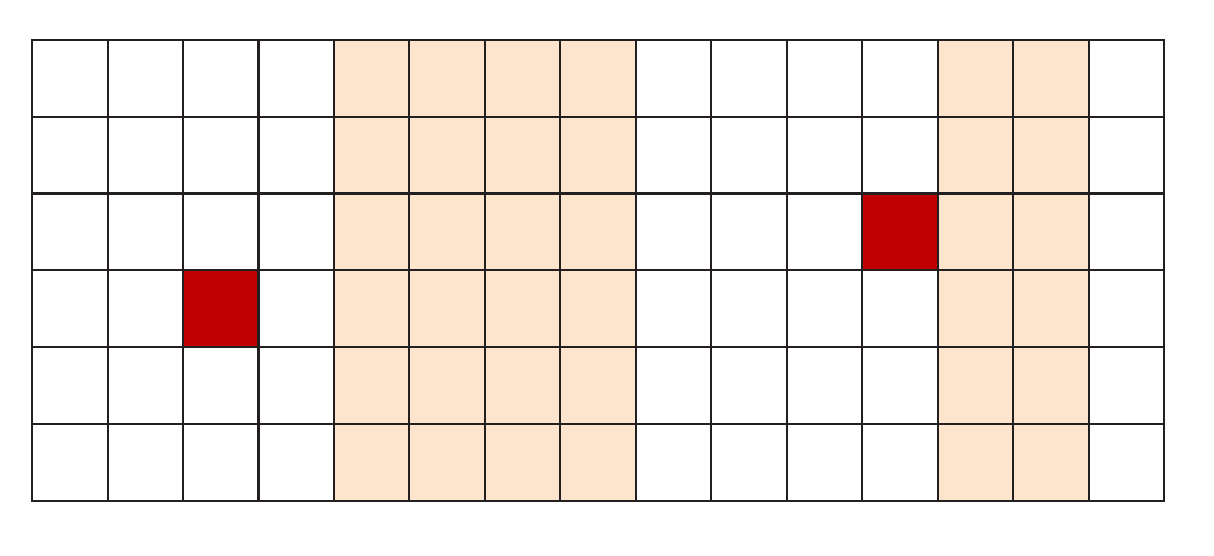}
    \includegraphics[width=0.2\linewidth]{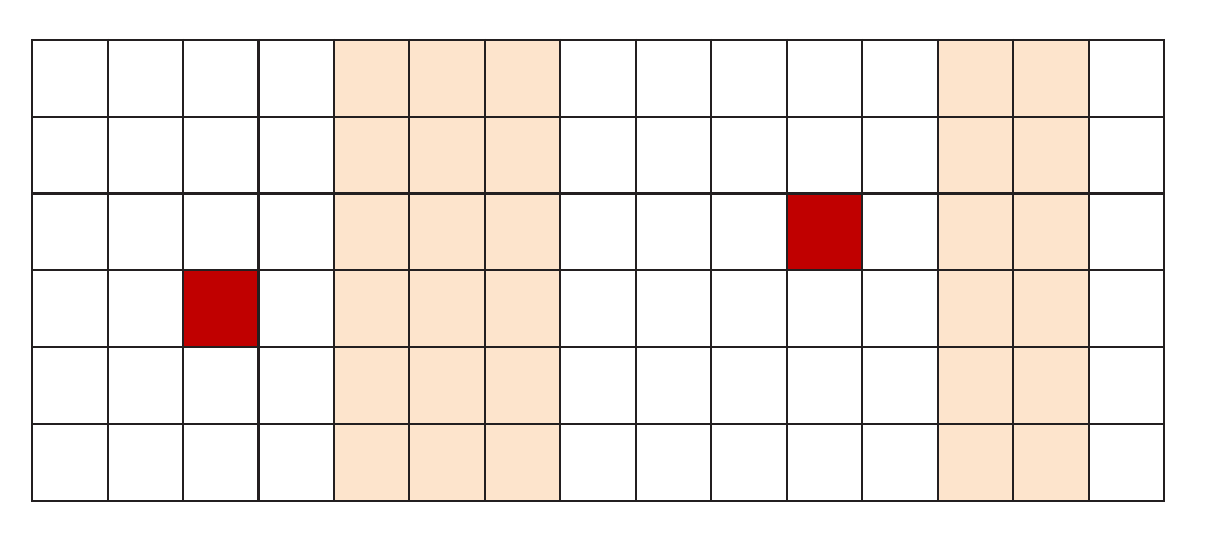}
\caption
    {Visualization of the action sets defined in 
    \eqref{eq:action-ssnc}. Top: sets $A_1(i,j)$ for $i\ge5$ and $j=2,3,4$ and $j\ge5$ (from the left to the right).
   Bottom: as above for $A_2(i,j)$.}
   \label{opt_pol_two_strips_fig-ssnc}
\end{figure}

\subsubsection{Discussion of results}
\label{s:ss-res}
In order to simulate the Ising MDP, a parameter must be
considered. Indeed, after each epoch, that is,
after the MDP flips one spin, the Metropolis zero-temperature
dynamics must be run until a robust configuration is reached,
that is, until the system settles in a local minimum of the
Hamiltonian. To save simulation time, we run it for $\kappa$
steps and choose $\kappa$ sufficiently large so that a local
minimum is reached with a reasonably good approximation.

%In the simulations, we use $N=32$ and the value
%$\kappa=200,000$. To justify this choice, we plot the 
%time-dependent configuration of the system in
%Fig.~\ref{Two_strips_simulation_k5000}, also for
%$\kappa=5000$.

The pictures in Fig.~\ref{Two_strips_simulation_k5000} show the expected behavior for $N = 100$ and $\kappa = 5,000$ (left group) and $\kappa = 20,000$ (right group): after
each MDP action, that is, after a plus spin at distance one or
two is added, the plus cluster grows in the direction
orthogonal to the stripe. Because $\kappa$ is finite, the stripe
structure is not maintained during the evolution, as should
occur according to the theoretical definition of the dynamics.
However, it is clear that this property is better satisfied
when $\kappa$ is larger.

An important characteristic of the dynamics is the expected hitting
time to the all-plus configuration. It is difficult to infer
which of the two policies minimizes the hitting time by simply
looking at the configuration plots of Fig.~\ref{Two_strips_simulation_k5000}. Indeed, this
random variable is strongly affected by the fact that, since
$\kappa$ is finite, the growth process is far from being a
simple horizontal thickening of the stripes, as it would be in
the theoretical case.

We have, however, computed the expected value of the hitting
time in for $N = 32$, with initial seeds consisting of two stripes
of width 3 at distance 13, averaging over $2,000$ independent realizations of the MDP process. In these simulations, we use $\kappa=100,000$. Given the evolution of the system shown in Fig.~\ref{Two_strips_simulation_k5000} for $N = 100$, this value of $\kappa$ is expected to be large enough to maintain the stripe structure to a satisfactory extent. We found $34.915$ for policy $\pi_1$, with a $95\%$-confidence interval of $(34.732, 35.098)$ and $37.569$ for
policy $\pi_2$, with a $95\%$-confidence interval of $(37.278, 37.861)$. Thus, we can conclude that policy $\pi_1$ is
faster in reaching the all-plus configuration. Policy $\pi_2$ exhibits greater variability in its hitting times. 

We remark that this result is not at all obvious. Indeed, in
policy $\pi_2$ the stripes may increase their width by two
units at each epoch. However, a plus spin added at distance 2 is
surrounded by four minus spins, and therefore has a high
probability of being flipped back to minus by the subsequent
Metropolis zero-temperature dynamics, producing no net increase
in stripe width. Hence, the expected hitting time results from
the combination of these two contrasting effects.

%\begin{figure}
%\centering
%\includegraphics[width=0.14\linewidth]
%{Two_strips_pi1_k5000_t50.pdf}
%\includegraphics[width=0.14\linewidth]
%{Two_strips_pi1_k5000_t150.pdf}
%\includegraphics[width=0.14\linewidth]
%{Two_strips_pi1_k5000_t250.pdf}
%\hspace{1.cm}
%\includegraphics[width=0.14\linewidth]
%{Two_strips_pi1_k20000_t50.pdf}
%\includegraphics[width=0.14\linewidth]
%{Two_strips_pi1_k20000_t100.pdf}
%\includegraphics[width=0.14\linewidth]
%{Two_strips_pi1_k20000_t150.pdf}
%\\
%\vspace{2mm}
%\includegraphics[width=0.14\linewidth]
%{Two_strips_pi1_k5000_t50.pdf}
%\includegraphics[width=0.14\linewidth]
%{Two_strips_pi2_k5000_t150.pdf}
%\includegraphics[width=0.14\linewidth]
%{Two_strips_pi2_k5000_t250.pdf}
%\hspace{1.cm}
%\includegraphics[width=0.14\linewidth]
%{Two_strips_pi2_k20000_t50.pdf}
%\includegraphics[width=0.14\linewidth]
%{Two_strips_pi2_k20000_t100.pdf}
%\includegraphics[width=0.14\linewidth]
%{Two_strips_pi2_k20000_t150.pdf}
%\caption{Illustration of policies $\pi_1$ (upper row) and $\pi_2$ (bottom row) for $N=100$ with initial seeds two stripe of width
%$3$ at distance $47$.
%Left group: $\kappa = 5000$ 
%at times $t = 50, 150,250$ (from left to right).
%Right group:
%$\kappa = 20,000$ at times 
%$t = 50,100,150$ (from left to right).
%}
%\label{Two_strips_simulation_k5000}
%\end{figure}

The data from the simulation were also used
to estimate the value function,
averaging \\$\sum_{t=0}^{\infty} \lambda^t r_t(X^\pi_t)$ over the several realizations of the process; 
see the definition in expression \eqref{value_function}. To compute this quantity, we inserted the observed first hitting times to the all-plus configuration in expression (\ref{exp_1}). Our results are
reported in the left panel of Fig.~\ref{Two_strips_policy_comparison},
together with the exact results computed in the next
Section~\ref{s:rigorous}.

The first remark is that the numerical values are compatible
with the exact ones within the statistical error; indeed, we
may say that they are very close. We also observe that the
statistical error becomes quite large when $\lambda$
approaches $1$, reasonably because the value function diverges
as $1/(1-\lambda)$; see \eqref{eq:divergence}.

\begin{figure}
\centering
\includegraphics[width=0.12\linewidth]
{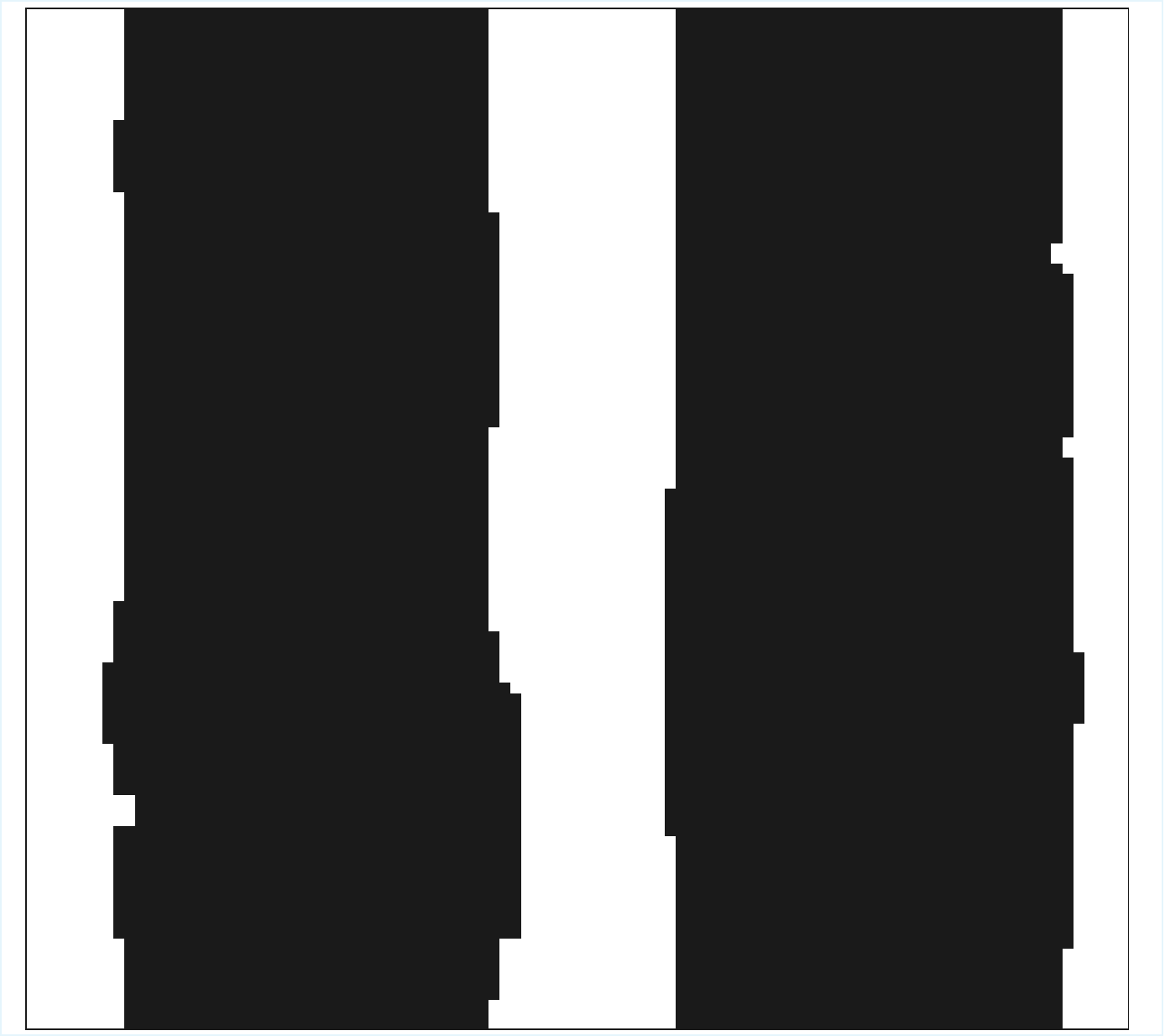}
\includegraphics[width=0.12\linewidth]
{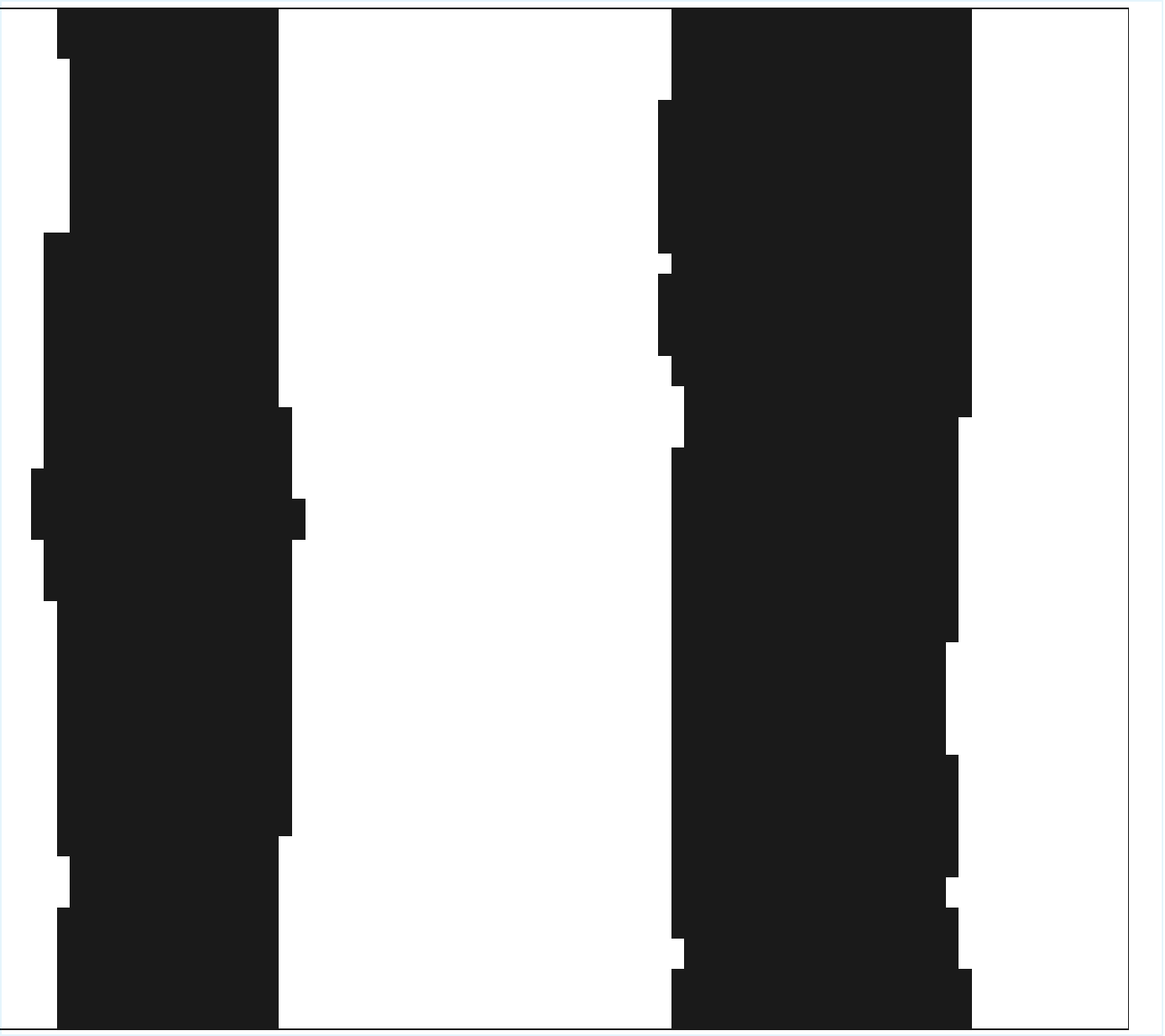}
\includegraphics[width=0.12\linewidth]
{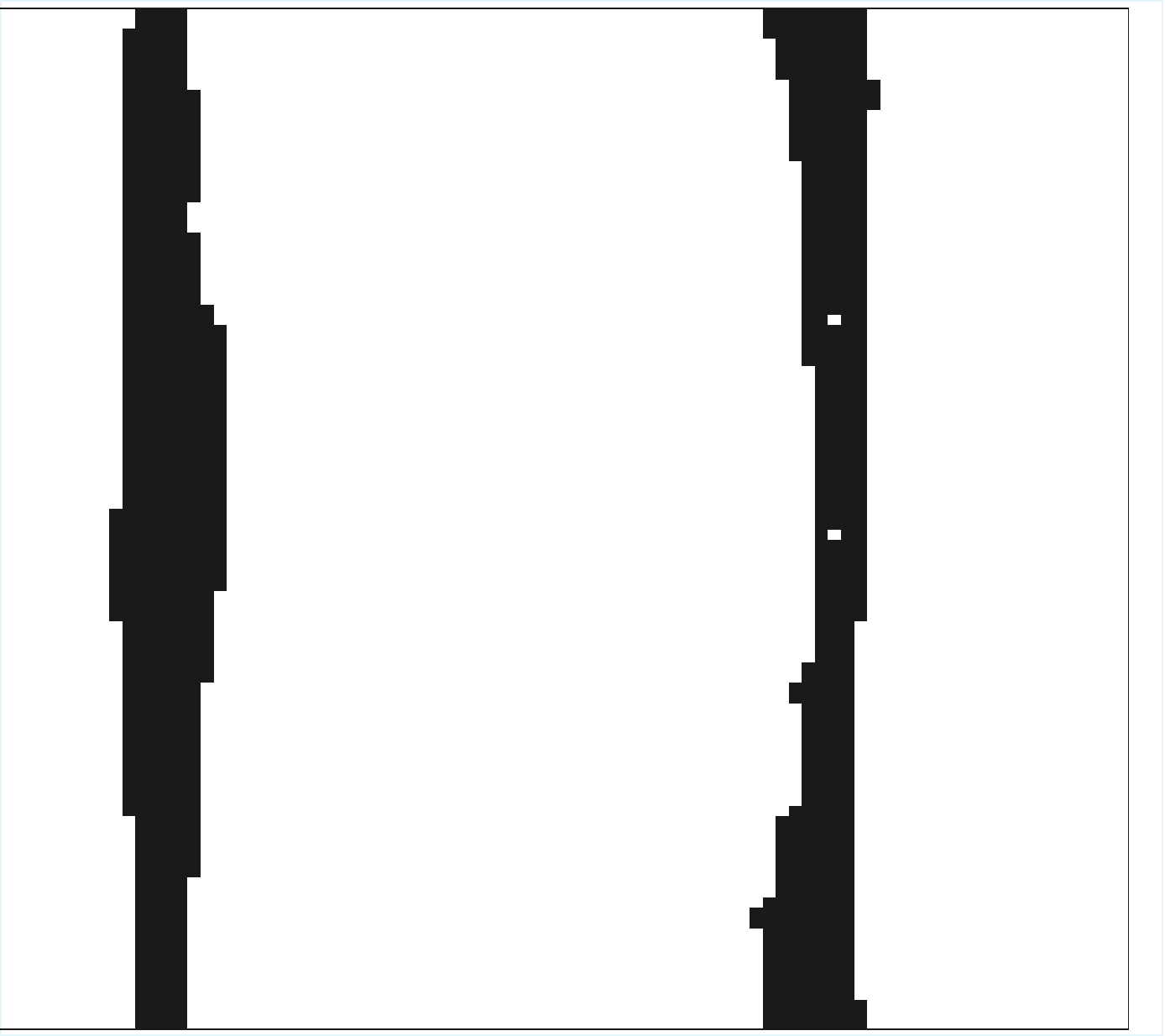}
\hspace{1.cm}
\includegraphics[width=0.12\linewidth]
{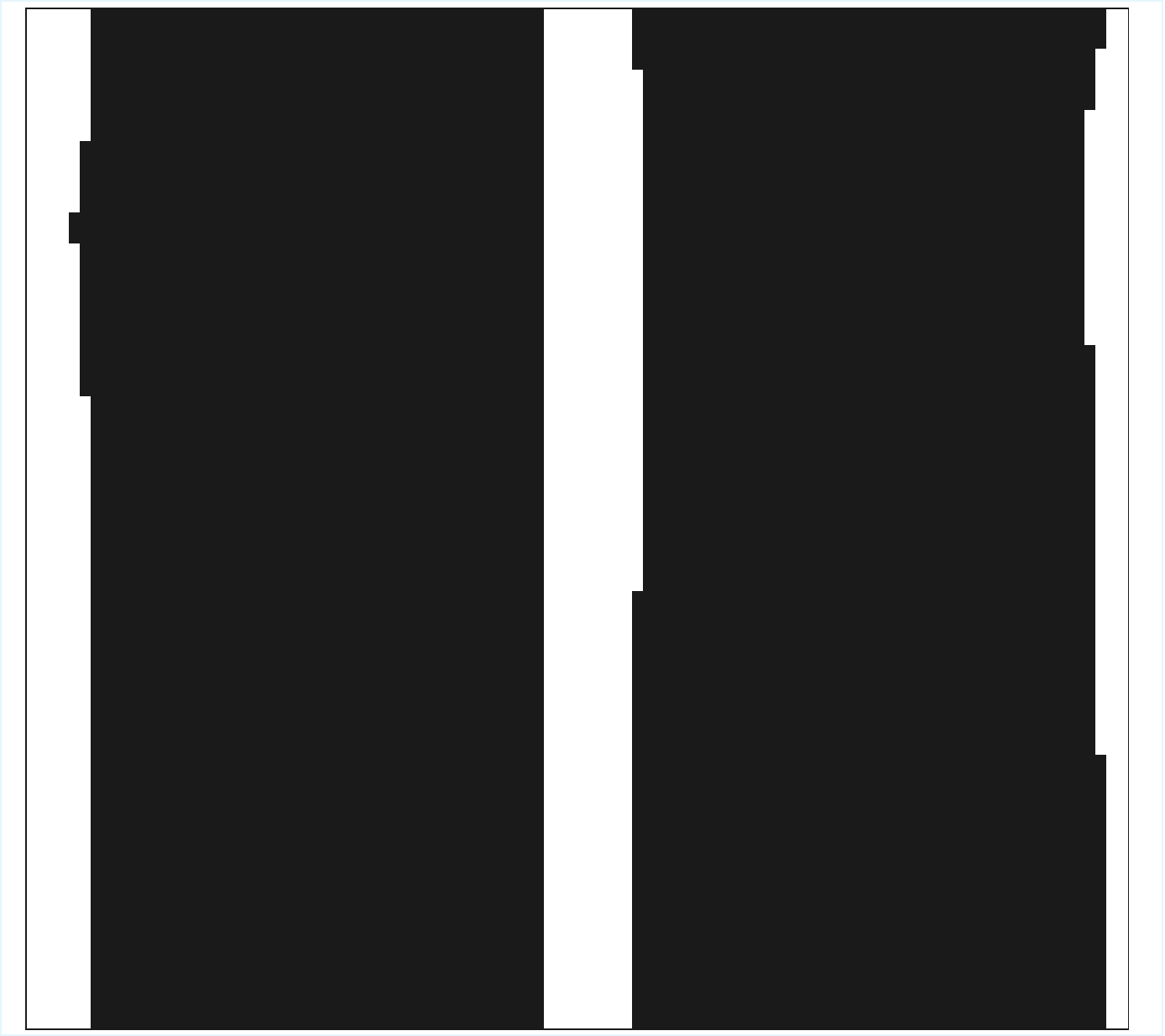}
\includegraphics[width=0.12\linewidth]
{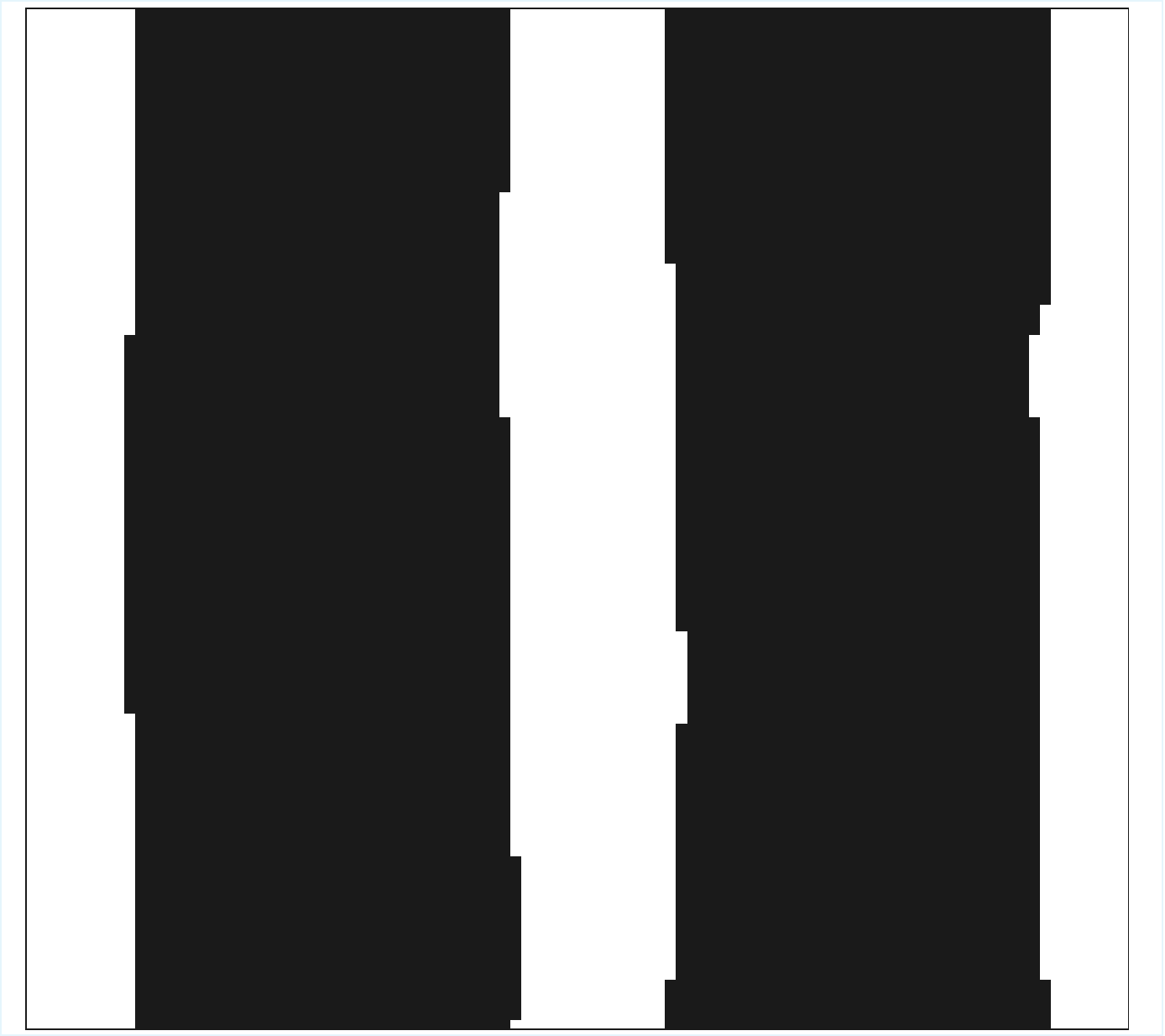}
\includegraphics[width=0.12\linewidth]
{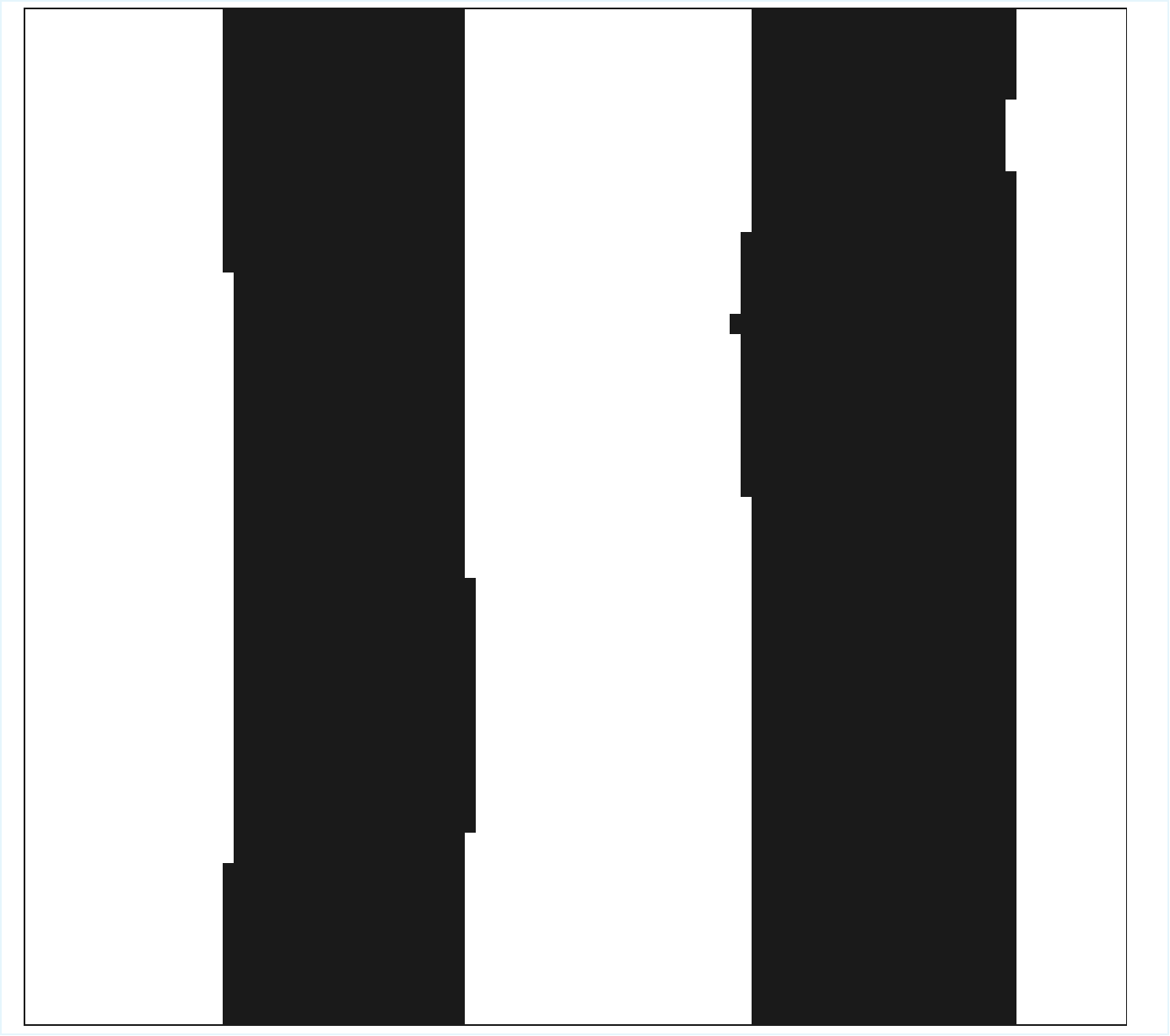}
\\
\vspace{1mm}
\includegraphics[width=0.12\linewidth]
{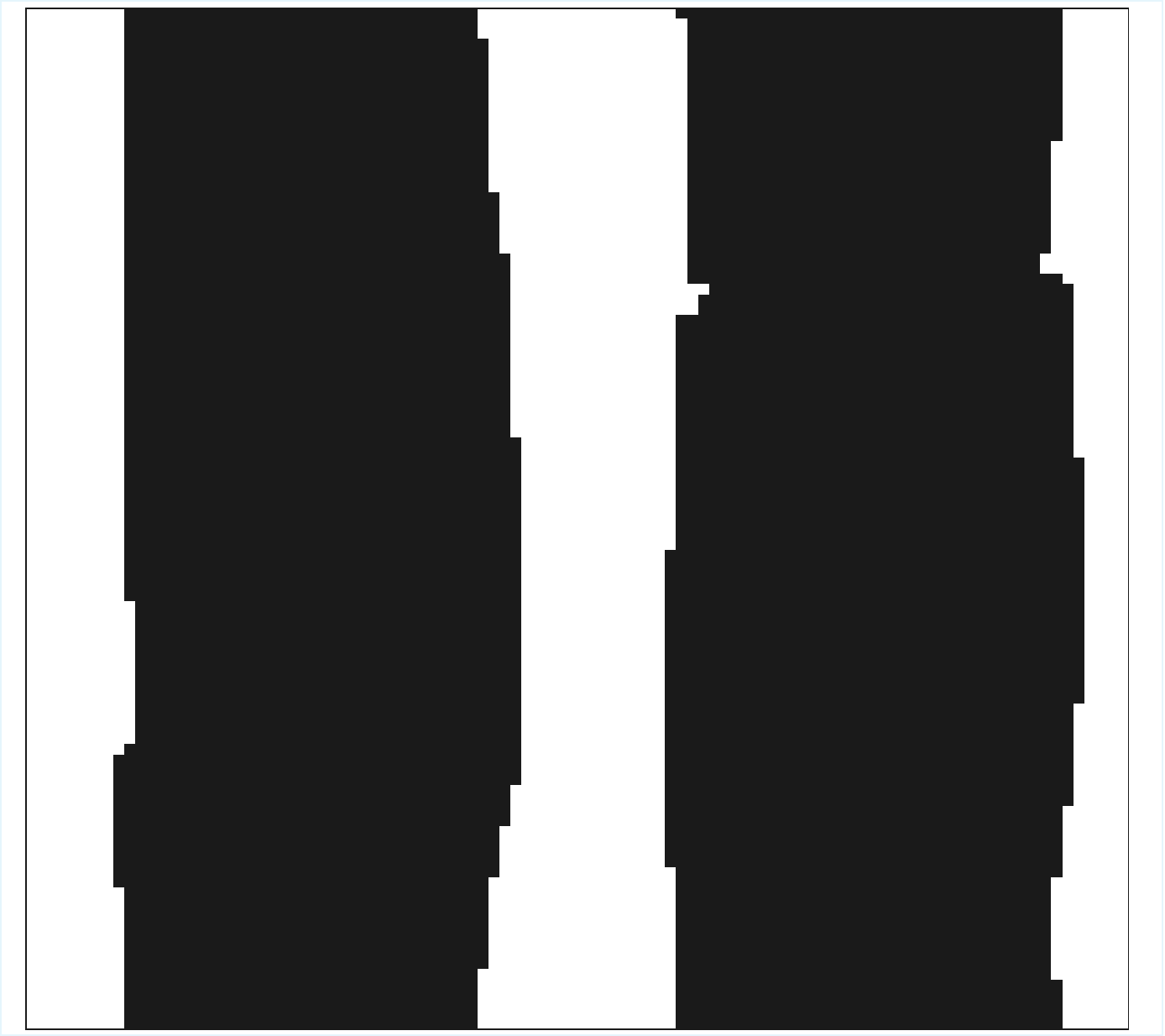}
\includegraphics[width=0.12\linewidth]
{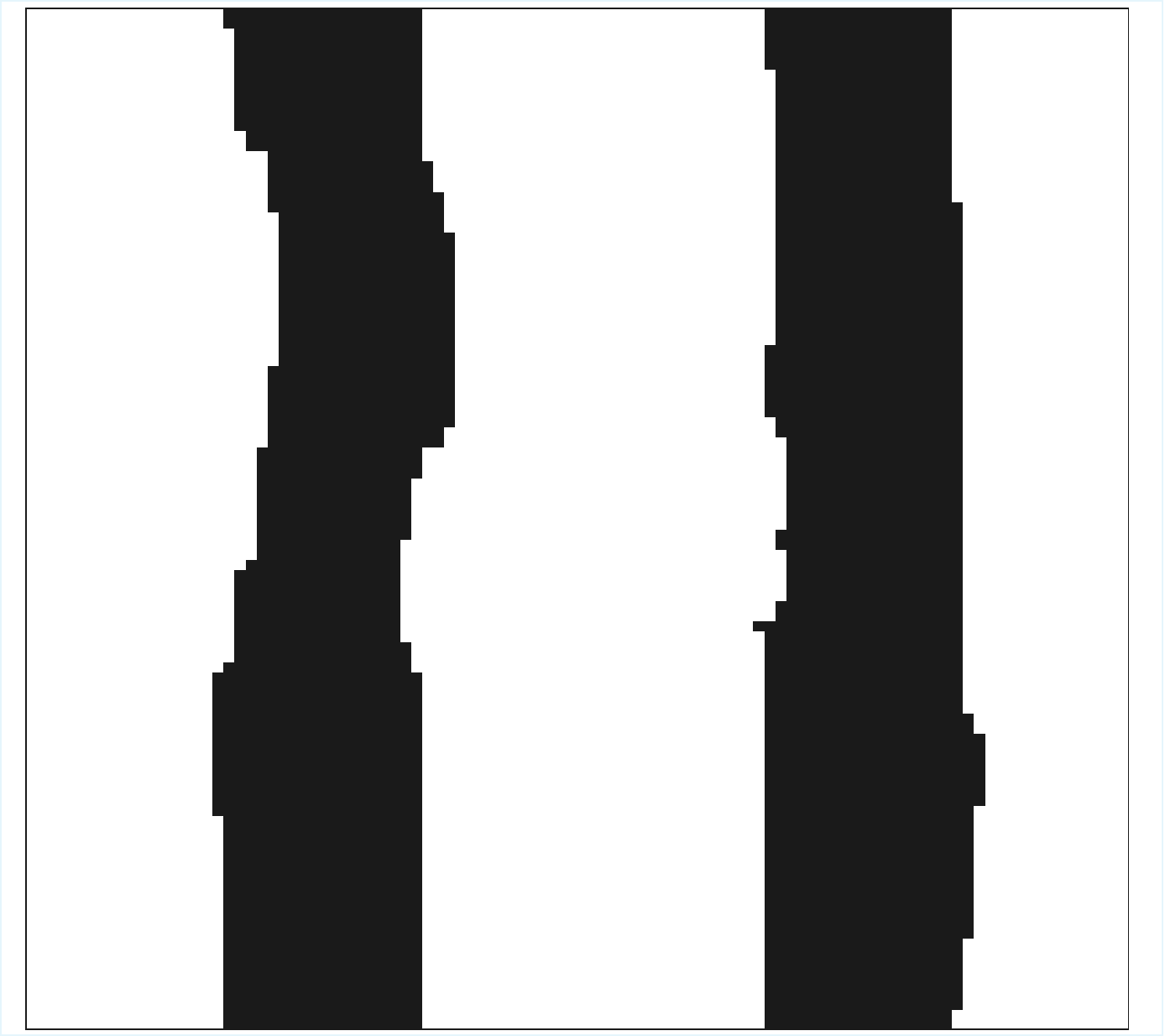}
\includegraphics[width=0.12\linewidth]
{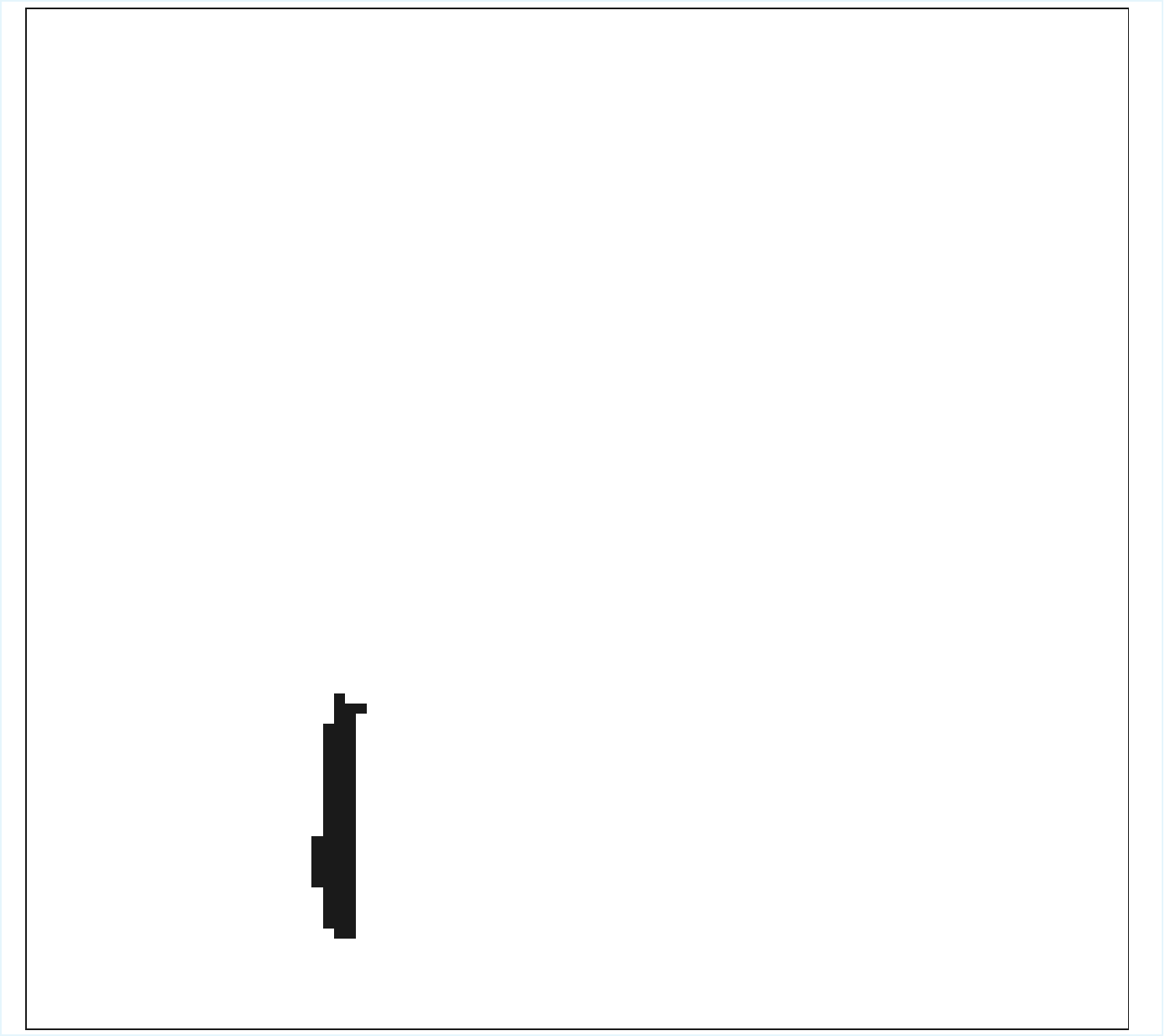}
\hspace{1.cm}
\includegraphics[width=0.12\linewidth]
{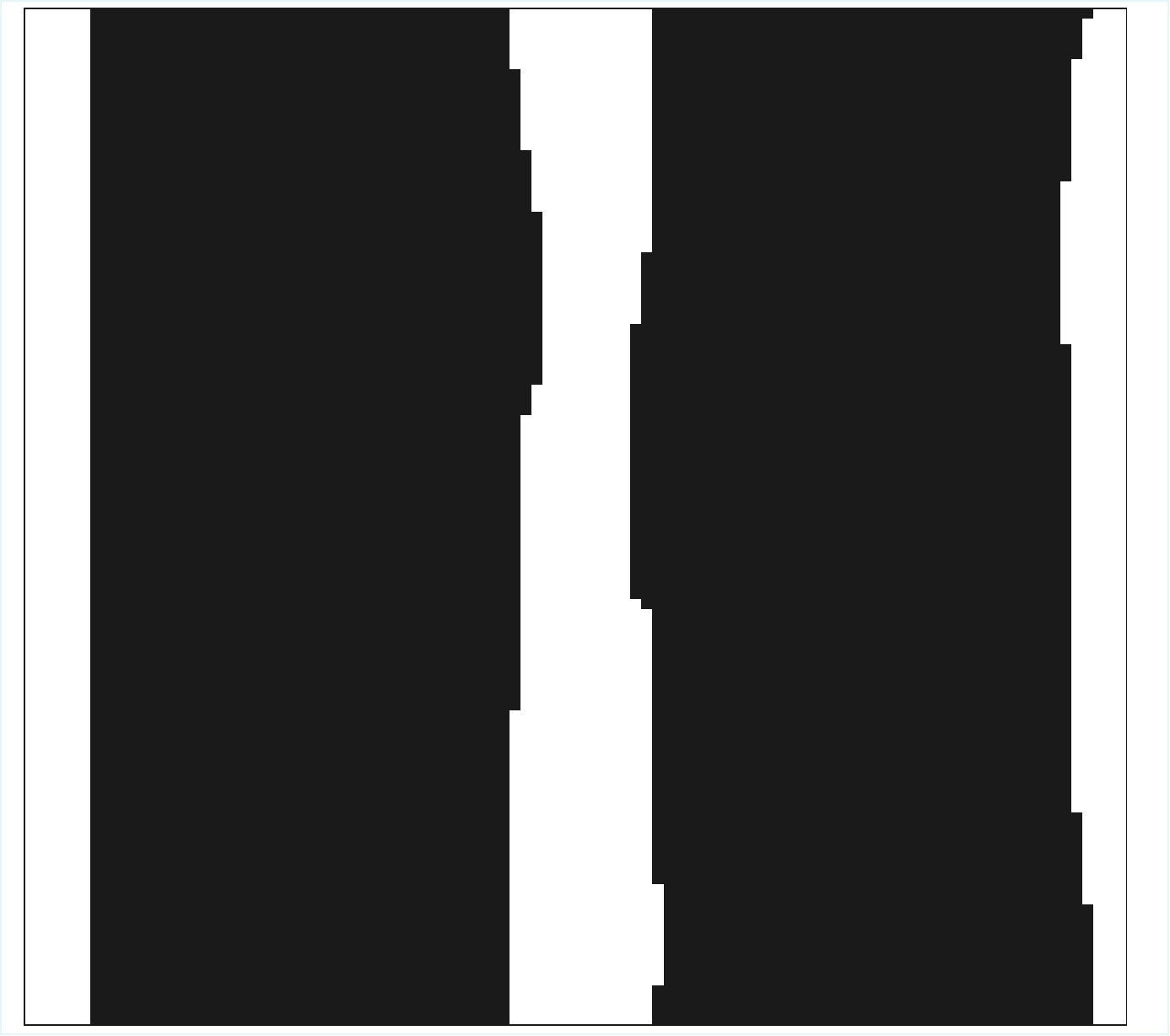}
\includegraphics[width=0.12\linewidth]
{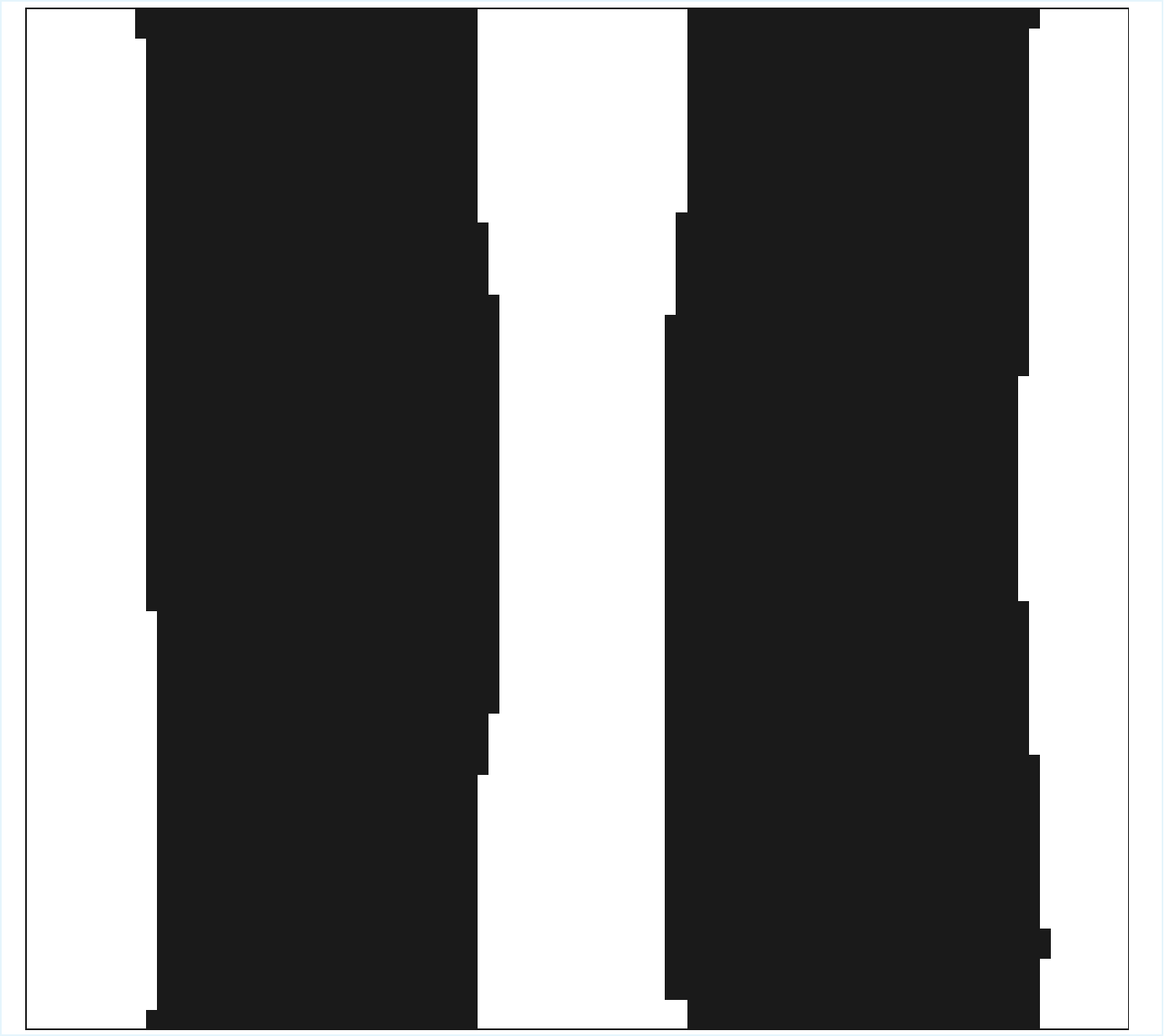}
\includegraphics[width=0.12\linewidth]
{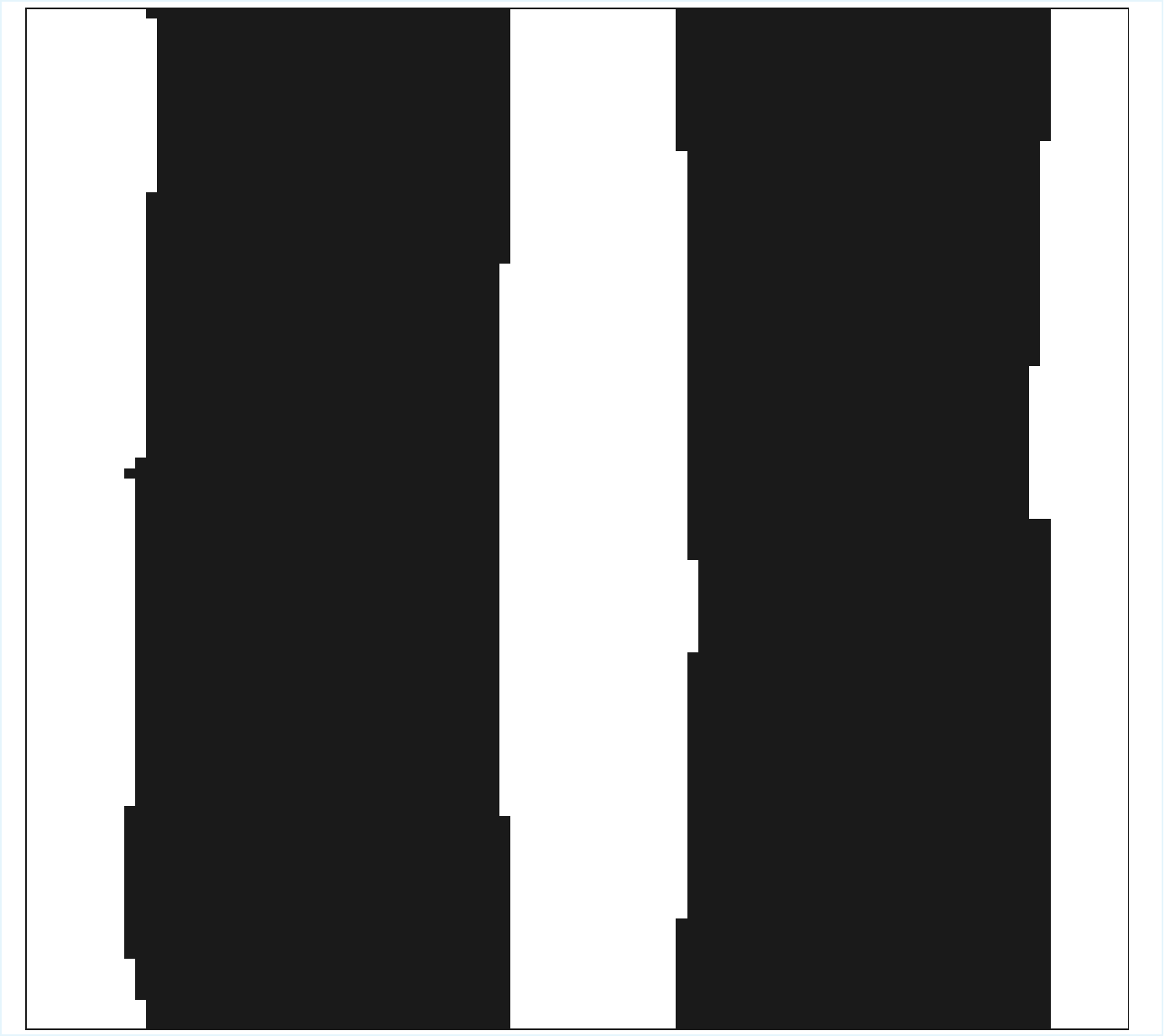}
\caption{Illustration of policies $\pi_1$ (upper row) and $\pi_2$ (bottom row) for $N=100$ with initial seeds two stripe of width
$3$ at distance $47$.
Left group: $\kappa = 5000$ 
at times $t = 200, 400,600$ (from left to right).
Right group:
$\kappa = 20,000$ at times 
$t = 50,100,150$ (from left to right).
}
\label{Two_strips_simulation_k5000}
\end{figure}

Another important observation is that the numerical
computations cannot help us determine which policy, among
$\pi_1$ and $\pi_2$, is the best, since the corresponding value
functions are so close that their difference is much smaller
than the statistical error.

On this basis, we may rely on the rigorous analytical result
to be established in Section~\ref{s:rigorous} and anticipated
in Fig.~\ref{Two_strips_policy_comparison}. The right panel of
the figure clearly indicates that the value functions of the
two policies intersect at $\lambda_\text{c}\approx 0.88$, showing
that $\pi_1$ is optimal for $\lambda>\lambda_\text{c}$, whereas
$\pi_2$ prevails for smaller values of $\lambda$. This outcome
is fully consistent with the discussion in
Section~\ref{s:valuemeaning}, where in Theorem~\ref{t:value010}
we proved that, for
$\lambda$ approaching $1$, maximizing the value function with
the reward \eqref{reward_function_isi} is equivalent to minimizing
the hitting time to the all-plus configuration. As observed
above, this hitting time under policy $\pi_1$ is shorter than
that obtained under the competing policy $\pi_2$.

On the other hand, recalling Theorem~\ref{t:value050}, 
we may argue that, when $\lambda$ is small,
only trajectories reaching the all-plus configuration within a
short time significantly contribute to the value function.
Therefore, $\pi_2$ is expected to be the policy capable of
selecting those trajectories that accomplish a short flight to
the all-plus configuration.

\begin{figure}[t]
\centering
\includegraphics[width=0.465\linewidth]
{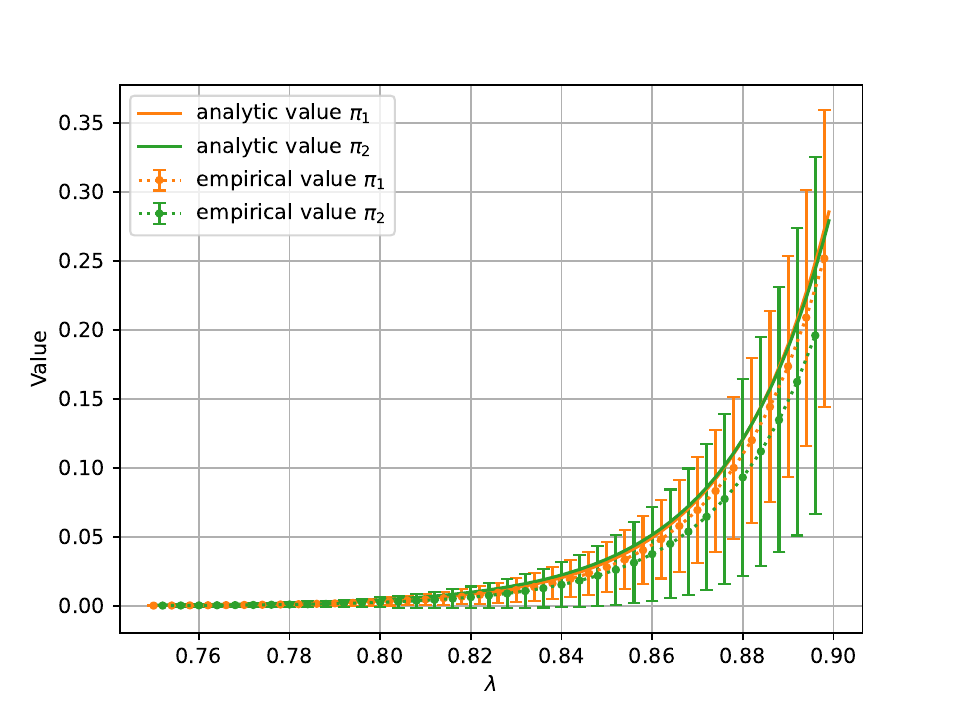}
\hspace{-1cm}
\raisebox{-0.5mm}{
\includegraphics[width=0.465\linewidth]
{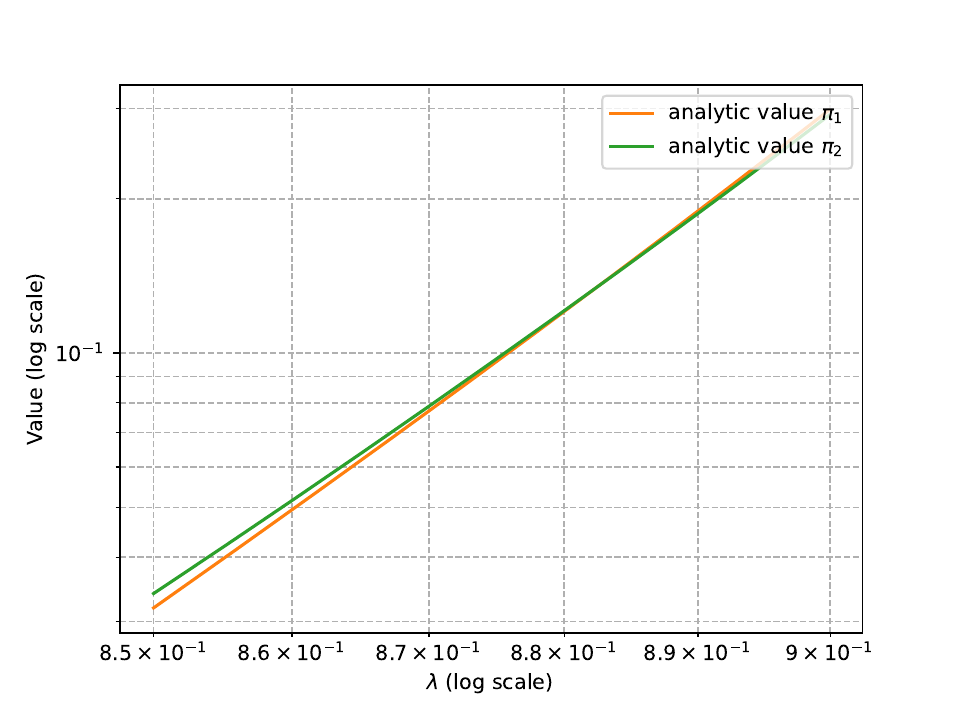}
}
\caption{Empirical and analytic values of policies $\pi_1$ and $\pi_2$
(left), and analytic values of policies $\pi_1$ and $\pi_2$ on
a log scale (right), for $N = 32$ and $\kappa = 100,000$, with
initial seeds consisting of two stripes of width $3$ at a
distance of $13$. Solid
lines indicate the analytic values, while dots with error bars
represent the numerical estimates. Orange and green are
respectively associated with policies $\pi_1$ and $\pi_2$.
}
\label{Two_strips_policy_comparison}
\end{figure}

\subsection{The stripe-droplet case}
We proceed to investigate the structure of the optimal policy for configurations in which only one of the droplets forms a stripe that is wrapped around the torus. We denote the set of such configurations by $U^{2,y}$. 

\subsubsection{The auxiliary MDP}
We again construct an auxiliary MDP for this particular type of configurations, which provides a compact description of the control problem. Let this MDP be denoted by $(S^y, A^y, P^y, r^y)$. Here, the state space $S^y$ is defined as 
\begin{equation*}
    S^y = \{(i,j,k)|i,j,k = 0, 2, 3, \ldots, N\}.
\end{equation*}
A state $(i, j, k) \in S^y$ is a representation of the set of configurations 
in which the distances between the stripe and the rectangle equal $i$ and $j$, and
$k=N-\ell$, where $\ell$ is the side length of the rectangle in the direction 
parallel to the stripe. In other words, $k$ is 
the distance between the two boundaries of the rectangle 
measured around the torus. Here, a state $(i,j,0) \in S^y$ is equivalent to state $(i,j) \in S^x$, in the sense that they represent the same set of configurations. The state space $S^y$ is illustrated in Figure \ref{Aux_spaces_strip_droplet}.

Let the action space $A^y$ be defined as 
\begin{equation*}
    A^y = \{a_{\ell 1}, a_{\ell 2}, a_{s1}, a_{s2}, a_{d1}, a_{d2}, a_0, 0\}.
\end{equation*}
Here, $a_{\ell1}$ and $a_{\ell2}$ represent the sets of sites at distance 1 and 2 respectively from either the stripe or the droplet in the longest gap between the stripe and the droplet. Similarly, $a_{s1}$ and $a_{s2}$ represent the sets of sites at distance 1 and 2 respectively from either the stripe or the droplet in the shortest gap between the two. The actions $a_{d1}$ and $a_{d2}$ represent the sets of sites at distance 1 or 2 from the sides of the droplet that are perpendicular to the front of the stripe. Finally, action $a_0$ represents the set of sites that are diagonally adjacent to the droplet. The action space $A^y$ is illustrated in Figure \ref{Aux_spaces_strip_droplet}.

\subsubsection{Candidates for optimality}
\label{s:sd-opt}
For this scenario, we conduct a numerical comparison between policies $\pi_1 = (d_1)^{\infty}$, $\pi_2 = (d_2)^{\infty}$, $\pi_3 = (d_3)^{\infty}$ and $\pi_4 = (d_4)^{\infty}$, defined in the auxiliary MDP, where $d_q(i,j, k) \in A_q(i,j, k)$, $q = 1, 2, 3, 4$. Here, the mappings $A_q: S^y \rightarrow P(A^y)$, $q = 1, 2, 3, 4$, again associate to each state $s \in S^y$ a corresponding set of actions, defined for states $(i,j, k) \in S^y$ as follows:
\begin{align*}
    A_1(i,j,k) &= \{a_{s1}, a_{\ell 1}\}, \\
    A_2(i,j,k) &= \begin{cases}
        \{a_{d1}\}, &\text{if } k \neq 0, \\
        \{a_{\ell 1}, a_{s1}\}, &\text{if } k = 0, 
    \end{cases}\\
    A_3(i,j,k) &= \begin{cases}
        \{a_{d1}, a_{\ell 1}, a_{s1}\}, &\text{if } k \neq 0, \\
        \{a_{\ell 1}, a_{s1}\}, &\text{if } k = 0, 
    \end{cases}\\
    A_4(i,j,k) &= \begin{cases}
        \{a_0, a_{\ell 1}, a_{s1}\}, &\text{if } k \neq 0, \\
        \{a_{\ell 1}, a_{s1}\}, &\text{if } k = 0, 
    \end{cases}
\end{align*}

\begin{figure}
\centering
\includegraphics[width=0.31\linewidth]
{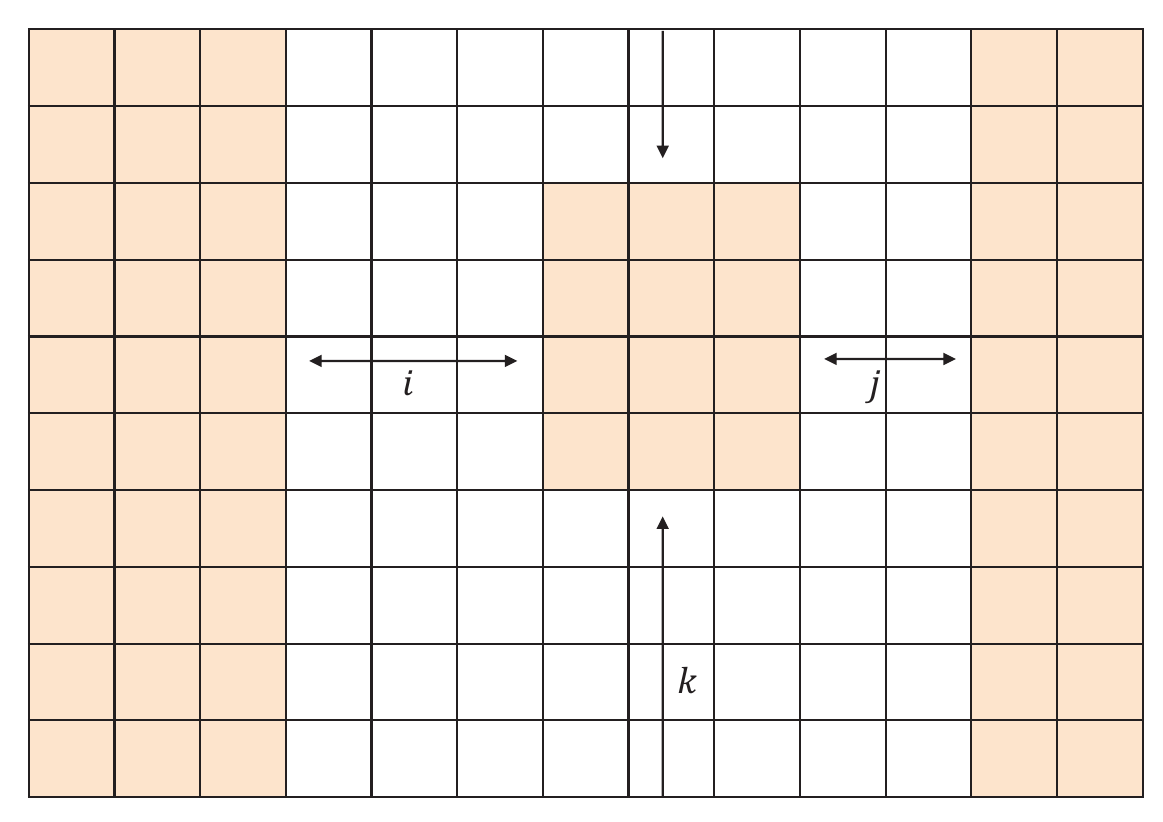}\hspace{2mm}
\raisebox{-0.5mm}{
\includegraphics[width=0.4\linewidth]
{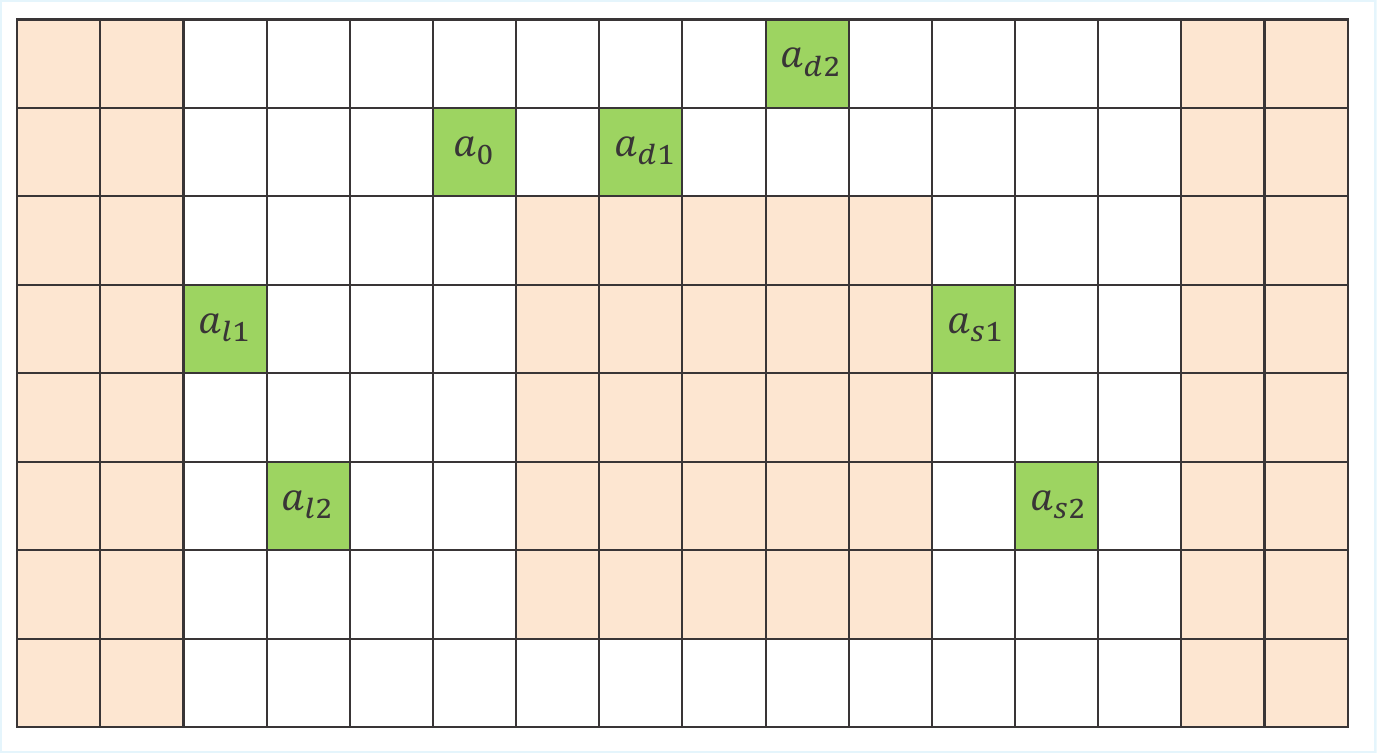}
}
\caption{Illustration of state space $S^y$ (left) and action space $A^y$ (right).}
\label{Aux_spaces_strip_droplet}
\end{figure}

As in the two-stripe case, we use randomized versions of decision rules $d_q$, $q = 1, 2, 3, 4$, to simulate the policies. Under policy $\pi_1$, the stripe and the droplet grow simultaneously towards each other until they meet. Policy $\pi_2$ causes the droplet to first grow into a stripe, after which the two stripes grow in each others direction. Under policy $\pi_3$, the droplet grows into a stripe, while the stripe grows in the direction of the droplet until the two components form a single stripe. Finally, policy $\pi_4$ causes the droplet to grow diagonally, while the stripe expands in the direction of the droplet. The four policies are illustrated in Fig.~\ref{policies_strip_droplet}.

\begin{figure}
\centering
\includegraphics[width=0.2\linewidth]
{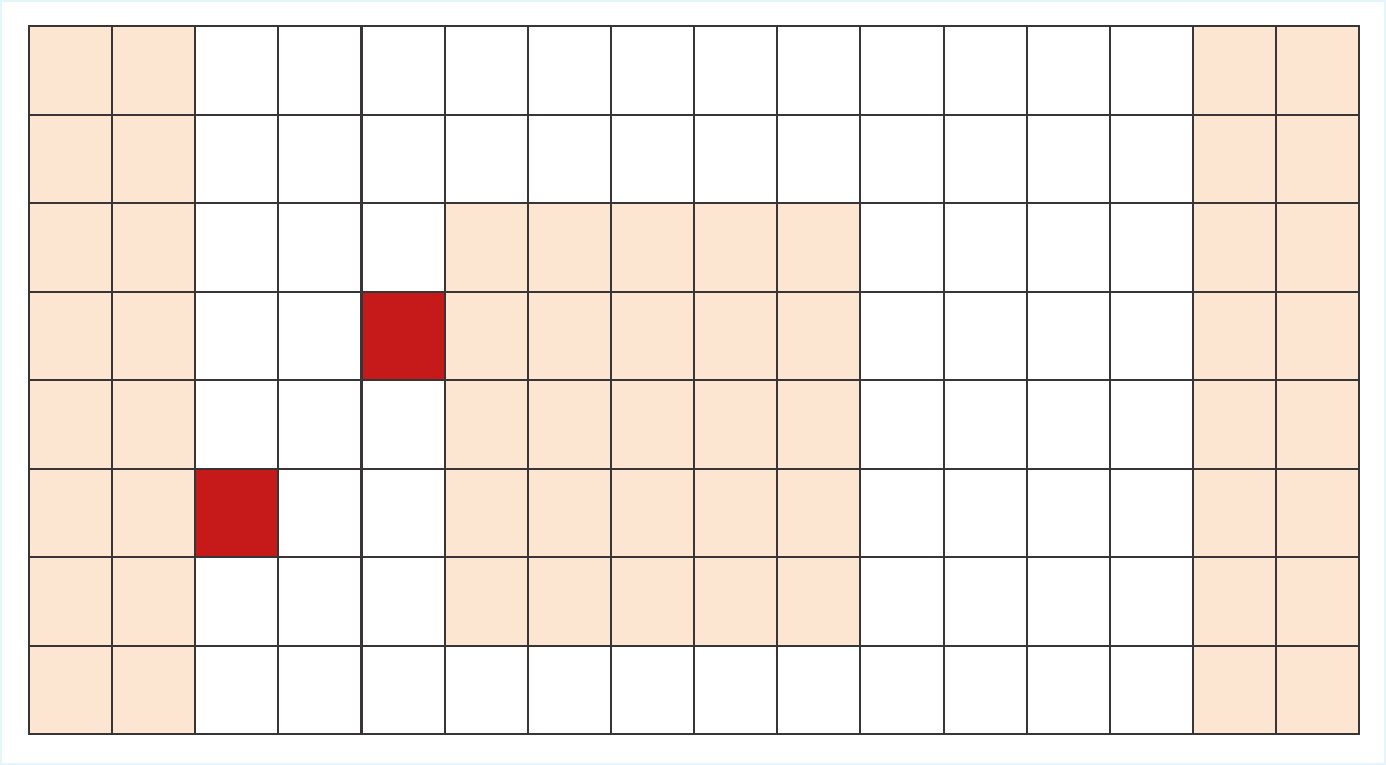}\hspace{2mm}
\includegraphics[width=0.2\linewidth]
{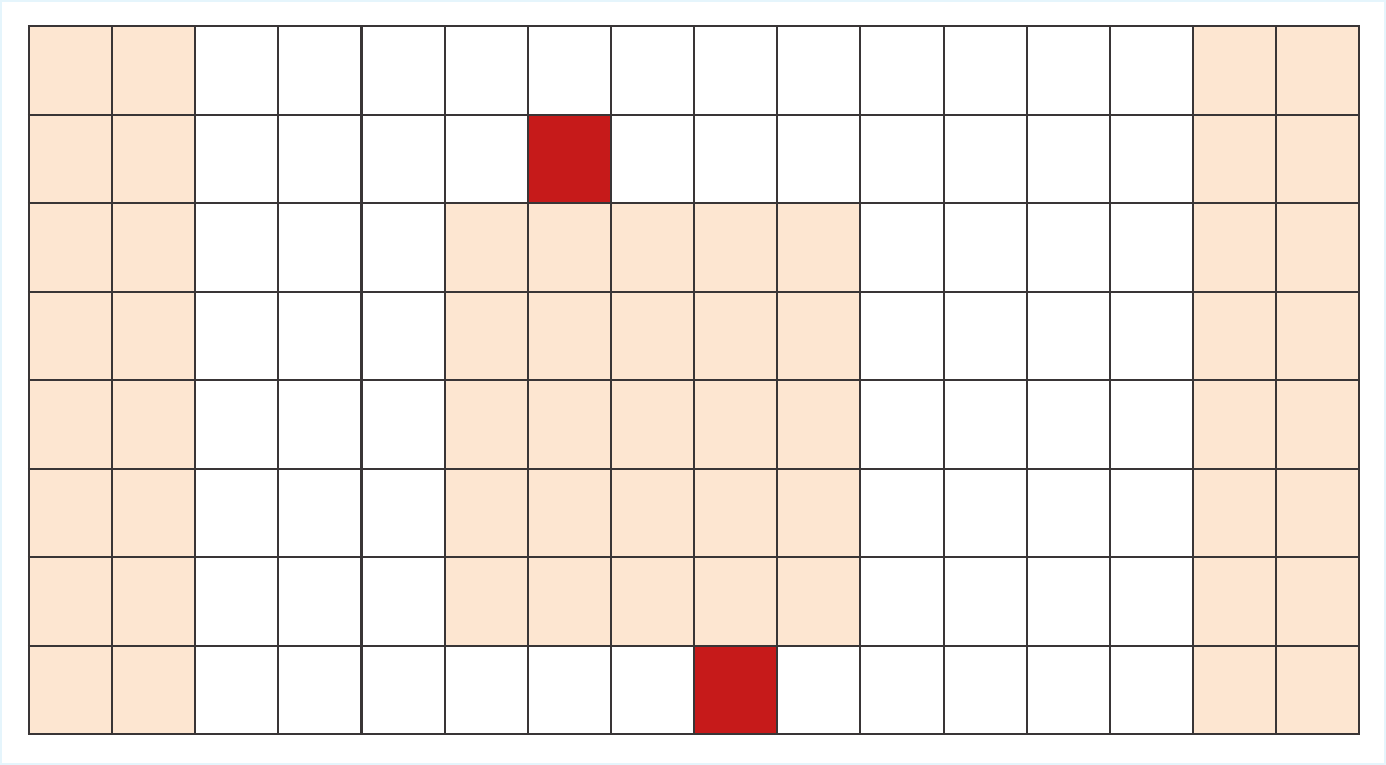}
\includegraphics[width=0.2\linewidth]
{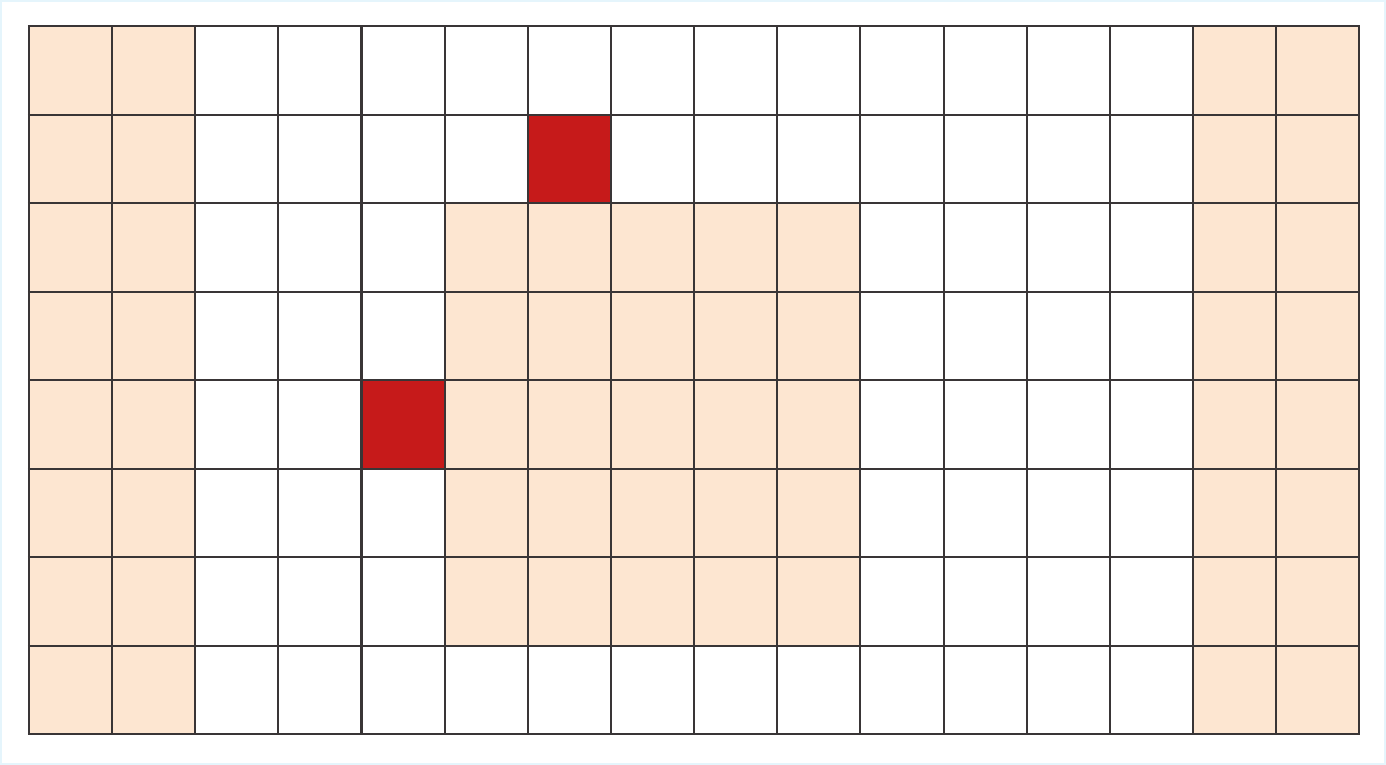}\hspace{2mm}
\includegraphics[width=0.2\linewidth]
{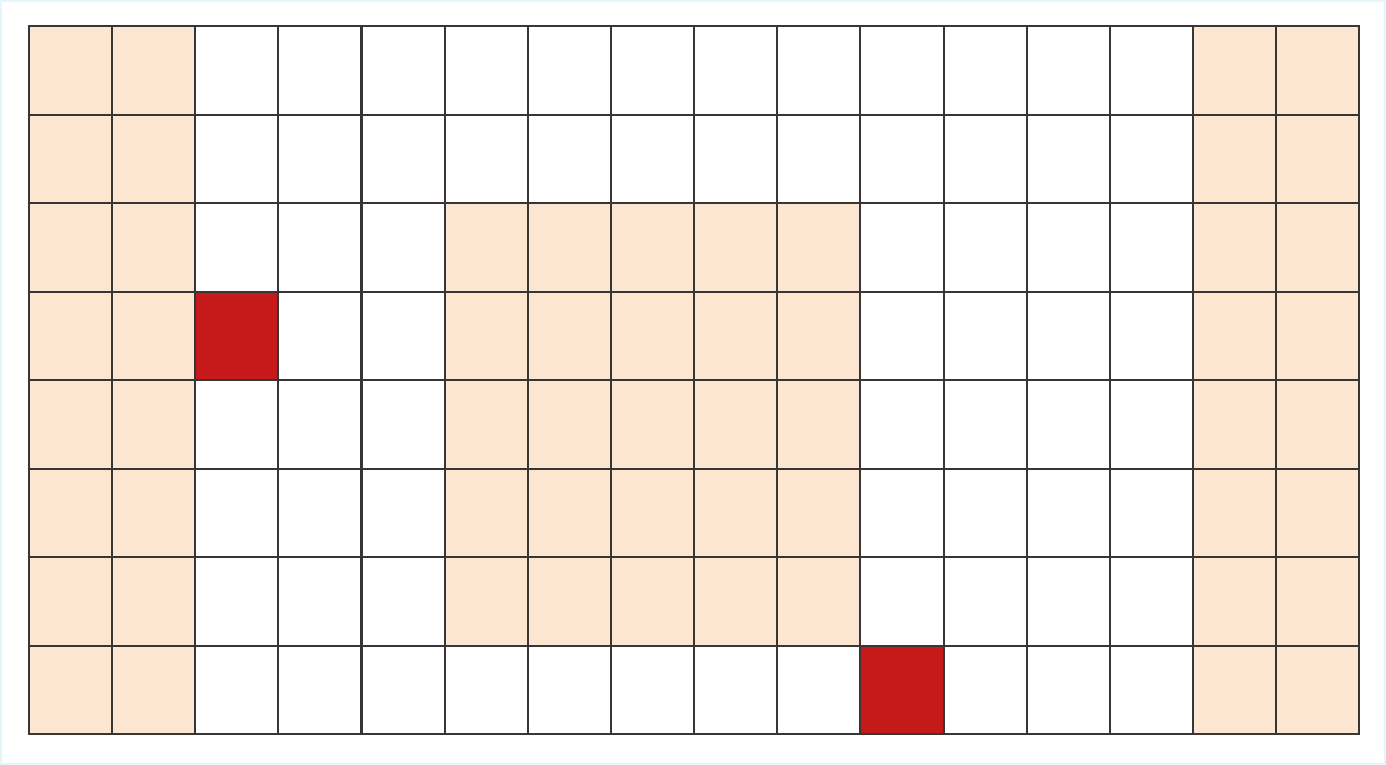}
\caption{Illustration of policies $\pi_1$, $\pi_2$, $\pi_3$, and $\pi_4$ 
(from left to right).}
\label{policies_strip_droplet}
\end{figure}

\subsubsection{Discussion of results}
\label{s:sd-res}

The behavior of the system under the four candidate policies is illustrated in Figure \ref{Strip_droplet_simulations} for $N = 100$ and $\kappa = 5,000$ (left group) and $\kappa = 20,000$ (right group). Again, we observe that for larger $\kappa$, the system evolves more accurately in accordance with a stripe-rectangle structure. 

As in the two-stripe case, we assess the performance of each of the candidate policies by measuring the mean hitting time and mean value across 2,000 independent realizations of the MDP. This simulation study is based on $N = 32$ and $\kappa = 100,000$, with initial seeds one stripe of width 3 and one 3x3 droplet at distance 13.  

\begin{figure}
\centering
\includegraphics[width=0.12\linewidth]
{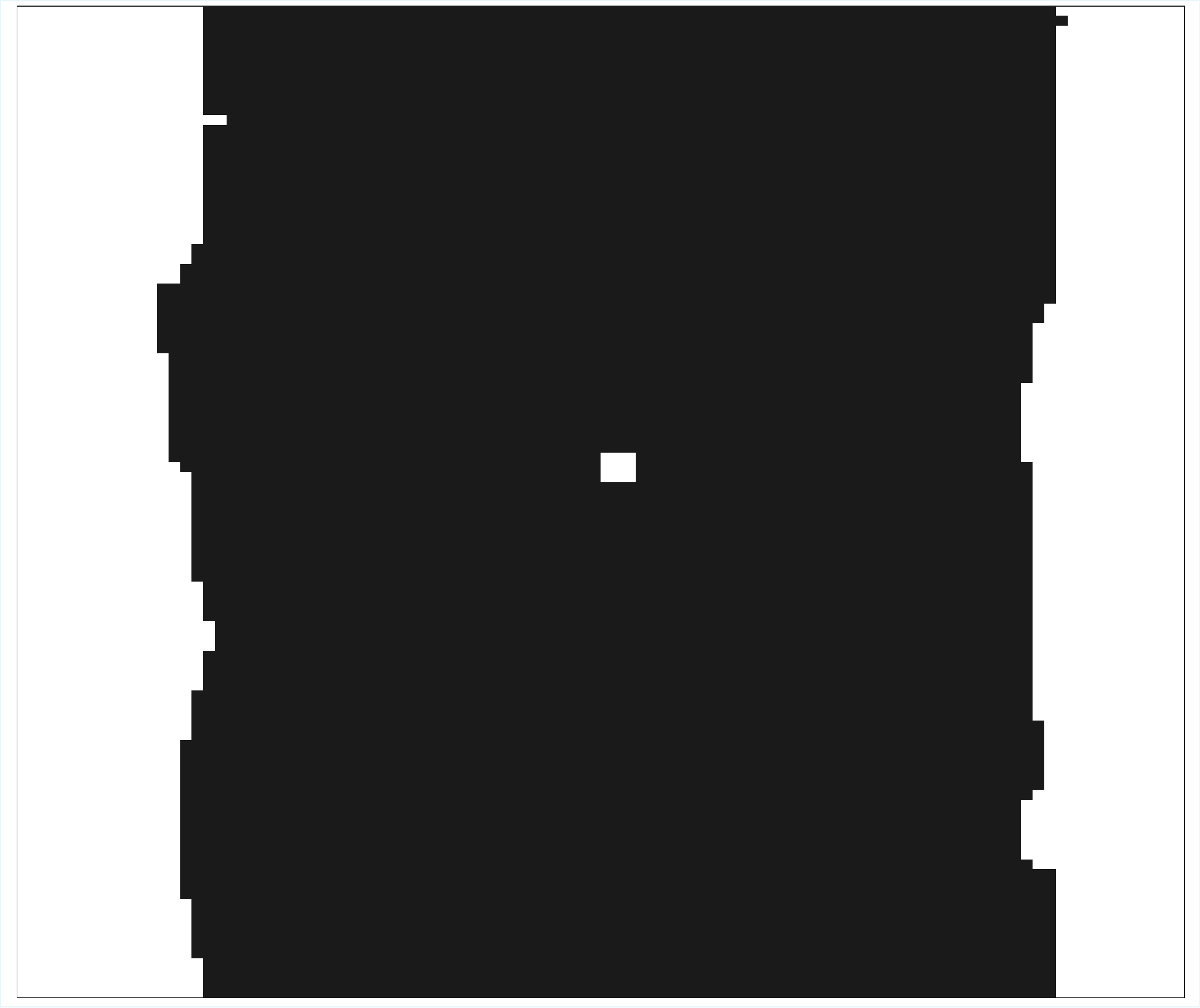}
\includegraphics[width=0.12\linewidth]
{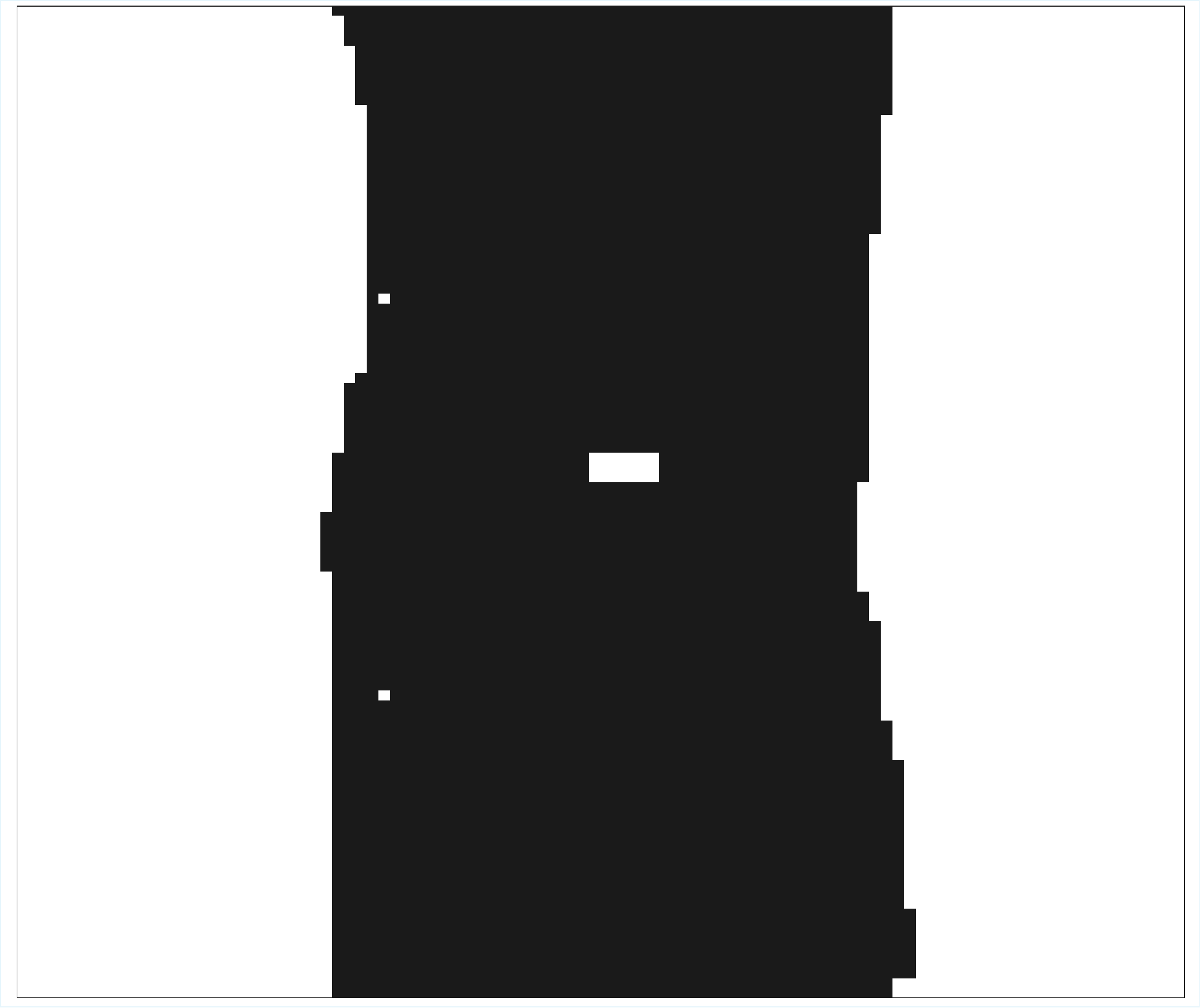}
\includegraphics[width=0.12\linewidth]
{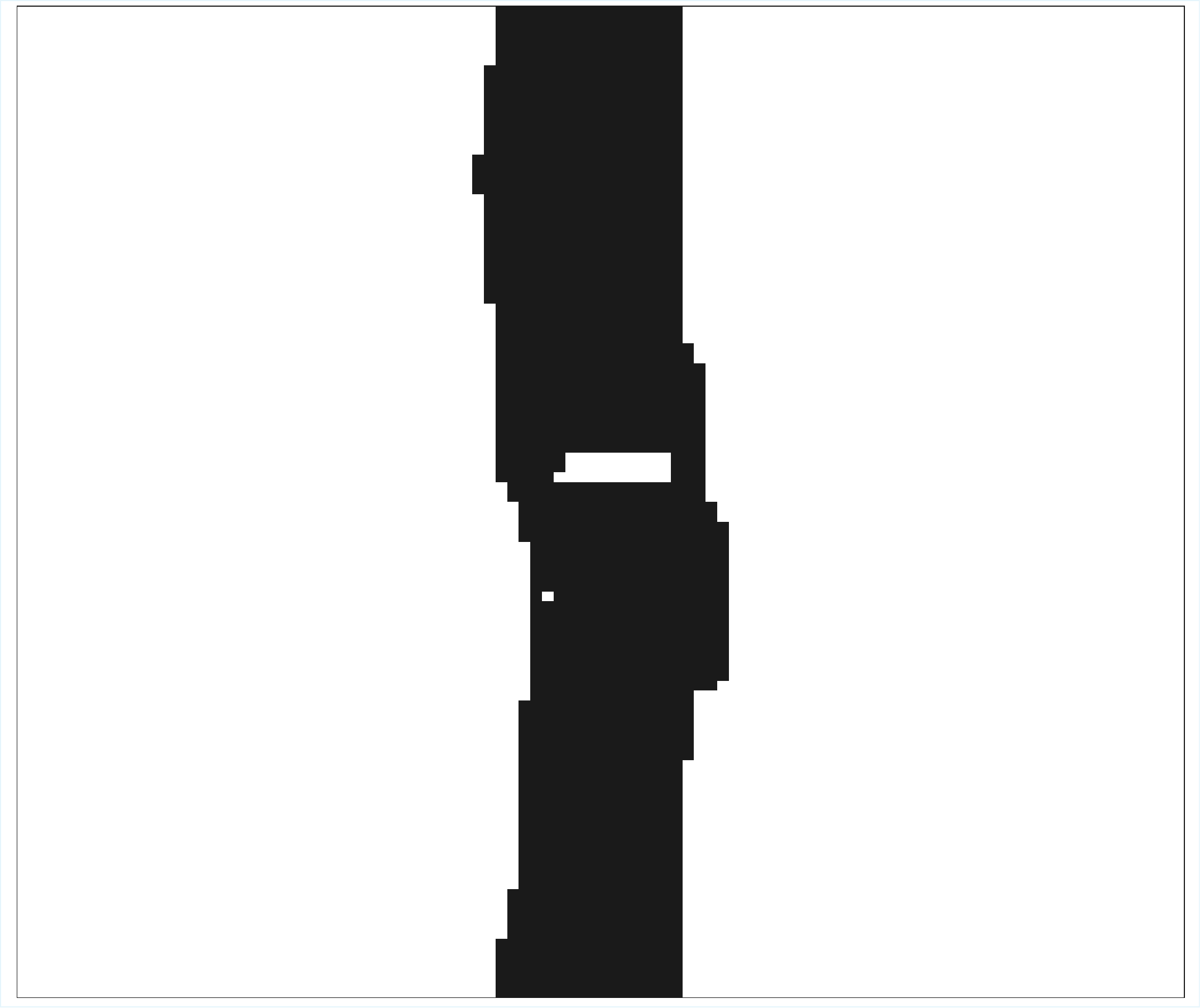}
\hspace{1.cm}
\includegraphics[width=0.12\linewidth]
{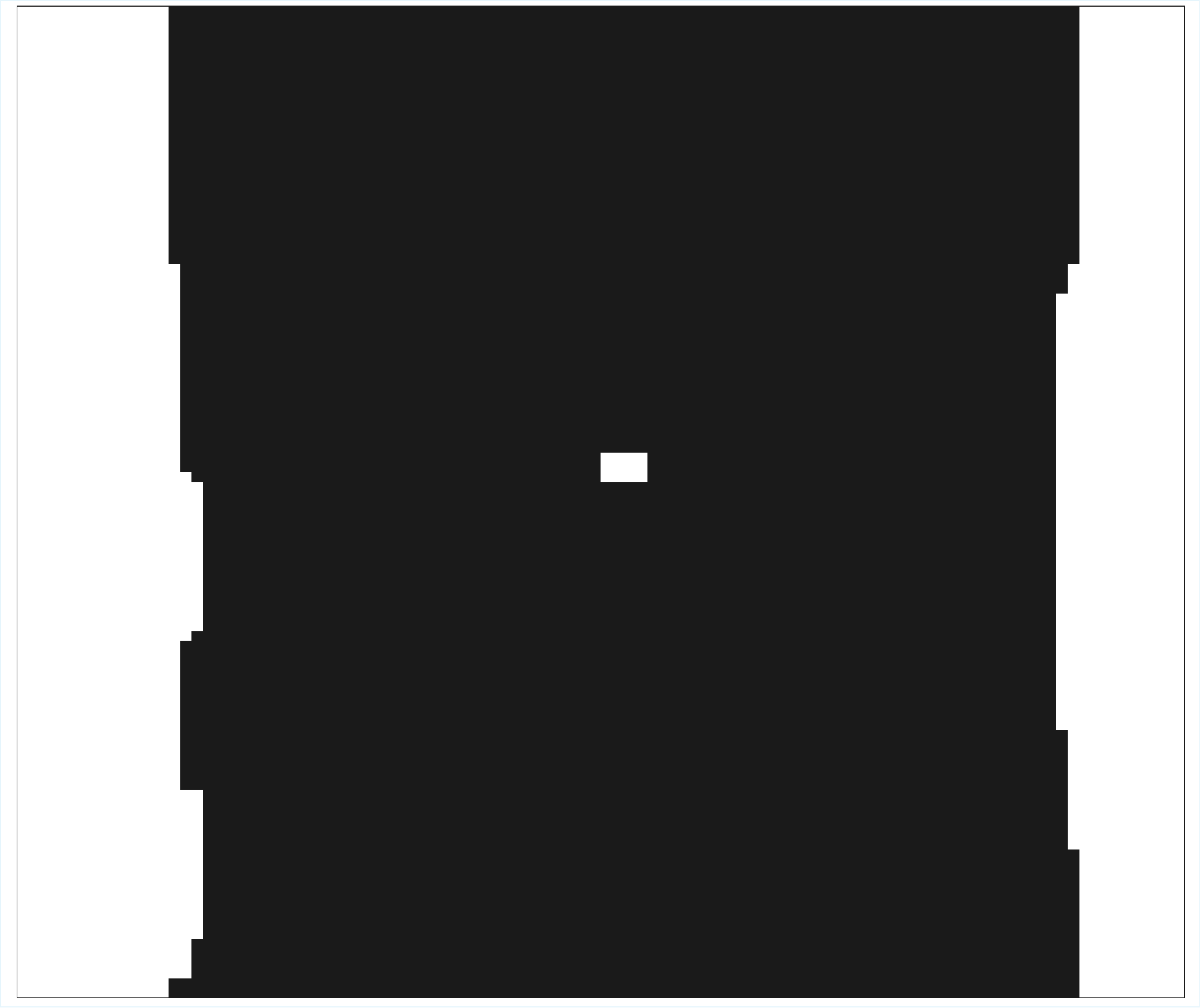}
\includegraphics[width=0.12\linewidth]
{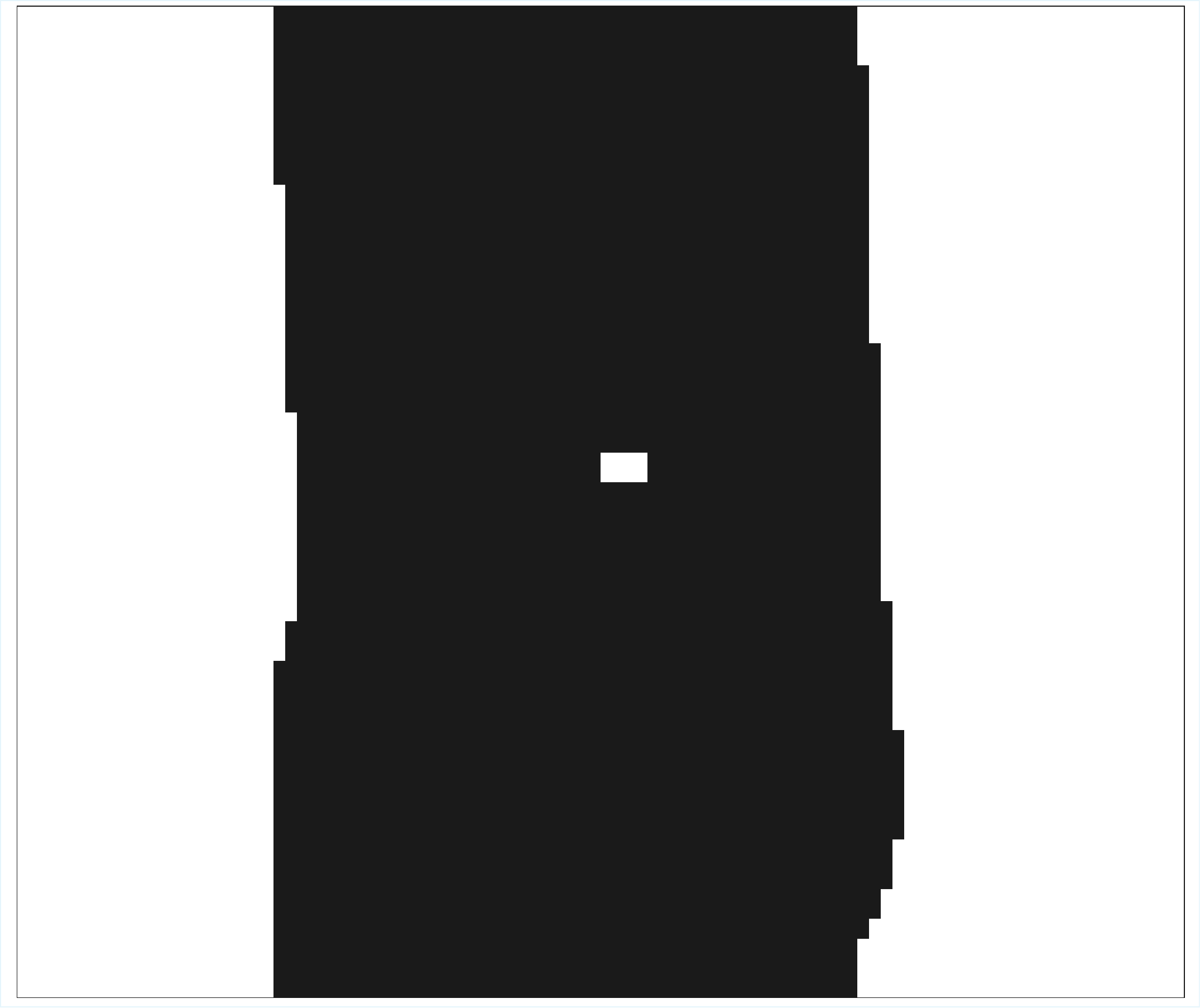}
\includegraphics[width=0.12\linewidth]
{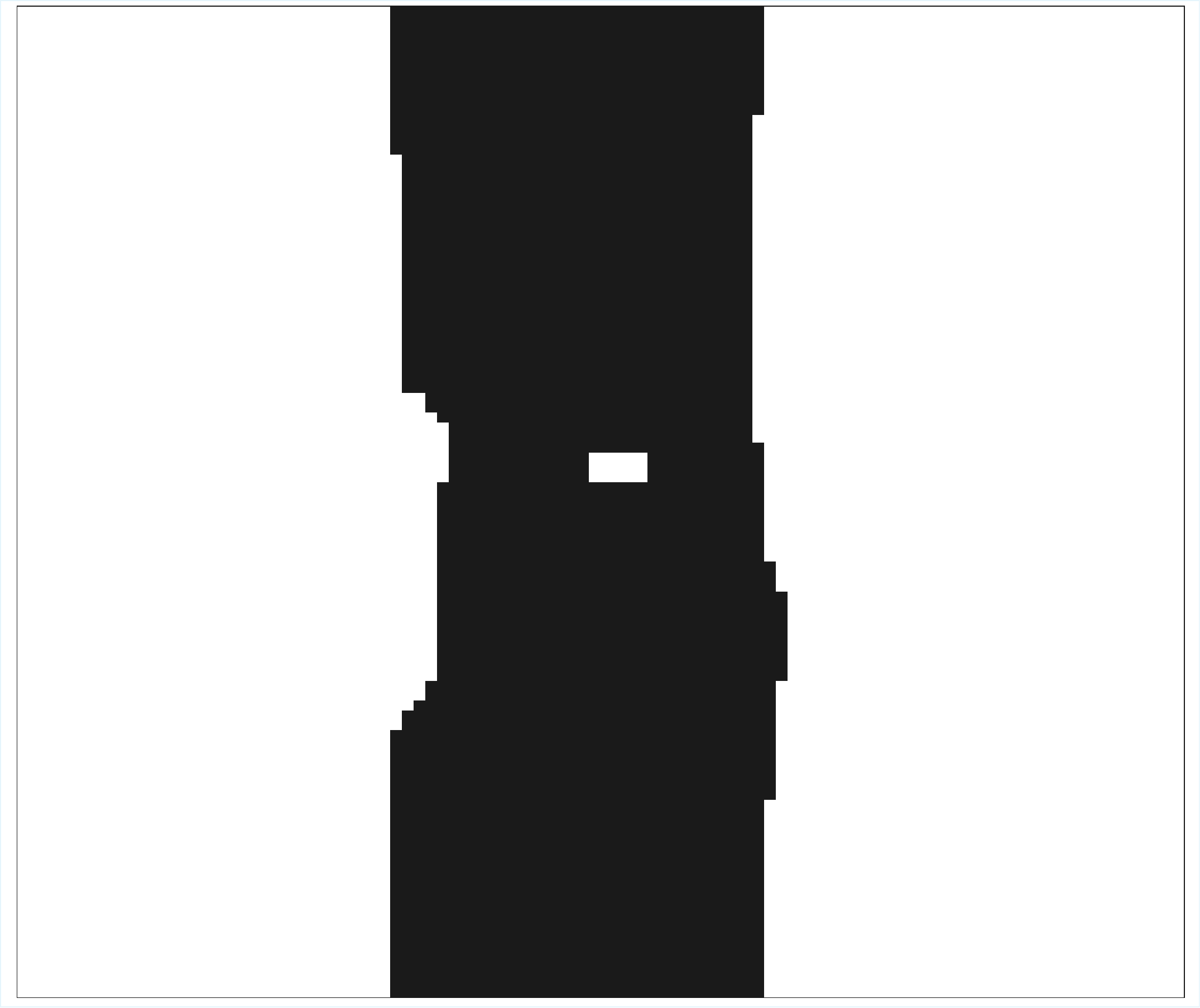}
\\
\vspace{1mm}
\includegraphics[width=0.12\linewidth]
{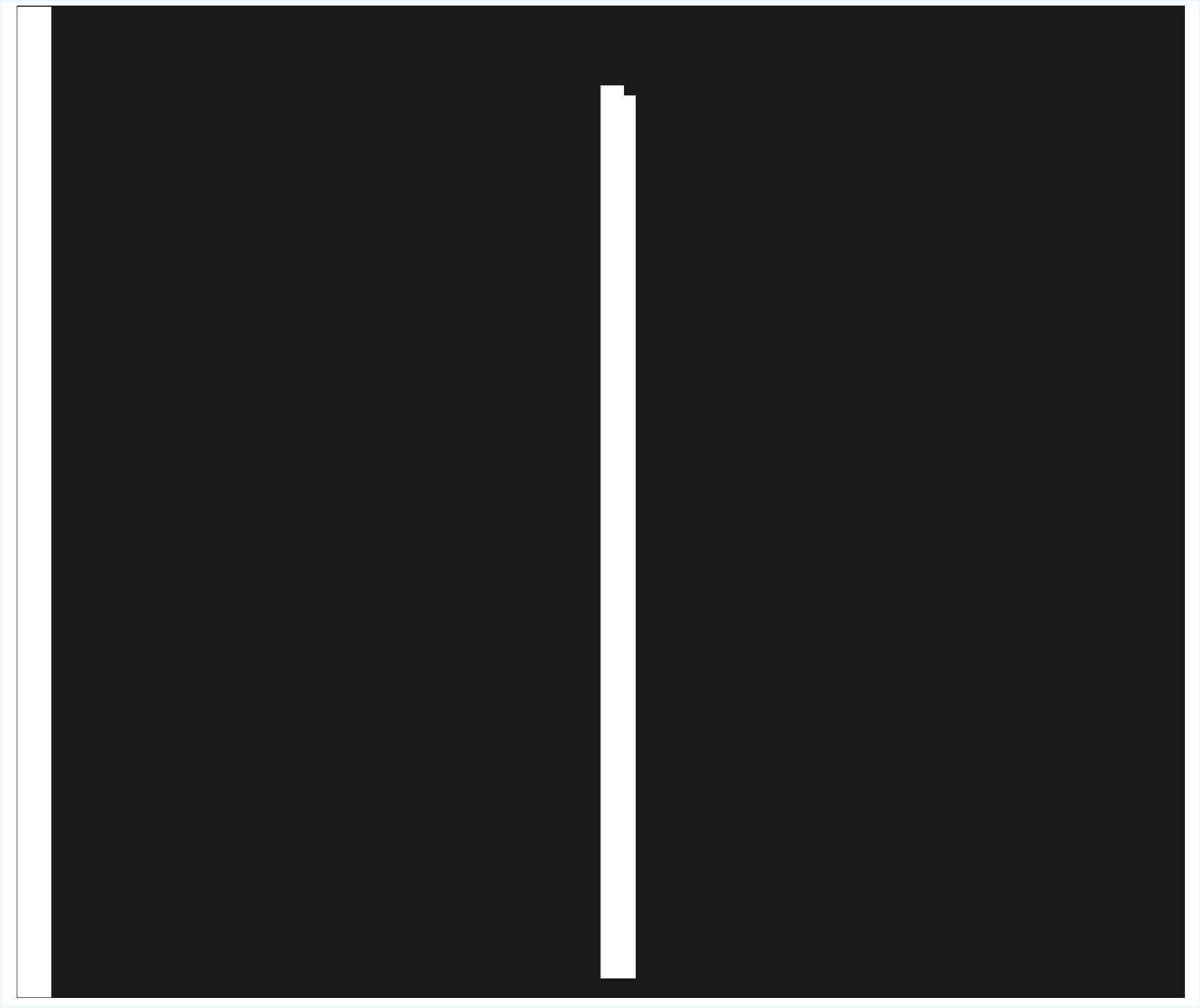}
\includegraphics[width=0.12\linewidth]
{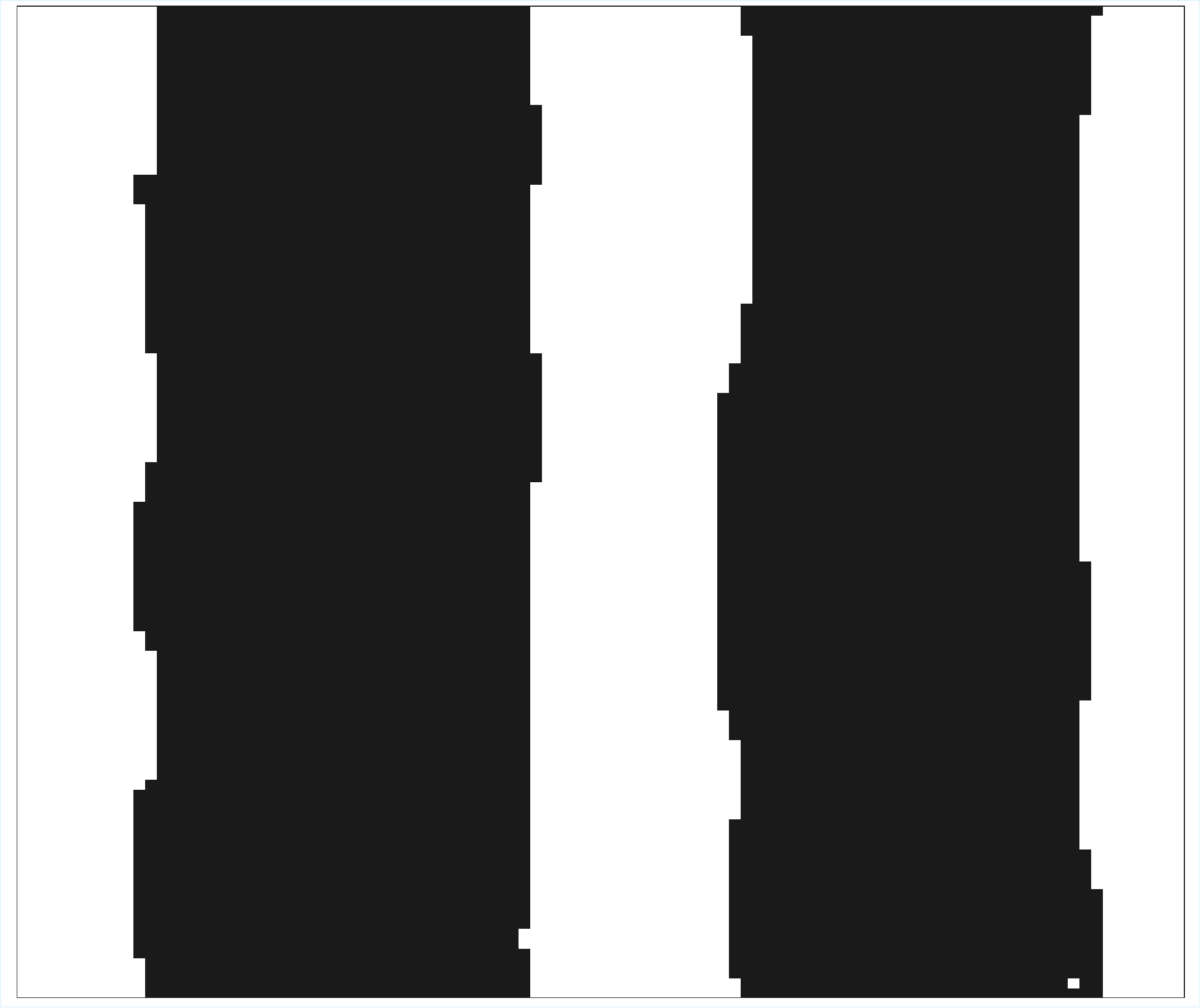}
\includegraphics[width=0.12\linewidth]
{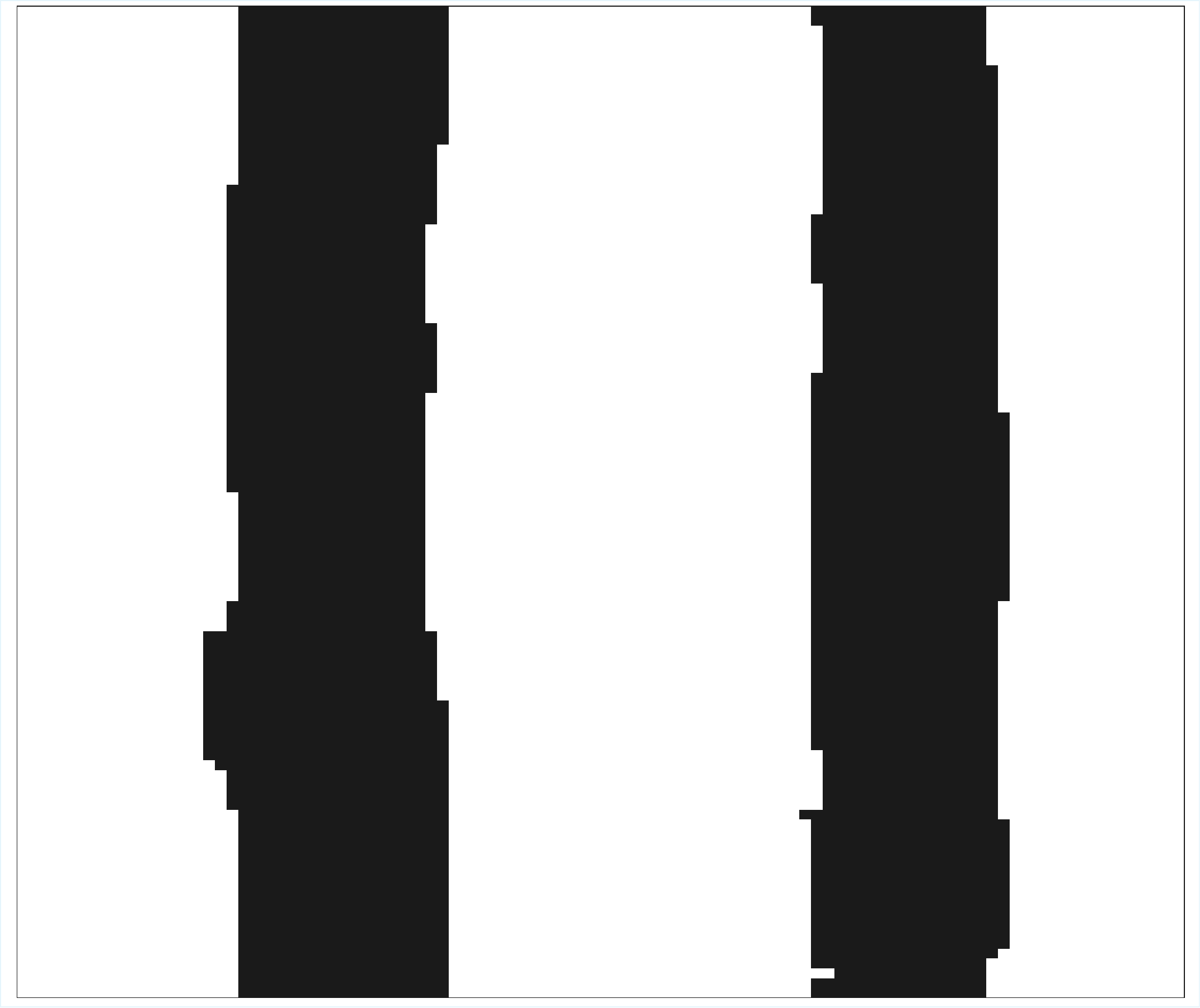}
\hspace{1.cm}
\includegraphics[width=0.12\linewidth]
{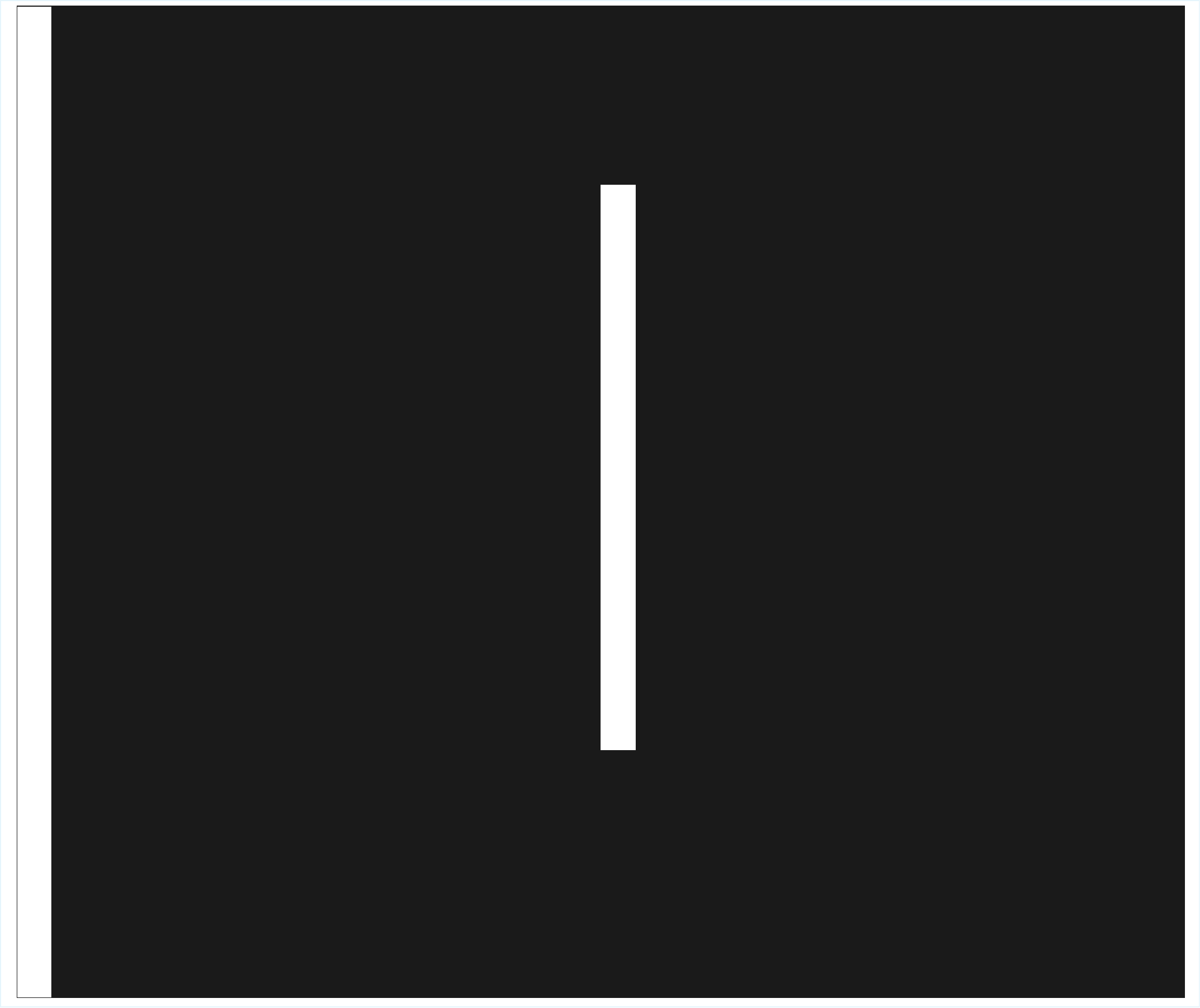}
\includegraphics[width=0.12\linewidth]
{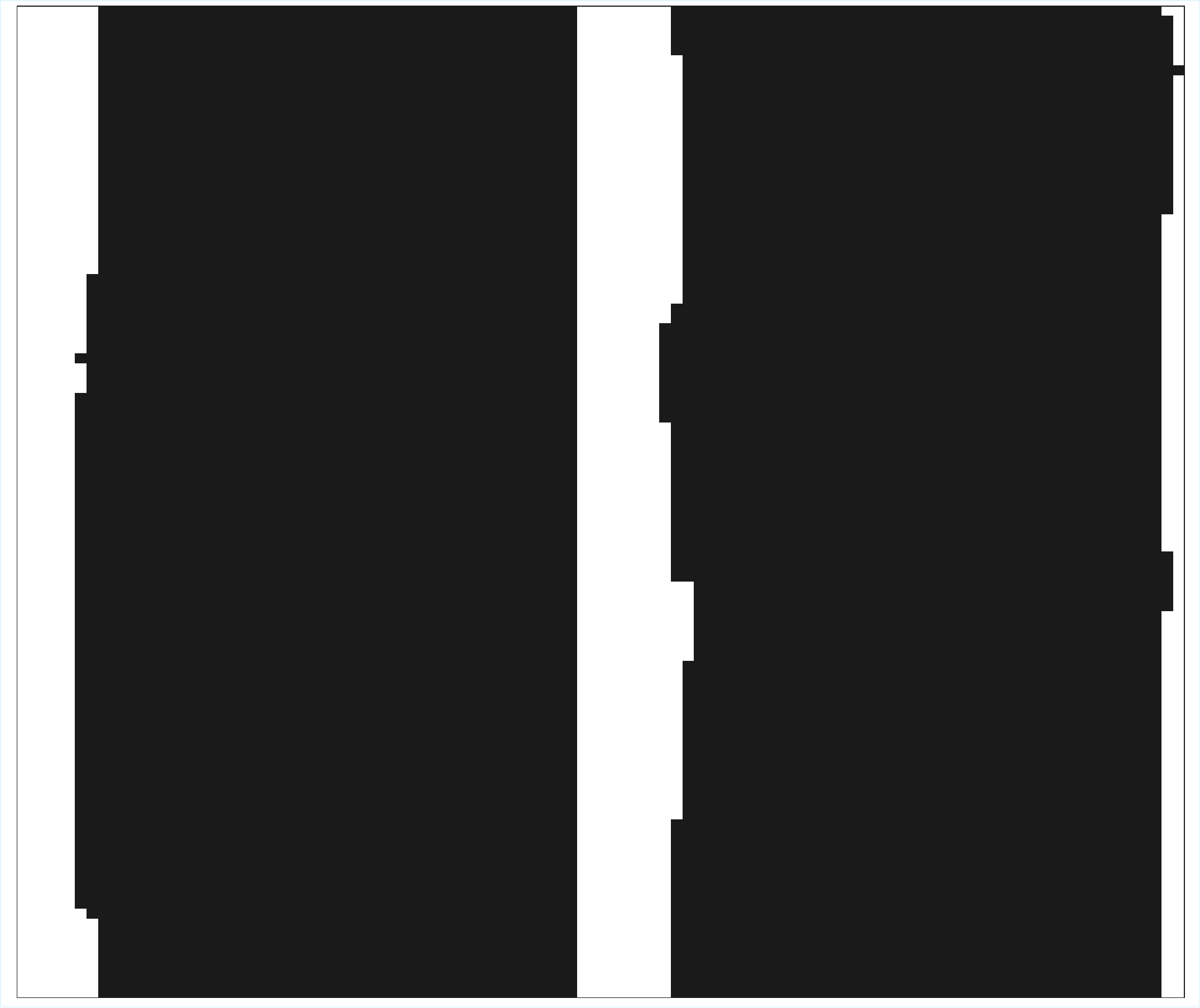}
\includegraphics[width=0.12\linewidth]
{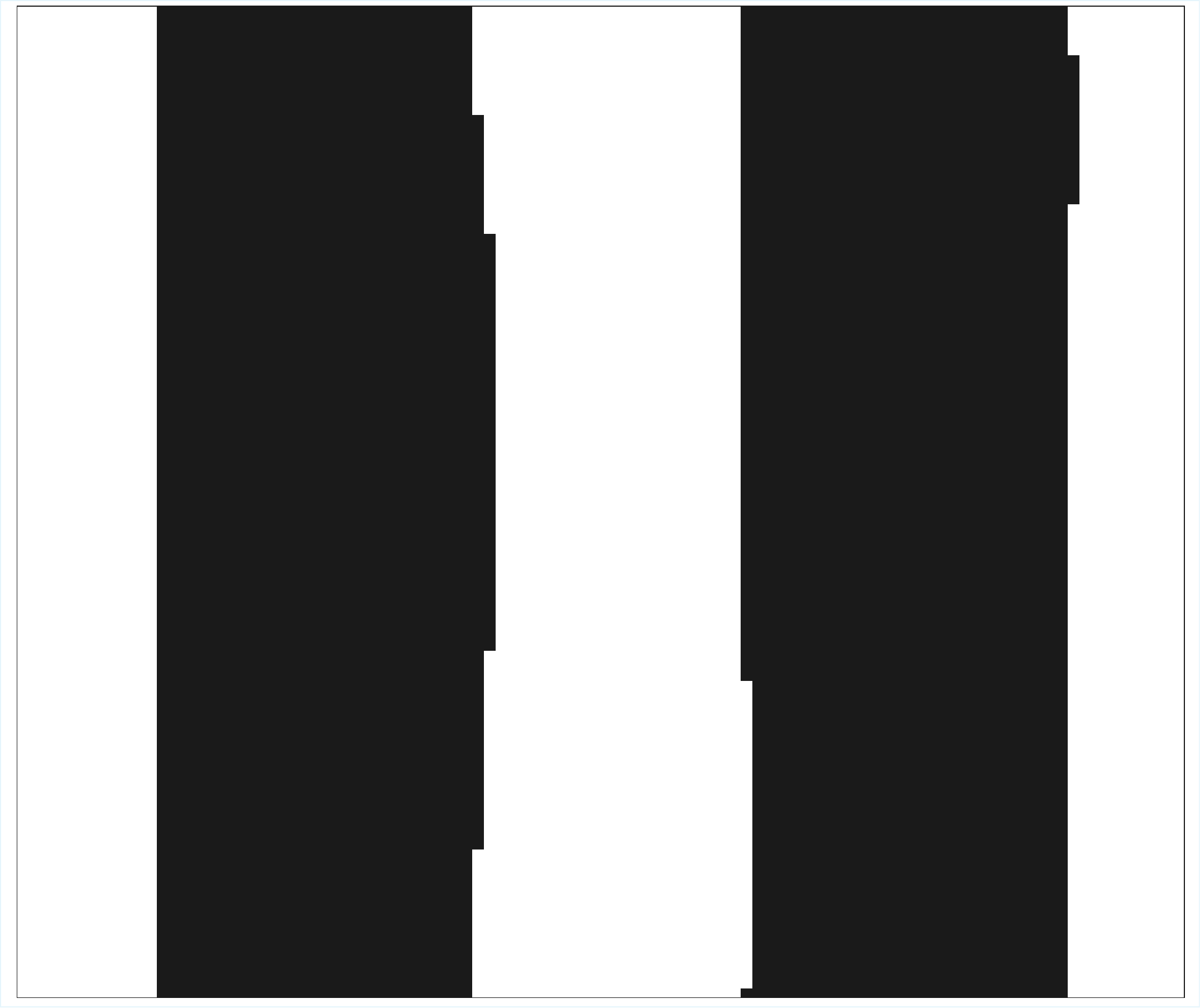}
\\
\vspace{1mm}
\includegraphics[width=0.12\linewidth]
{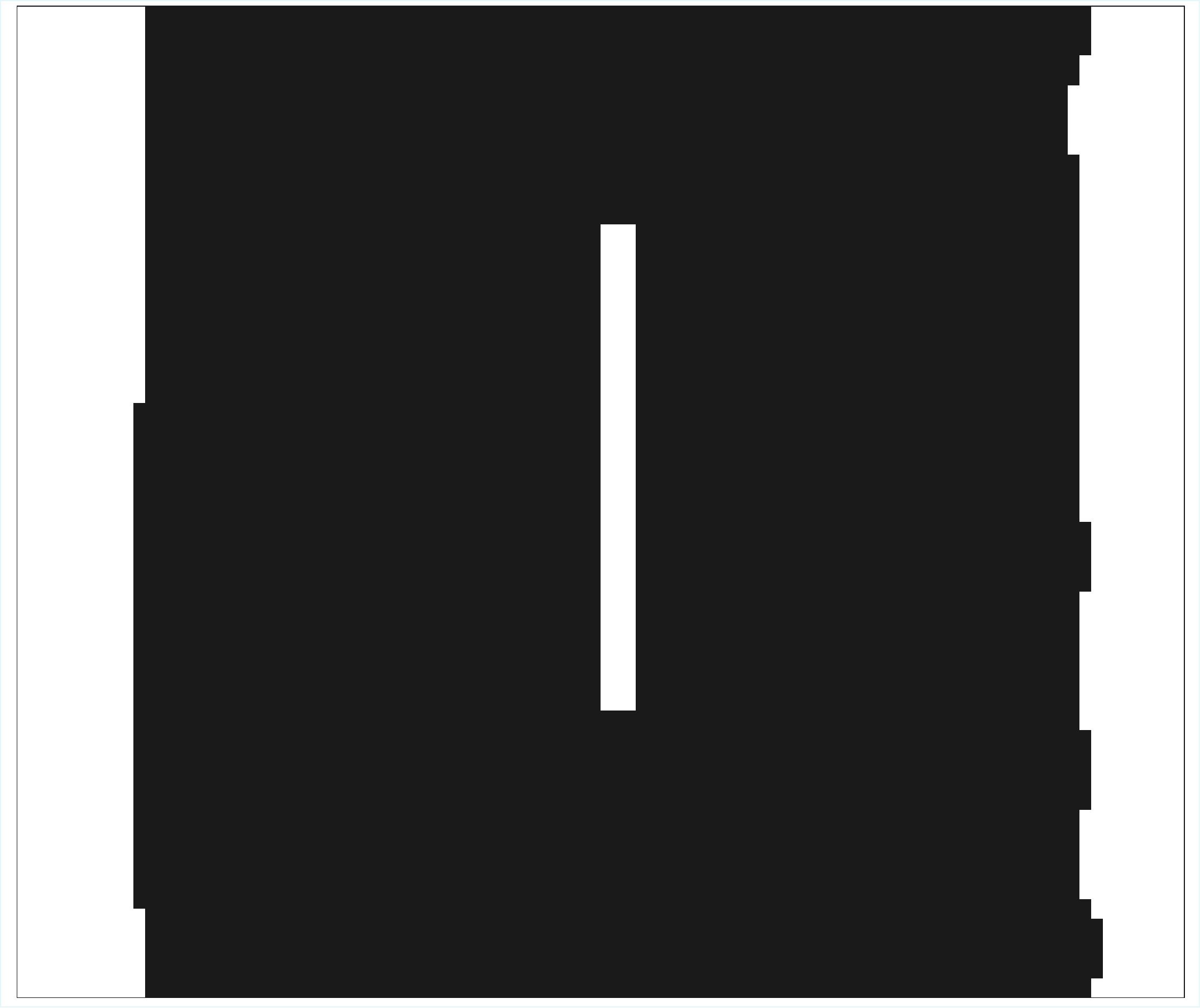}
\includegraphics[width=0.12\linewidth]
{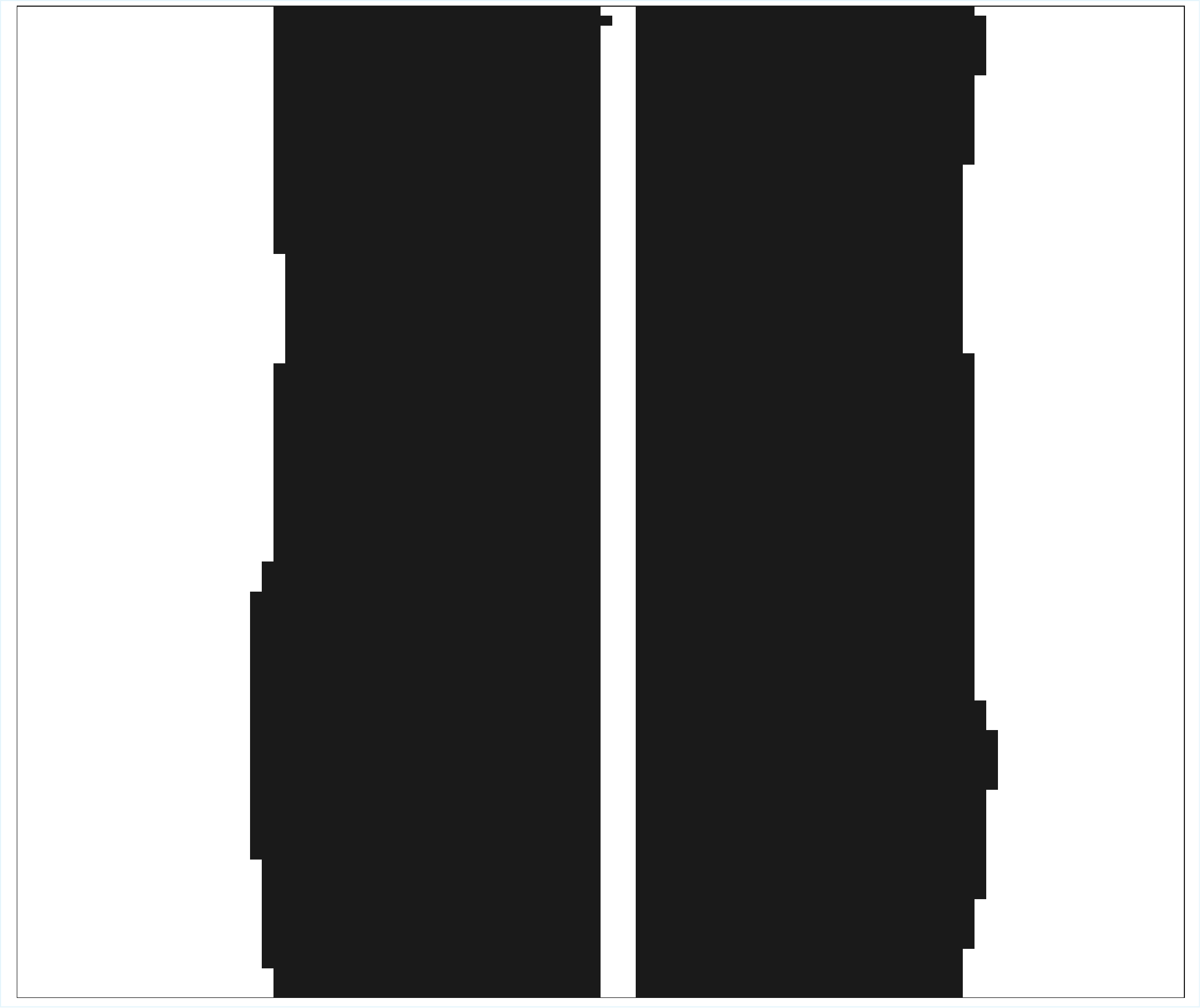}
\includegraphics[width=0.12\linewidth]
{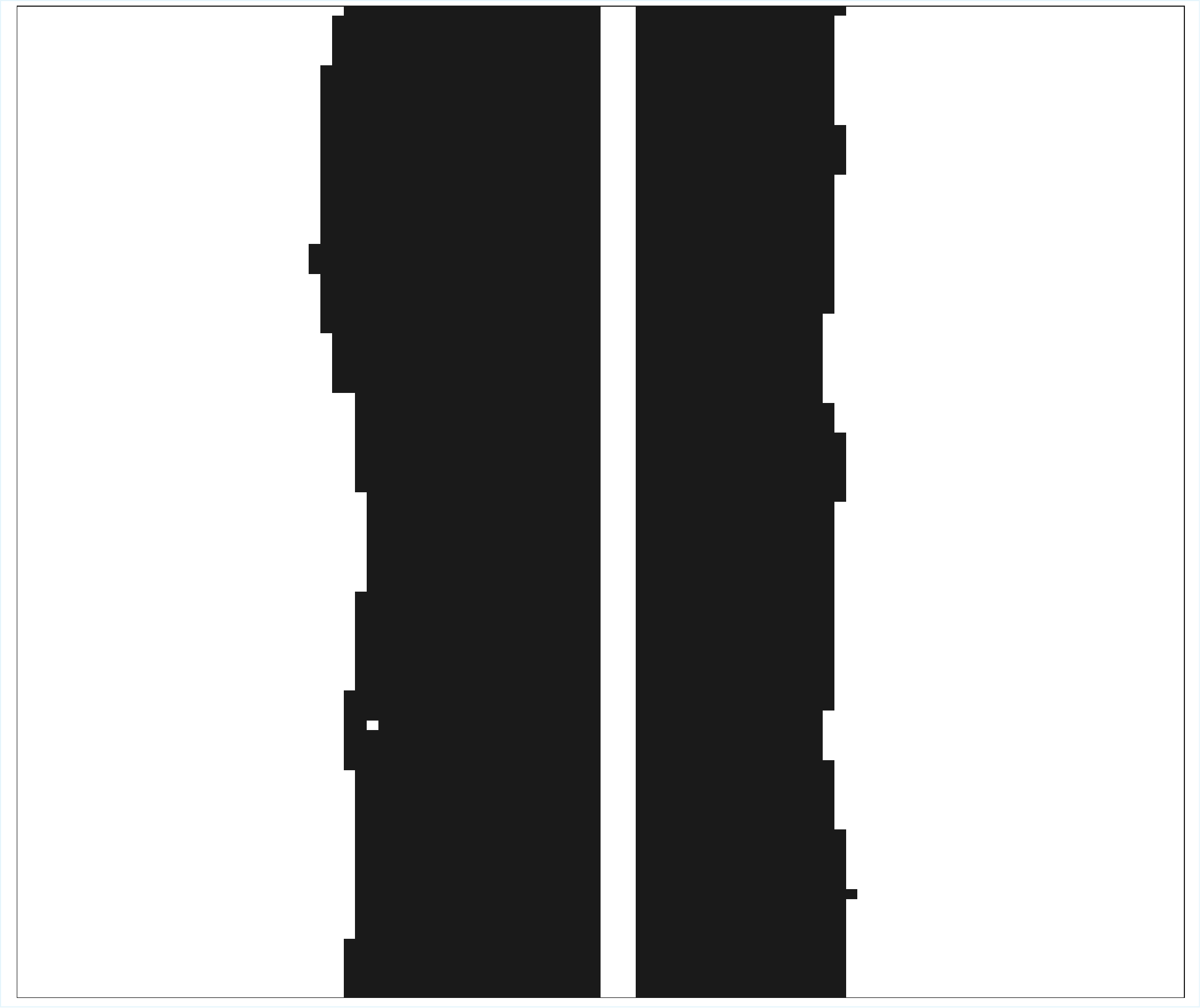}
\hspace{1.cm}
\includegraphics[width=0.12\linewidth]
{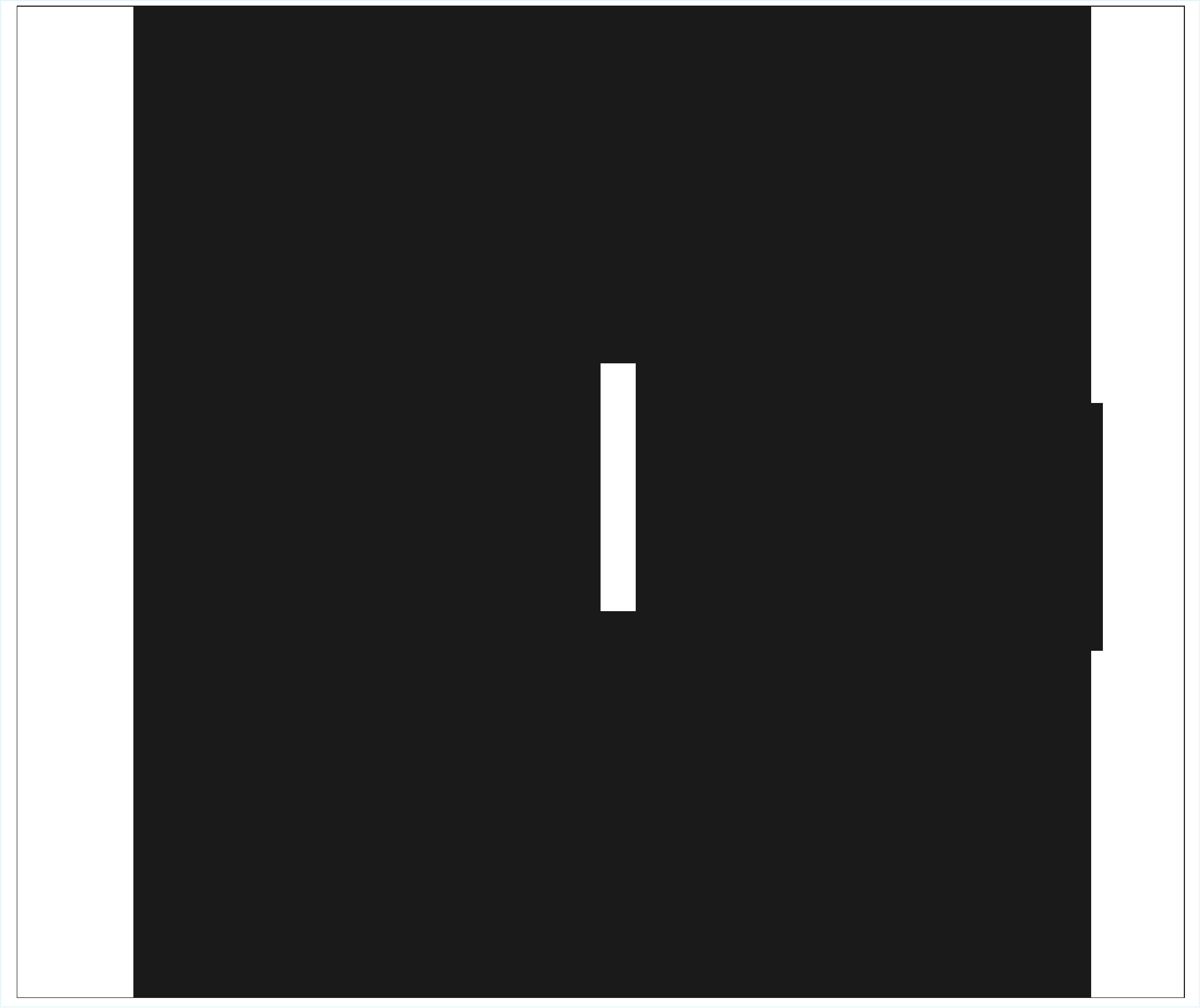}
\includegraphics[width=0.12\linewidth]
{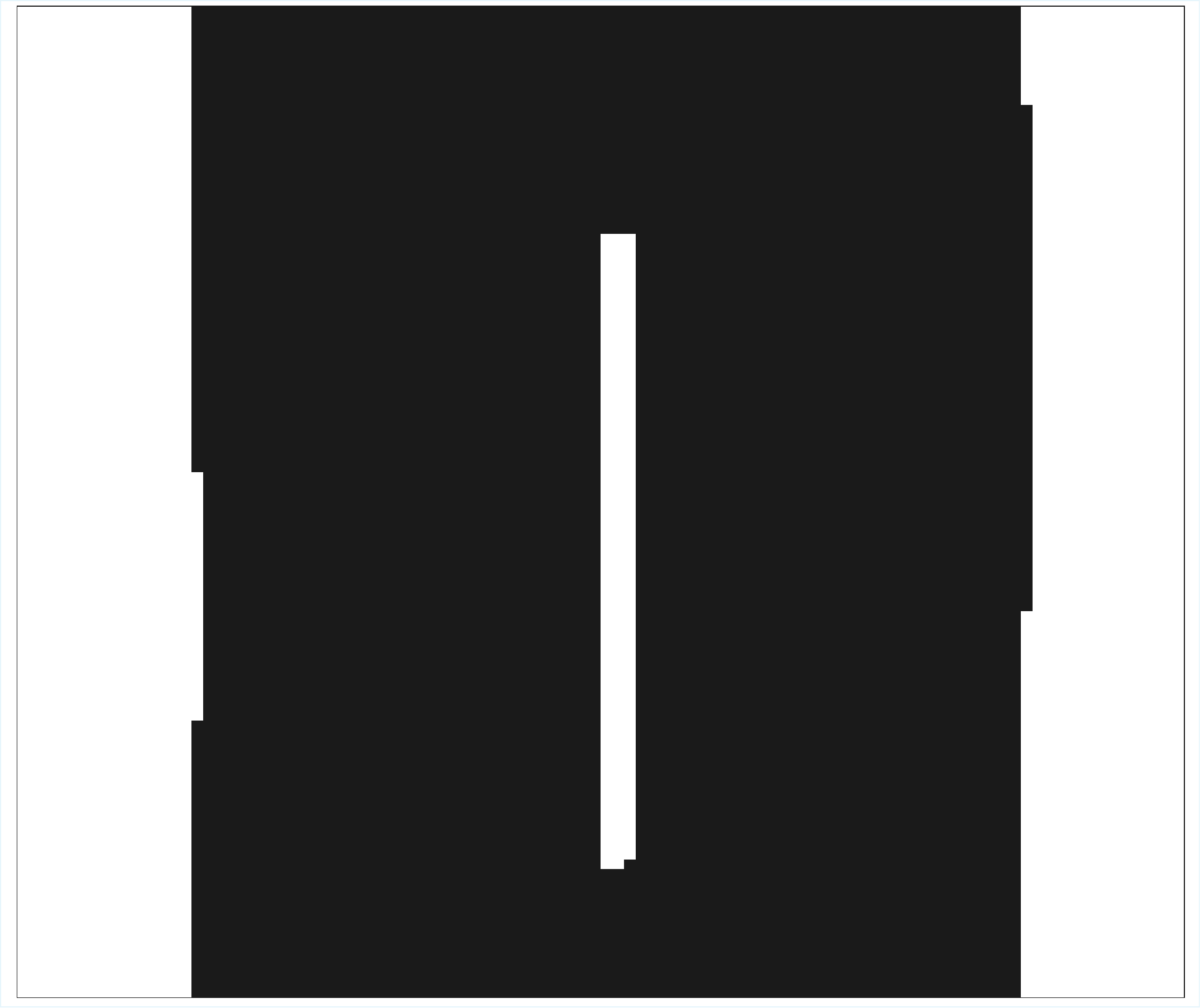}
\includegraphics[width=0.12\linewidth]
{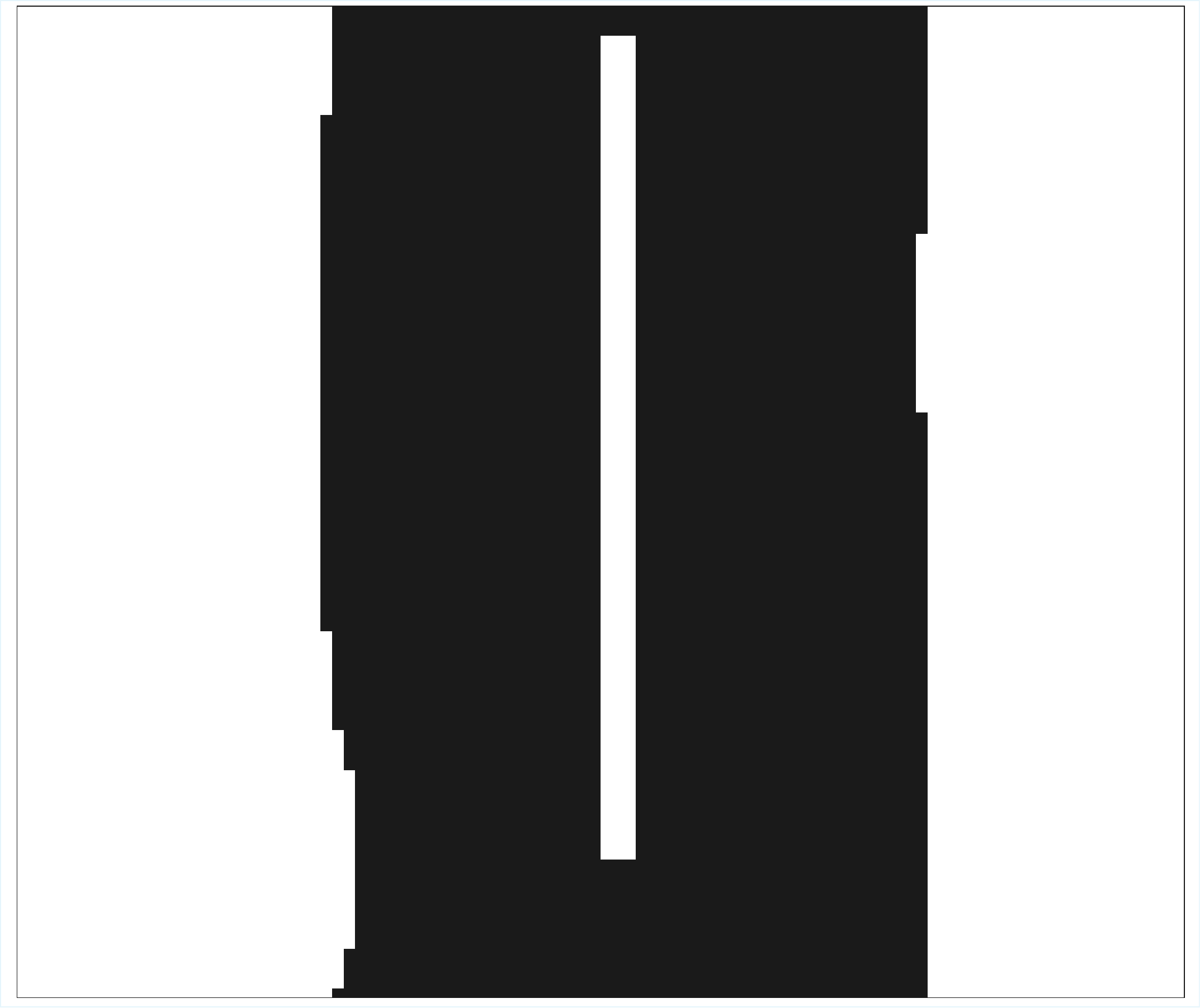}
\\
\vspace{1mm}
\includegraphics[width=0.12\linewidth]
{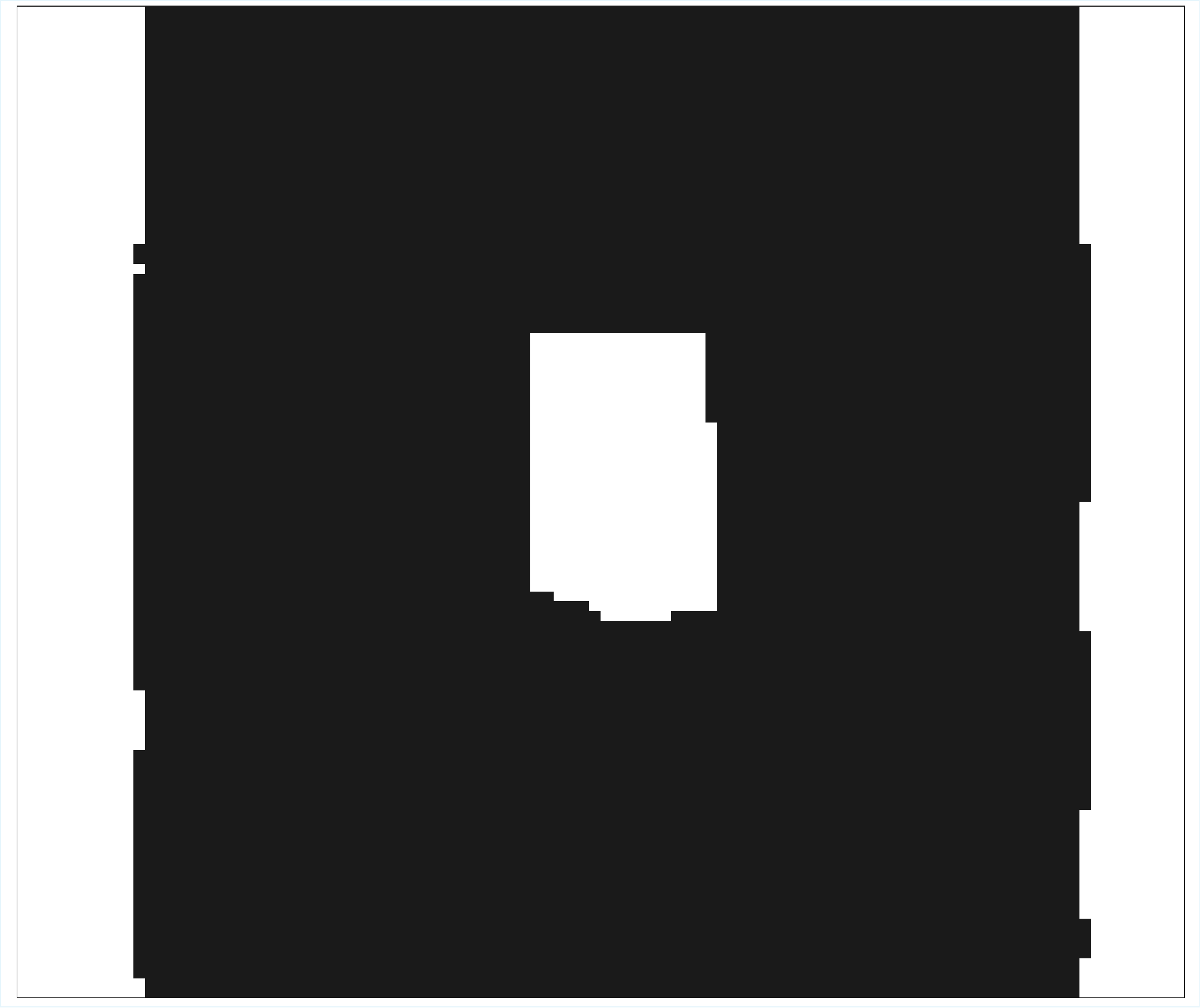}
\includegraphics[width=0.12\linewidth]
{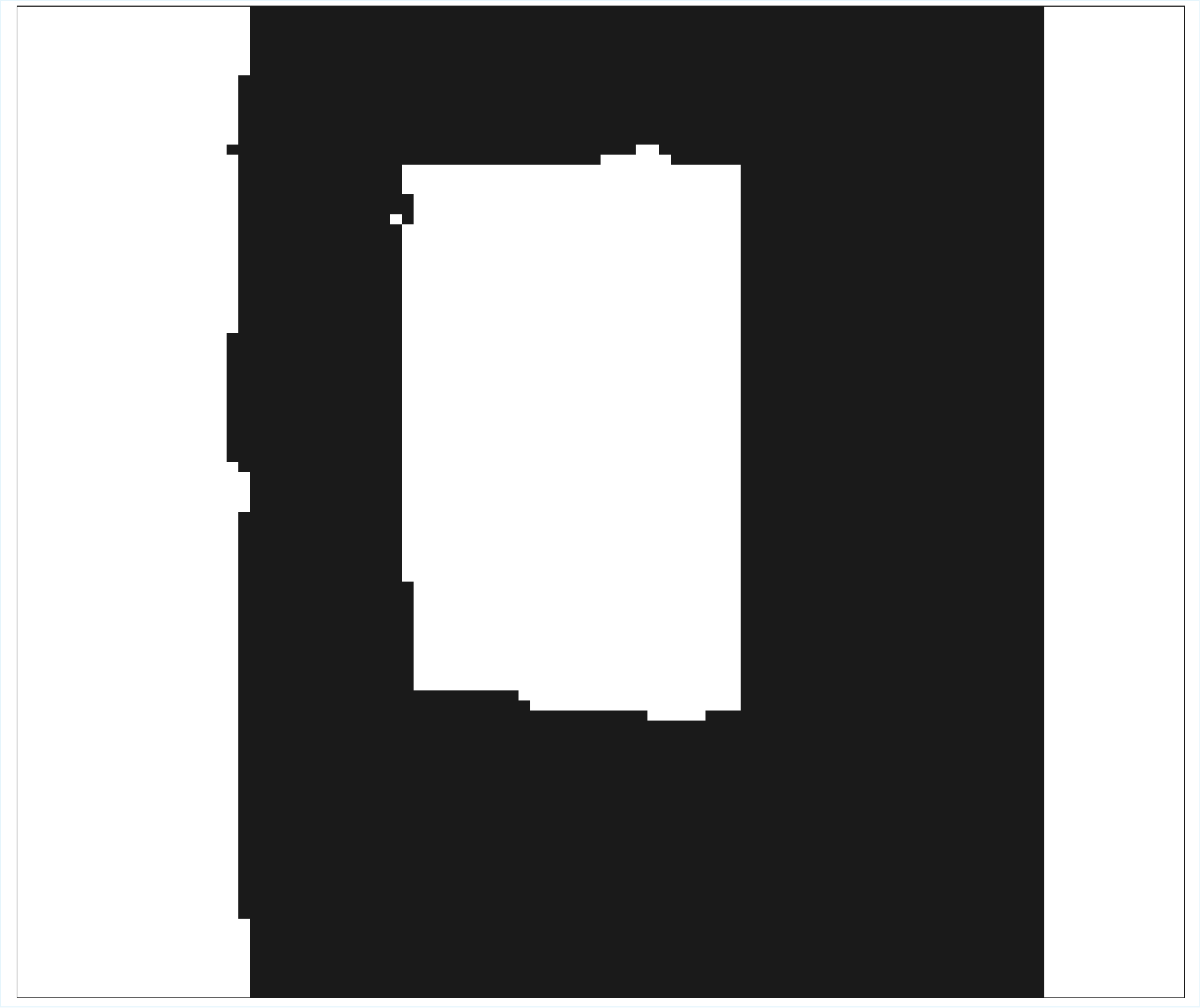}
\includegraphics[width=0.12\linewidth]
{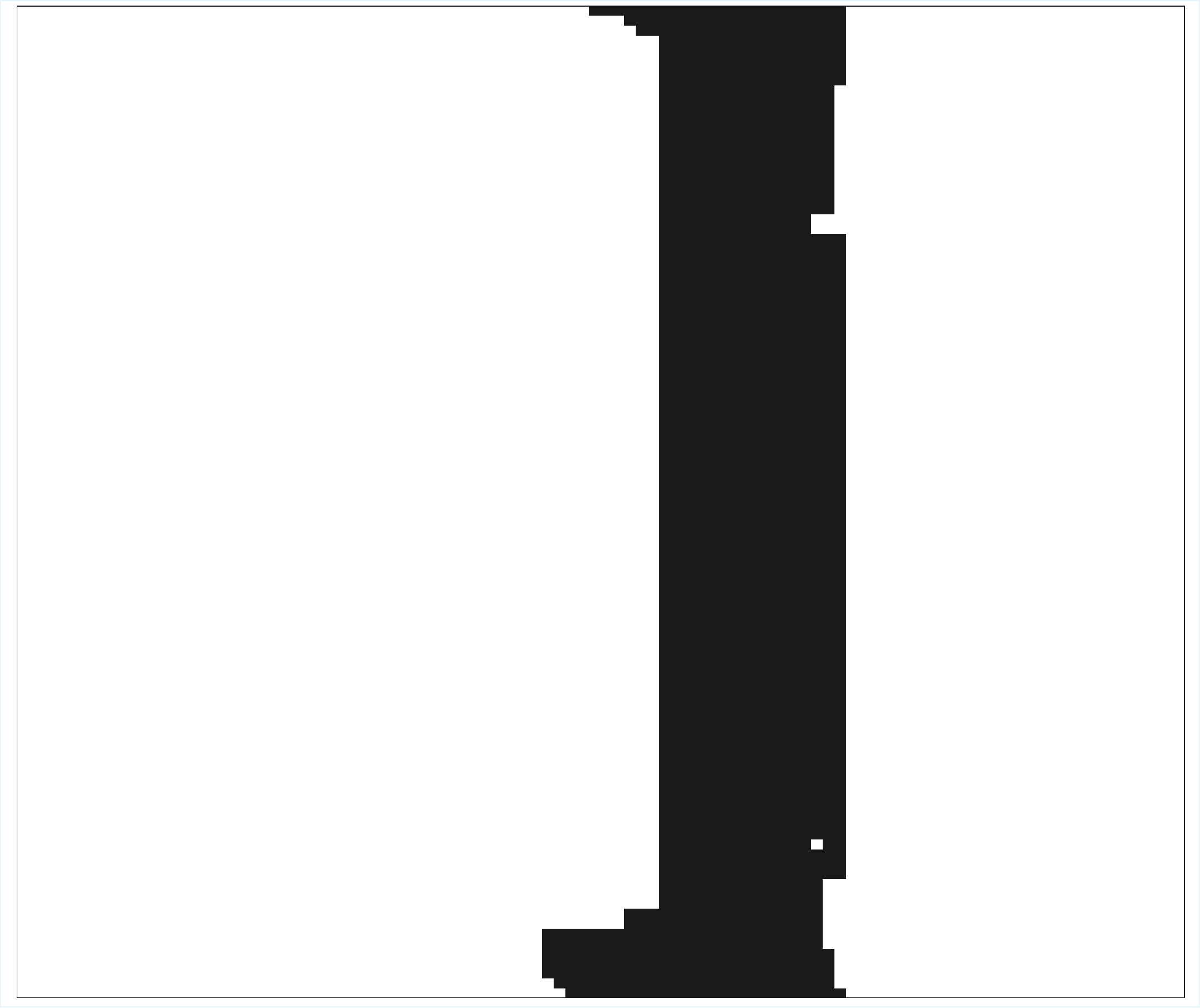}
\hspace{1.cm}
\includegraphics[width=0.12\linewidth]
{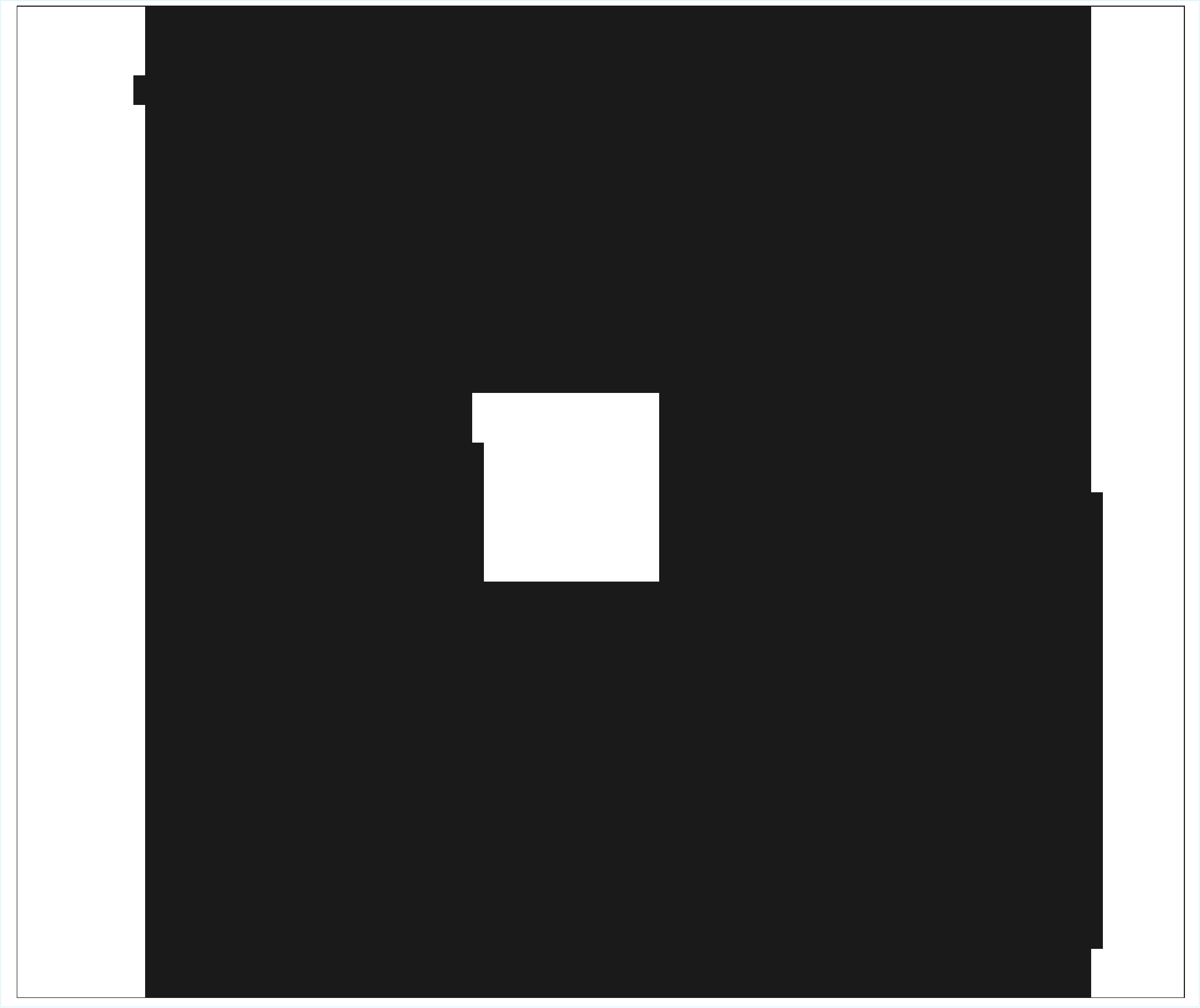}
\includegraphics[width=0.12\linewidth]
{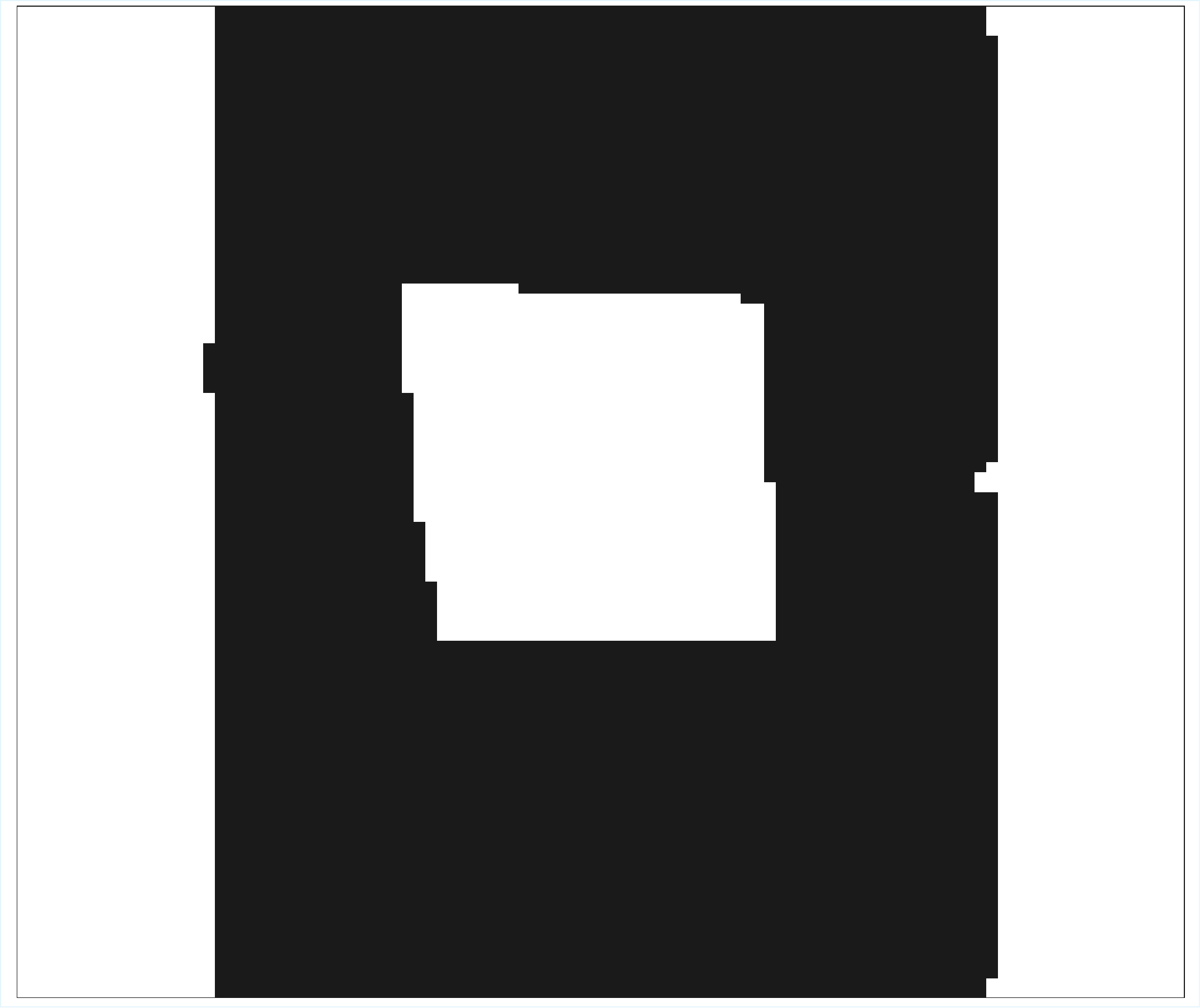}
\includegraphics[width=0.12\linewidth]
{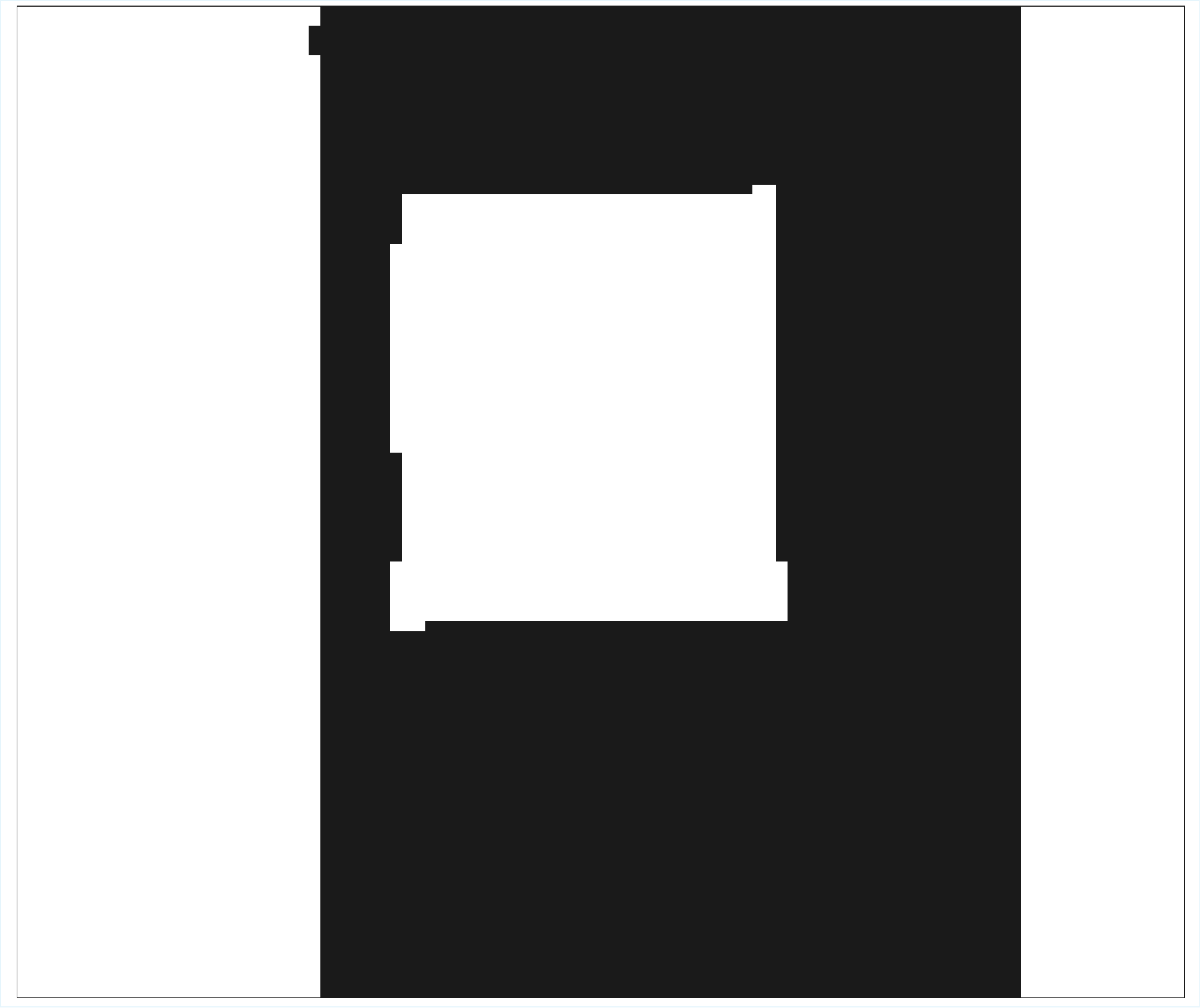}
\caption{Illustration of policies $\pi_1$ (first row), $\pi_2$ (second row), $\pi_3$ (third row) and $\pi_4$ (fourth row) for $N=100$ with initial seeds one stripe of width
$3$ and one 3x3 droplet at distance $47$.
Left group: $\kappa = 5,000$ 
at times $t = 150,250, 350$ (from left to right).
Right group:
$\kappa = 20,000$ at times 
$t = 100,150, 200$ (from left to right).
}
\label{Strip_droplet_simulations}
\end{figure}

Figure \ref{Strip_droplet_policy_comparison} displays the average values of the policies $\pi_1$, $\pi_2$, $\pi_3$ and $\pi_4$ on linear scale (left) and on log-scale (right), starting from a configuration with a stripe of width 3 and a 3x3 droplet at distance 13. The left panel also provides standard deviation error bars for the expected value of policy $\pi_1$. The results clearly indicate that policy $\pi_1$, in which the stripe and the droplet grow towards each other in a horizontal way, achieves the best performance in terms of the expected total discounted reward. 

To gain more insight into the characteristics of the optimal control strategy, we also consider an analogue of policy $\pi_1$, in which the decision maker flips spins at distance 2 rather than at distance 1 from either the stripe or the droplet. More precisely, we define a policy $\pi'_1 = (d_1')^{\infty}$, where $d_1'(i,j, k) \in A'_1(i,j, k)$ with $A_1': S^y \rightarrow P(A^y)$, as follows:
\begin{equation*}
    A_1'(i,j,k) = \begin{cases}
        \{a_{s1}, a_{\ell 1}\}, &\text{if } i = j = 2, \\
        \{a_{s1}, a_{\ell 2}\}, &\text{if } i > 2, j = 2, \text{ or } i = 2, j > 2,\\
        \{a_{s2}, a_{\ell 2}\}, &\text{if } i,j > 2.
    \end{cases}
\end{equation*}
The right panel of Figure \ref{Strip_droplet_policy_comparison} compares the performance of policies $\pi_1$ and $\pi_1'$, in which the stripe and the droplet grow horizontally, either through flipping spins at distance 1 or distance~2. The figure shows that the two policies achieve very similar performance and are not statistically distinguishable. Policy $\pi_1'$, which flips spins at distance 2, exhibits a higher standard deviation.

In addition to the value function, we use the data obtained from the simulations to measure the average value of the first hitting time to the all-plus configuration. The results, including 95\%-confidence intervals, are provided in Table \ref{Hitting_time_results_stripe_droplet}.

\begin{table}[h!]
\caption{Average values of first hitting times with 95\%-confidence intervals for policies $\pi_1$, $\pi_1'$, $\pi_2$, $\pi_3$ and $\pi_4$ for $N = 32$ and $\kappa = 100,000$.}
\centering
\begin{tabular}{l | l | l}
\hline
\textbf{Policy} & \textbf{Mean first hitting time} & \textbf{95\%-confidence interval} \\
\hline
$\pi_1$ & 35.816 & (35.636, 35.996) \\
$\pi_1'$ & 36.884 & (36.570, 37.198) \\
$\pi_2$ & 77.587 & (77.321, 77.853) \\
$\pi_3$ & 67.636 & (67.217, 68.056) \\
$\pi_4$ & 60.165 & (59.412, 60.918) \\
\hline
\end{tabular}

\label{Hitting_time_results_stripe_droplet}
\end{table}

The table indicates that $\pi_1$ is the optimal policy for minimizing the expected first hitting time among the candidate policies. Its counterpart $\pi_1'$, which flips spins at distance 2, is, however, not far behind.

\begin{figure}
\centering
\includegraphics[width=0.45\linewidth]
{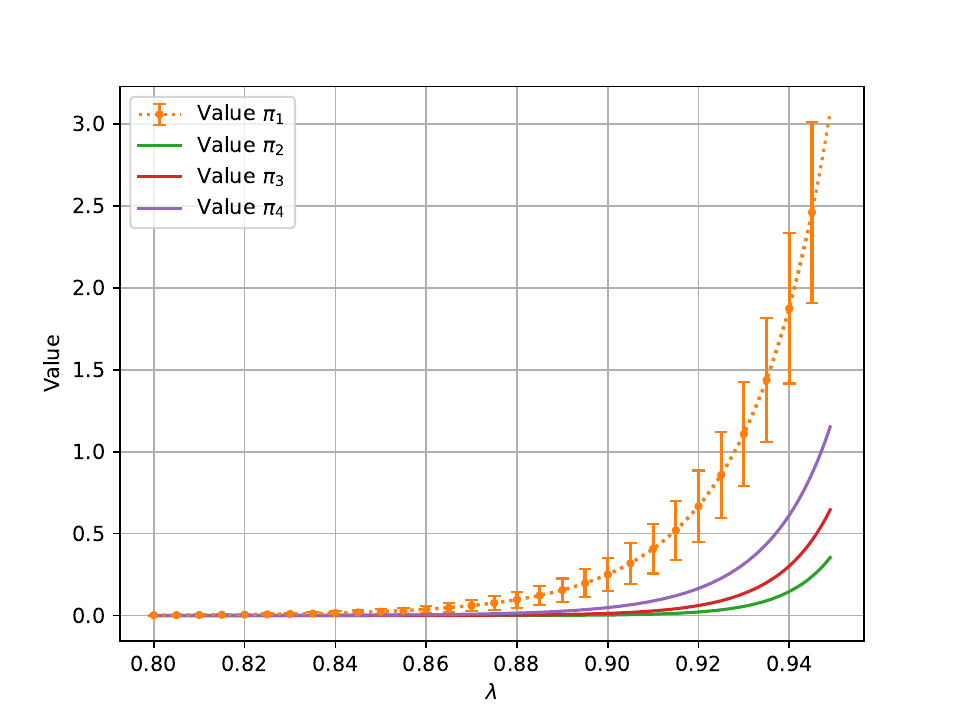}\hspace{2mm}
\raisebox{-0.5mm}{
\includegraphics[width=0.45\linewidth]
{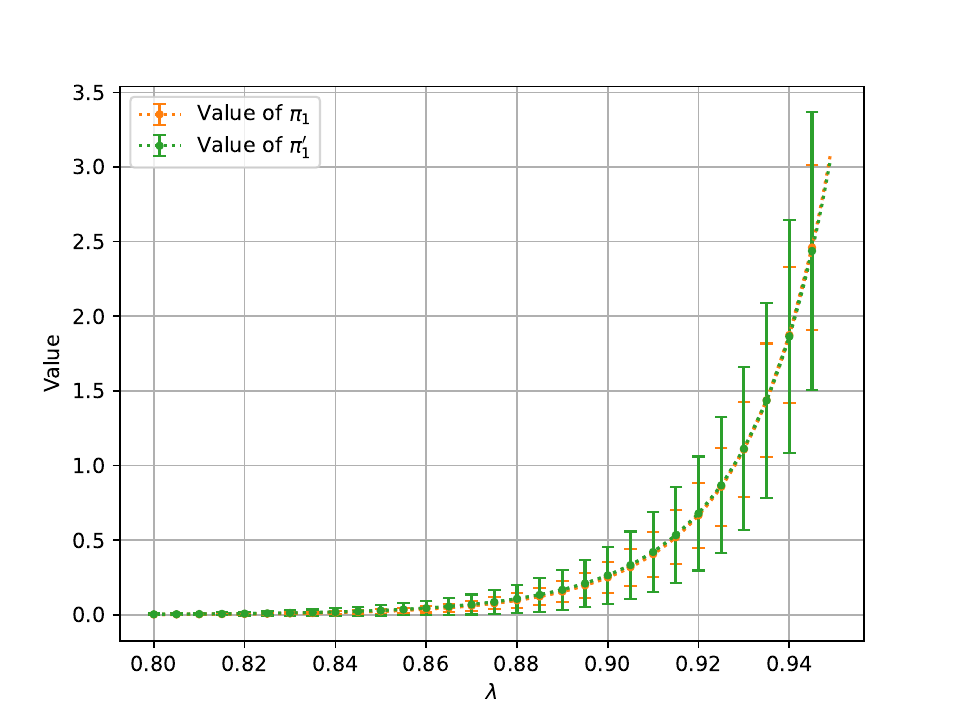}
}
\caption{Average values of policies $\pi_1$, $\pi_2$, $\pi_3$ and $\pi_4$ (left) and average values of policies $\pi_1$ and $\pi_1'$ (right) for $N = 32$ and $\kappa = 100,000$, starting from a configuration with a stripe of width 3 and a 3x3 droplet at distance 13.}
\label{Strip_droplet_policy_comparison}
\end{figure}

\subsection{The two-droplet case}
Finally, we consider configurations in which neither of the two droplets forms a stripe. Let the set of these configurations be denoted by $U^{2,z}$. 

\subsubsection{The auxiliary MDP}
\label{s:dd-aux}
We define an auxiliary MDP $(S^z, A^z, P^z, r^z)$ with state space and action space, respectively, 
\begin{equation*}
    S^z = \{(i, j, k, \ell, m, n)| i, j, k, \ell, m, n = 0, 2, 3, \ldots, N\}
\;\textup{ and }\;
    A^z = \{a_h, a_v, a_0\}.
\end{equation*}
Here, a state $(i, j, k, \ell, m, n)$ corresponds to the set of configurations in which the horizontal distances between the narrowest vertical stripes that circumscribe the droplets are $i$ and $j$, the vertical distances between the narrowest horizontal stripes that enclose the droplets are $k$ and $\ell$ and the vertical distances between the horizontal boundaries of the respective droplets are $m$ and $n$, as illustrated in Figure \ref{Aux_spaces_2drops} (left).

The actions $a_h$ and $a_v$ correspond to the sets of sites at horizontal and vertical distance 1 from either of the droplets. Action $a_0$ represents the set of sites that are diagonally adjacent to either of the droplets. By taking an action $a \in A^z$, we again mean flipping the spin at one of the sites in the corresponding set. A visualization of the action space $A^z$ is provided in Figure \ref{Aux_spaces_2drops} (right).

\begin{figure}
\centering
\includegraphics[width=0.4\linewidth]
{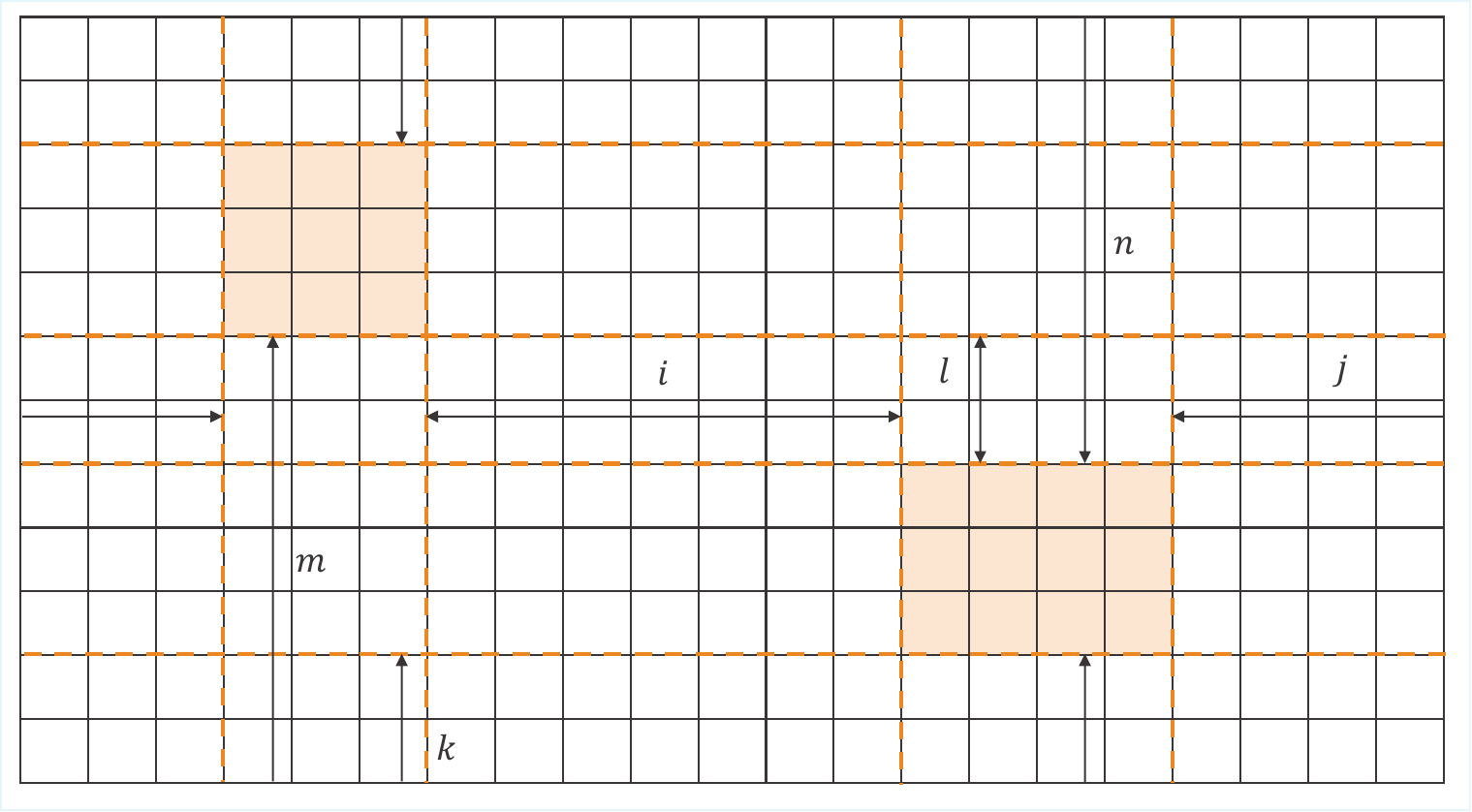}\hspace{2mm}
\raisebox{-0.5mm}{
\includegraphics[width=0.4\linewidth]
{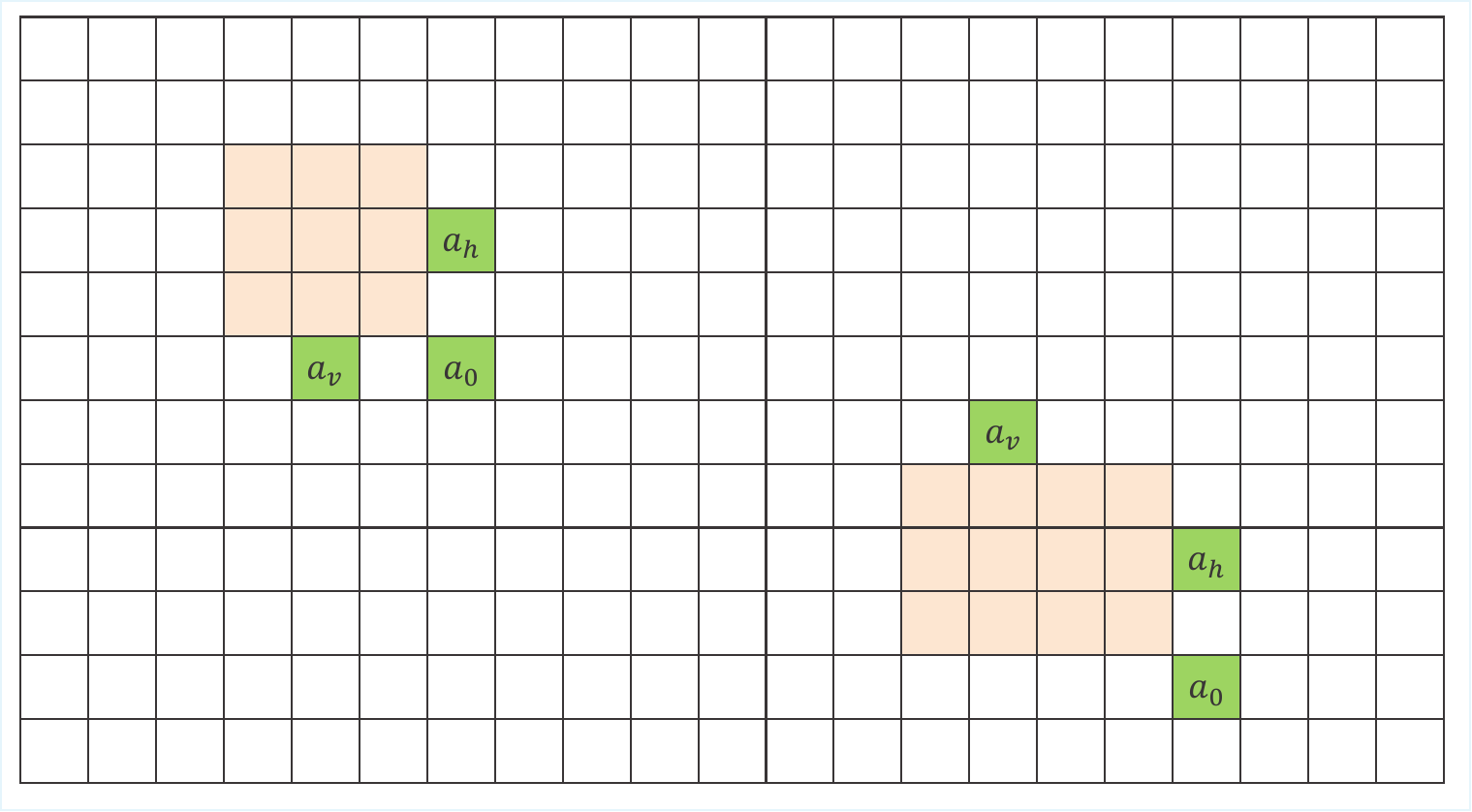}
}
\caption{Illustration of state space $S^z$ (left) and action space $A^z$ (right).}
\label{Aux_spaces_2drops}
\end{figure}

\subsubsection{Candidates for optimality}
\label{s:dd-opt}

We compare the performance of three heuristic policies $\pi_1 = (d_1)^{\infty}$, $\pi_2 = (d_2)^{\infty}$ and $\pi_3 = (d_3)^{\infty}$, defined in the auxiliary MDP, where $d_q(i,j,k, \ell, m, n) = A_q(i,j,k,\ell, m, n)$, $q = 1,2,3$. The mappings $A_q: S^z \rightarrow P(A^z)$, $q = 1, 2, 3$, are defined for states $(i,j,k,\ell, m, n) \in S^z$ as follows:
\begin{align*}
    A_1(i,j,k,\ell,m,n) &= \begin{cases}
        \{a_v\}, &\text{if } m > 0, \text{ or } n > 0, \\
        \{a_h\}, &\text{otherwise.}
    \end{cases}\\
    A_2(i,j,k,\ell,m,n) &= \begin{cases}
        \{a_v\}, &\text{if } k > 0, \text{ or } \ell > 0, \\
        \{a_h\}, &\text{otherwise.}
    \end{cases}\\
    A_3(i,j,k,\ell,m,n) &= \{a_0\}.
\end{align*}
The policies are illustrated in Figure \ref{policies_2drops}.

% \begin{figure}
% \centering
% \includegraphics[width=0.2\linewidth]
% {2drops_pi1_1.pdf}
% \hspace{2mm}
% \includegraphics[width=0.2\linewidth]
% {2drops_pi1_2.pdf}
% \\
% \vspace{2mm}
% \includegraphics[width=0.2\linewidth]
% {2drops_pi2_1.pdf}
% \hspace{2mm}
% \includegraphics[width=0.2\linewidth]
% {2drops_pi2_2.pdf}
% \\
% \vspace{2mm}
% \includegraphics[width=0.2\linewidth]
% {2drops_pi3_1.pdf} 
% \hspace{2mm}
% \includegraphics[width=0.2\linewidth]
% {2drops_pi3_2.pdf}
% \caption{Illustration of policies $\pi_1$ (first row), $\pi_2$ (second row) and $\pi_3$ (third row).
% }
% \label{policies_2drops}
% \end{figure}

\begin{figure}
\centering
\includegraphics[width=0.2\linewidth]
{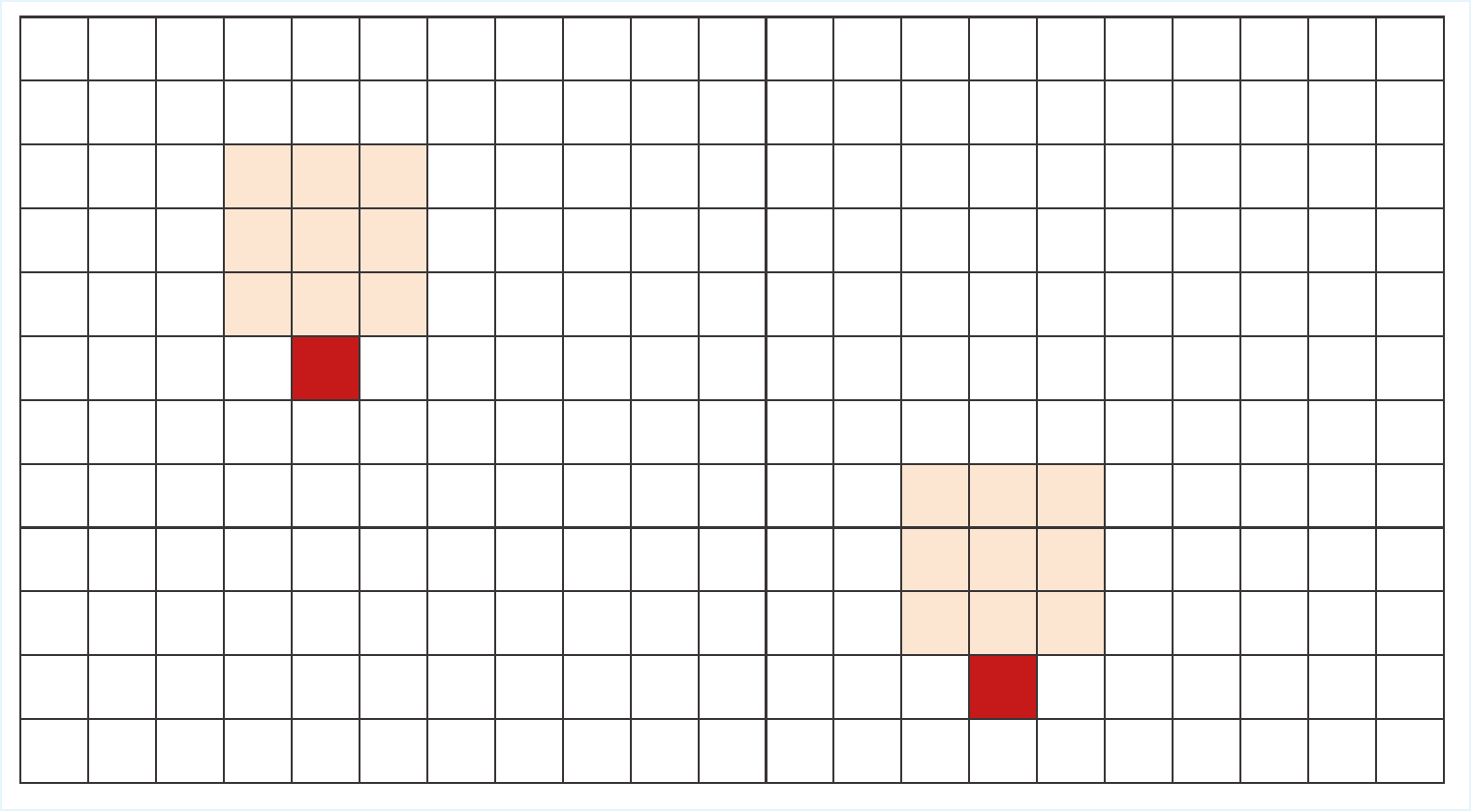}
\hspace{2mm}
\includegraphics[width=0.2\linewidth]
{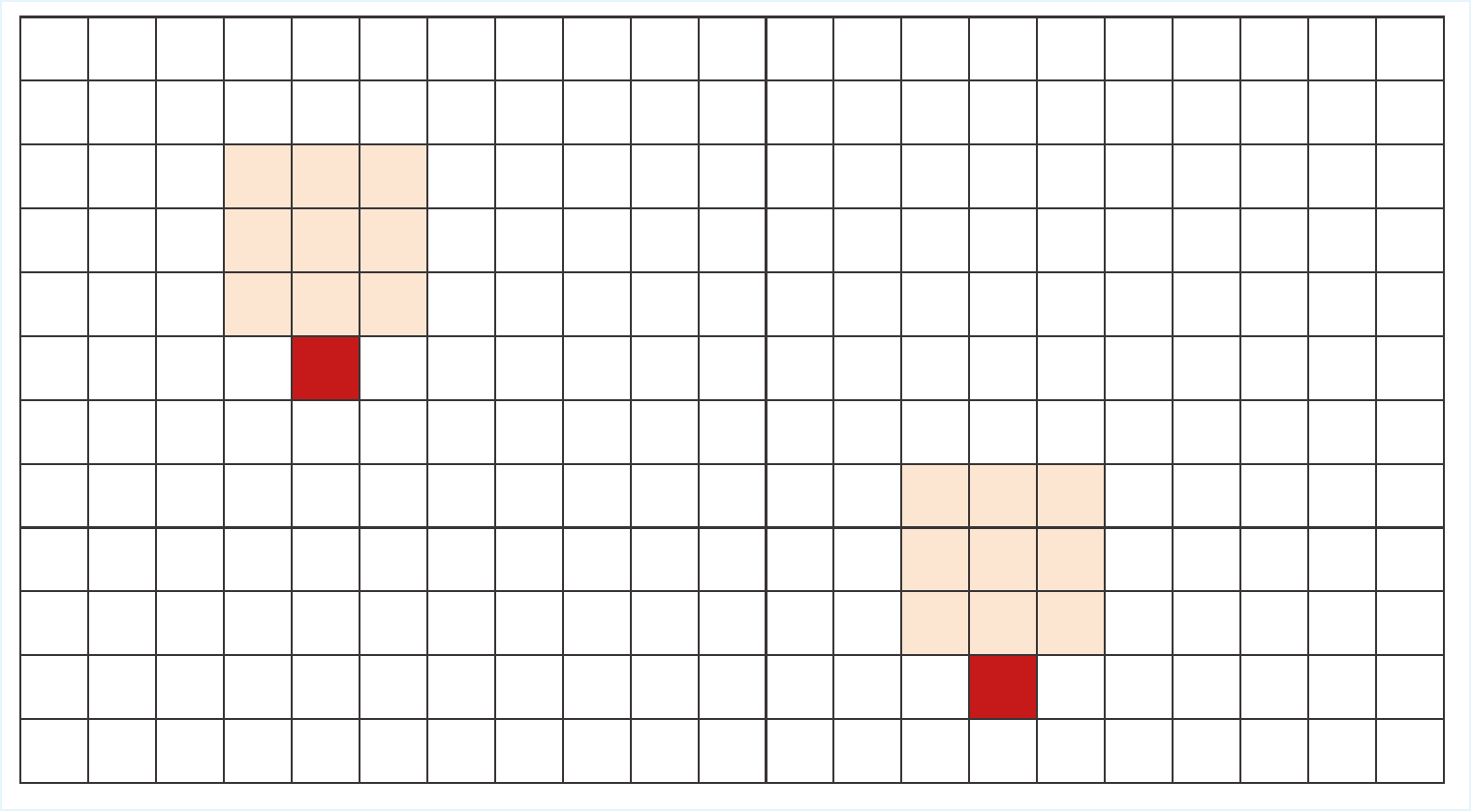}
\hspace{2mm}
\includegraphics[width=0.2\linewidth]
{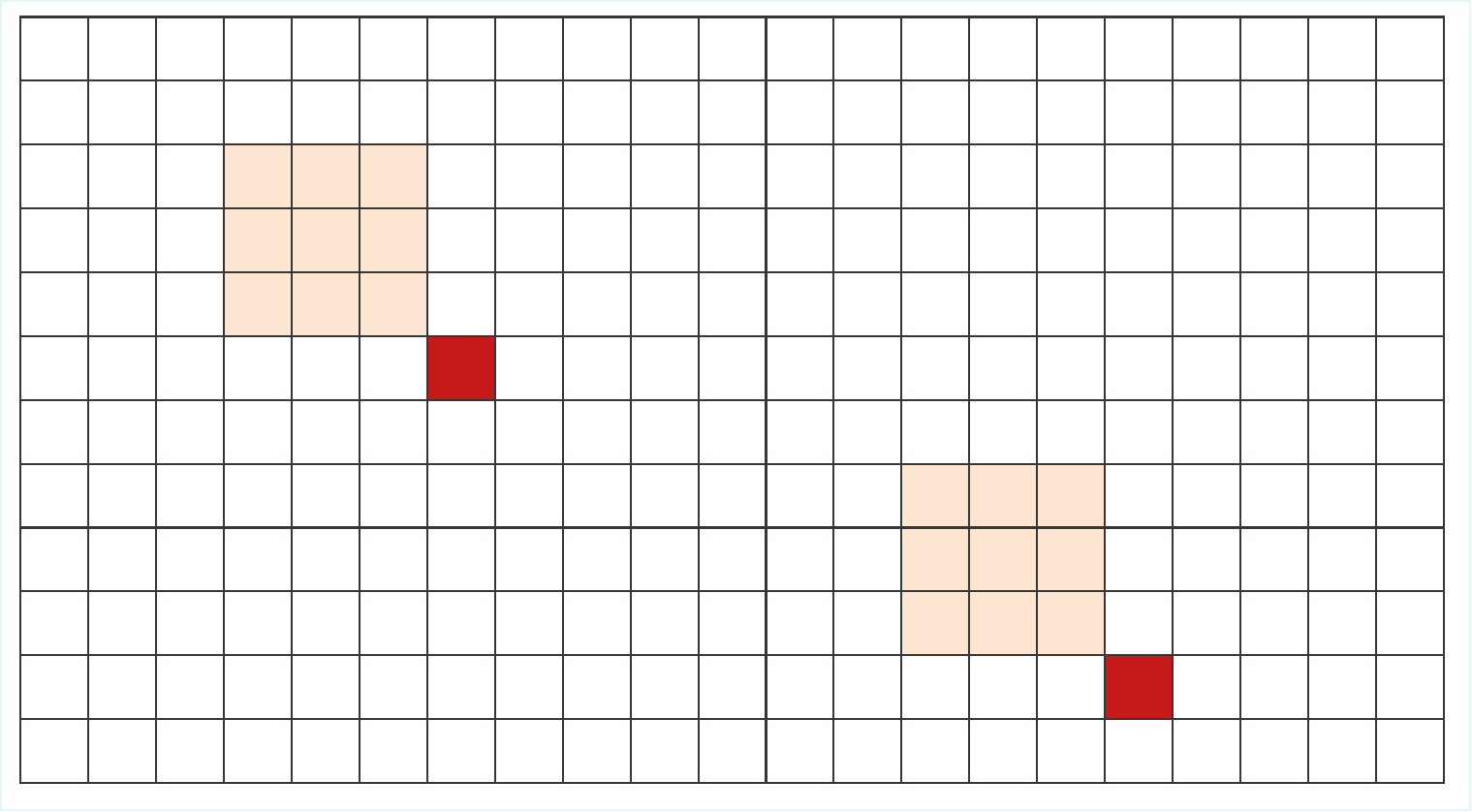}
\\
\vspace{4mm}
\includegraphics[width=0.2\linewidth]
{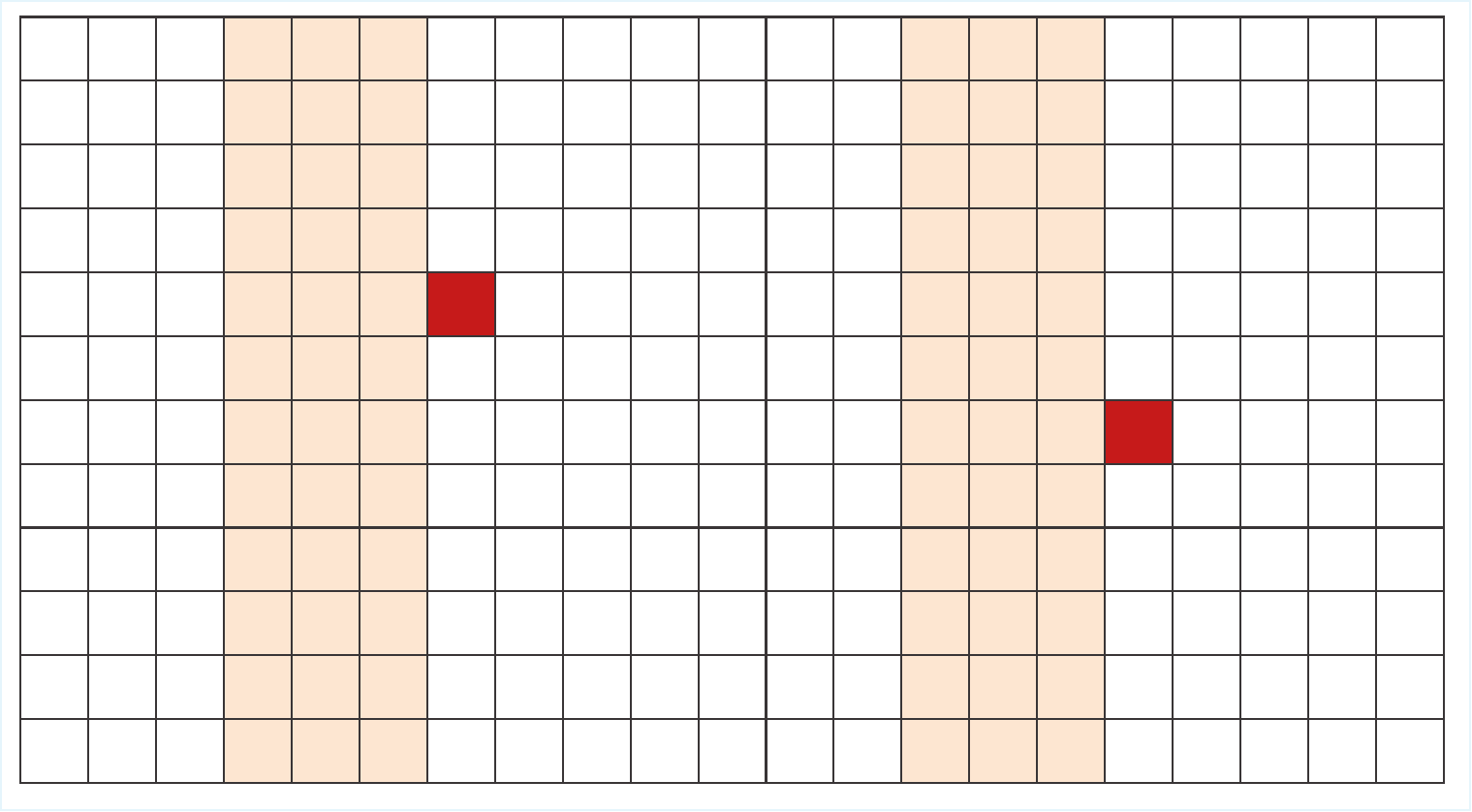}
\hspace{2mm}
\includegraphics[width=0.2\linewidth]
{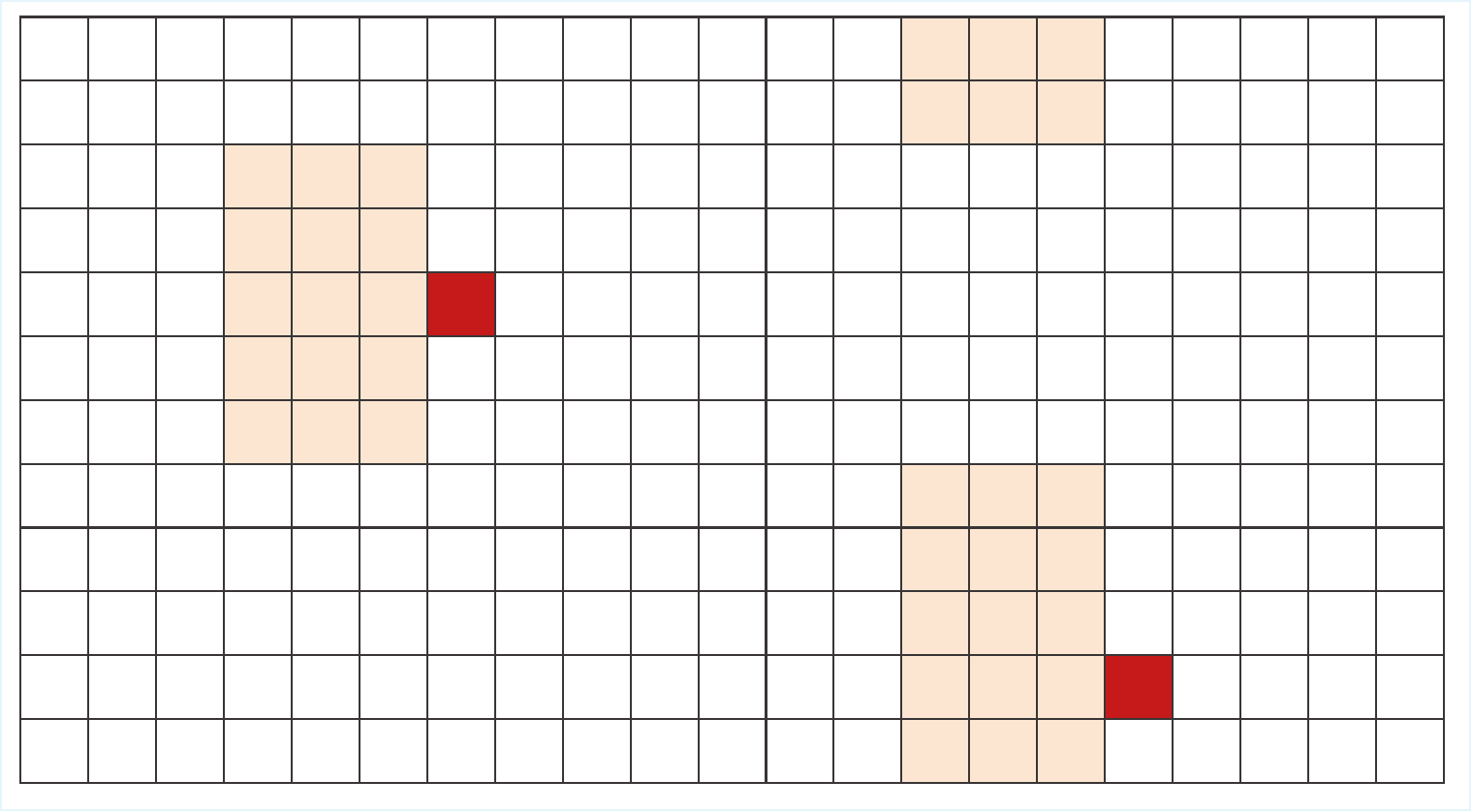} 
\hspace{2mm}
\includegraphics[width=0.2\linewidth]
{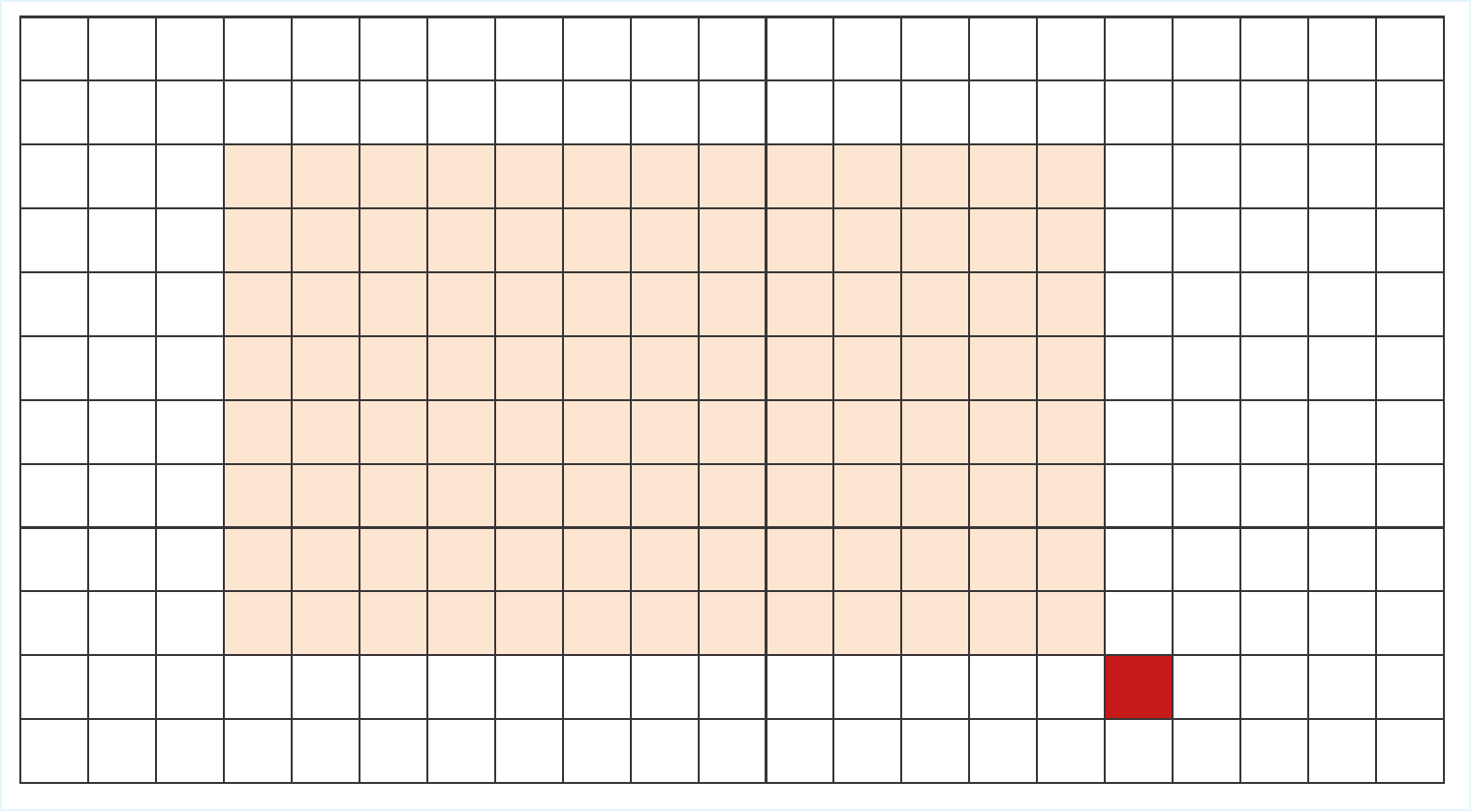}
\caption{Illustration of policies $\pi_1$ (first column), $\pi_2$ (second column) and $\pi_3$ (third column).
}
\label{policies_2drops}
\end{figure}

\subsubsection{Discussion of results}
\label{s:dd-res}

Figure \ref{Drop_drop_simulations} illustrates the evolution of the system under each of the candidate policies for $N = 100$ and $\kappa = 5,000$ (left group) and $\kappa = 20,000$ (right group). 

For each policy, we again evaluate the average hitting times and the average values across 2,000 independent runs of the MDP, for the case $N = 32$ and $\kappa = 100,000$. 

Figure \ref{Drop_drop_policy_comparison} shows the average values of the policies $\pi_1$, $\pi_2$ and $\pi_3$. In each experiment, we started from a configuration with two 3x3 droplets at distance 13 in both the horizontal and vertical directions. The results clearly imply that policy $\pi_3$ outperforms the other two candidates. The figure includes standard deviation error bars for the expected value of policy $\pi_4$. The results indicate that diagonal growth is the optimal choice for efficient nucleation, as quantified by the expected total discounted reward.

\begin{figure}
\centering
\includegraphics[width=0.12\linewidth]
{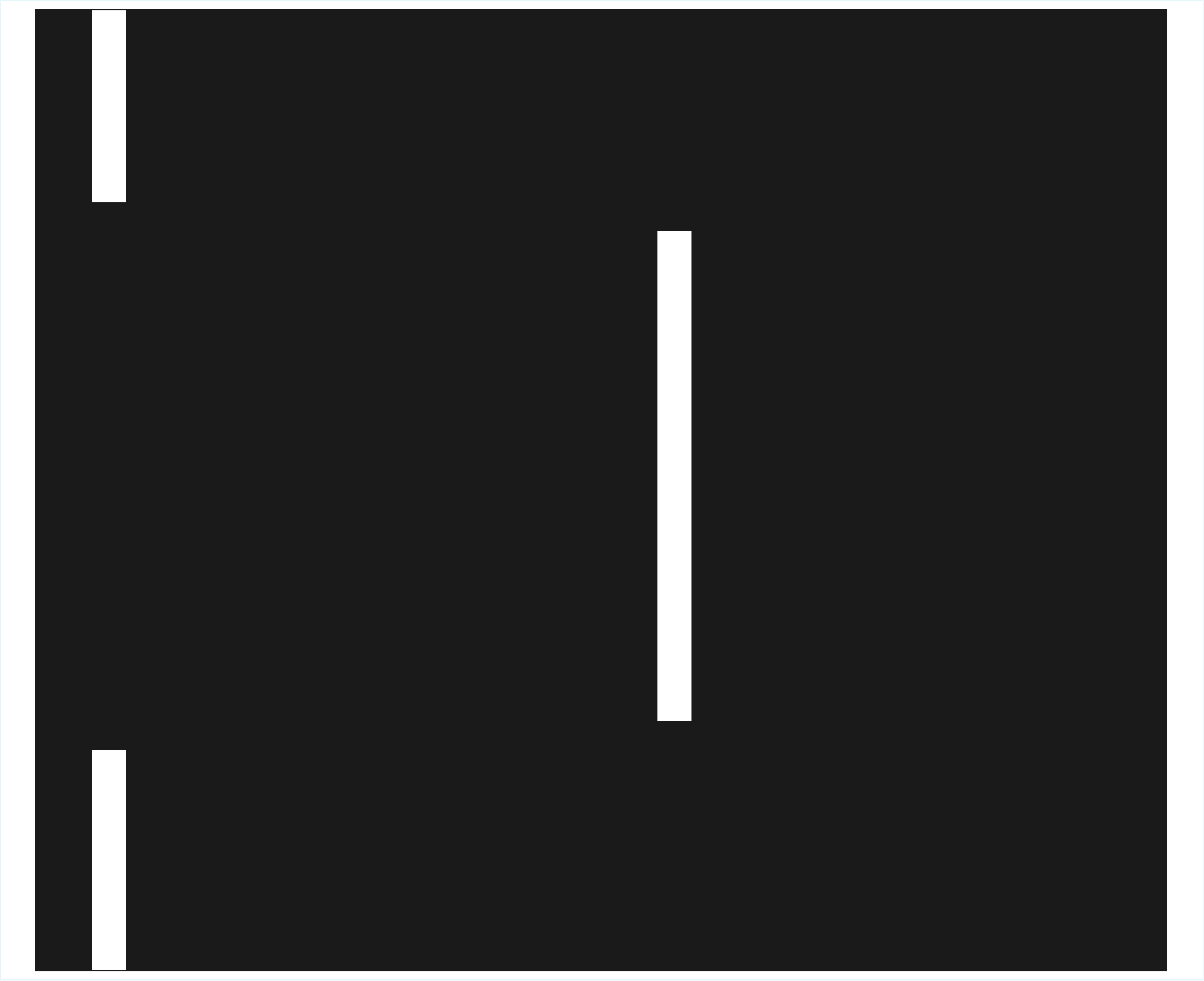}
\includegraphics[width=0.12\linewidth]
{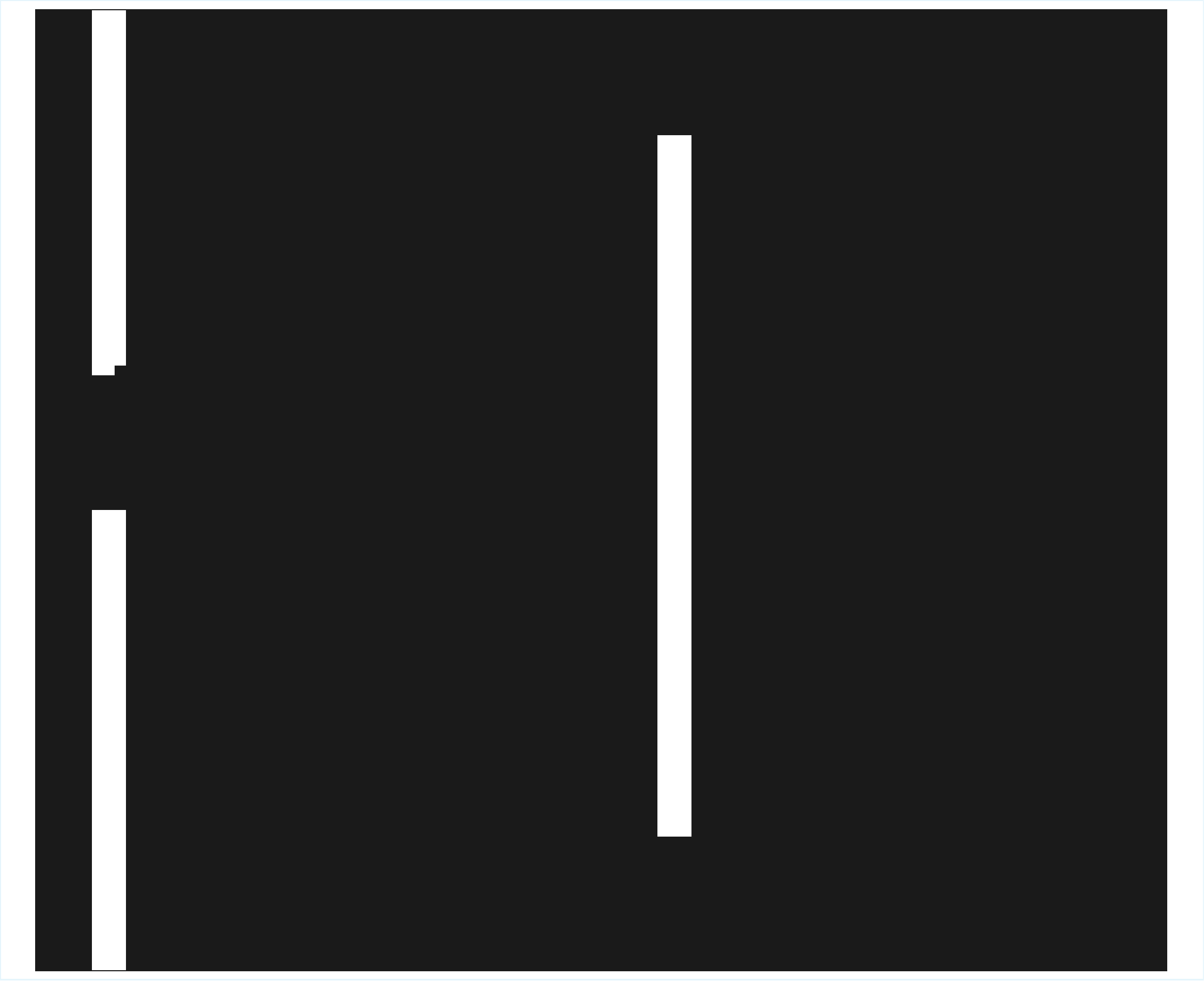}
\includegraphics[width=0.12\linewidth]
{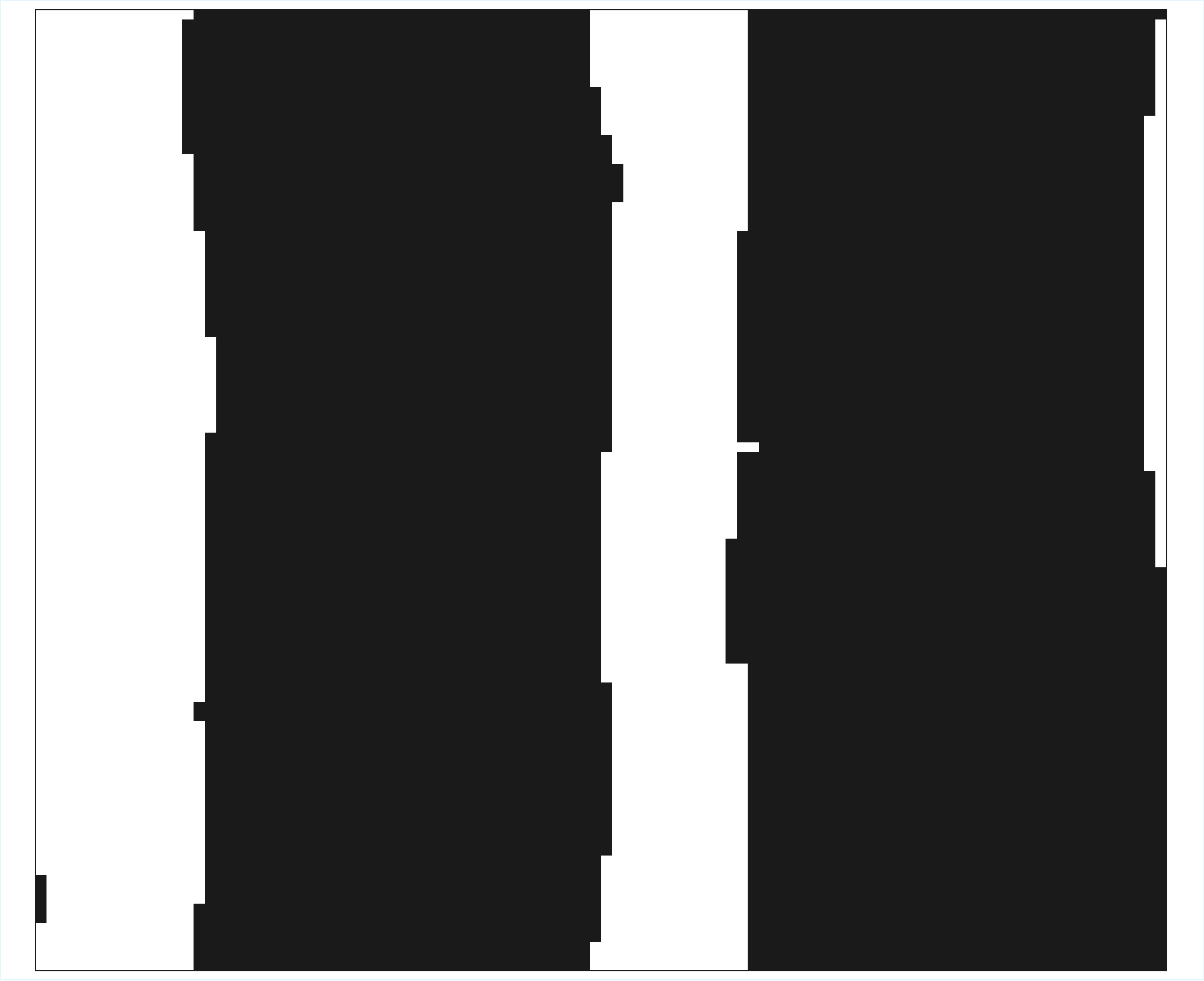}
\hspace{1.cm}
\includegraphics[width=0.12\linewidth]
{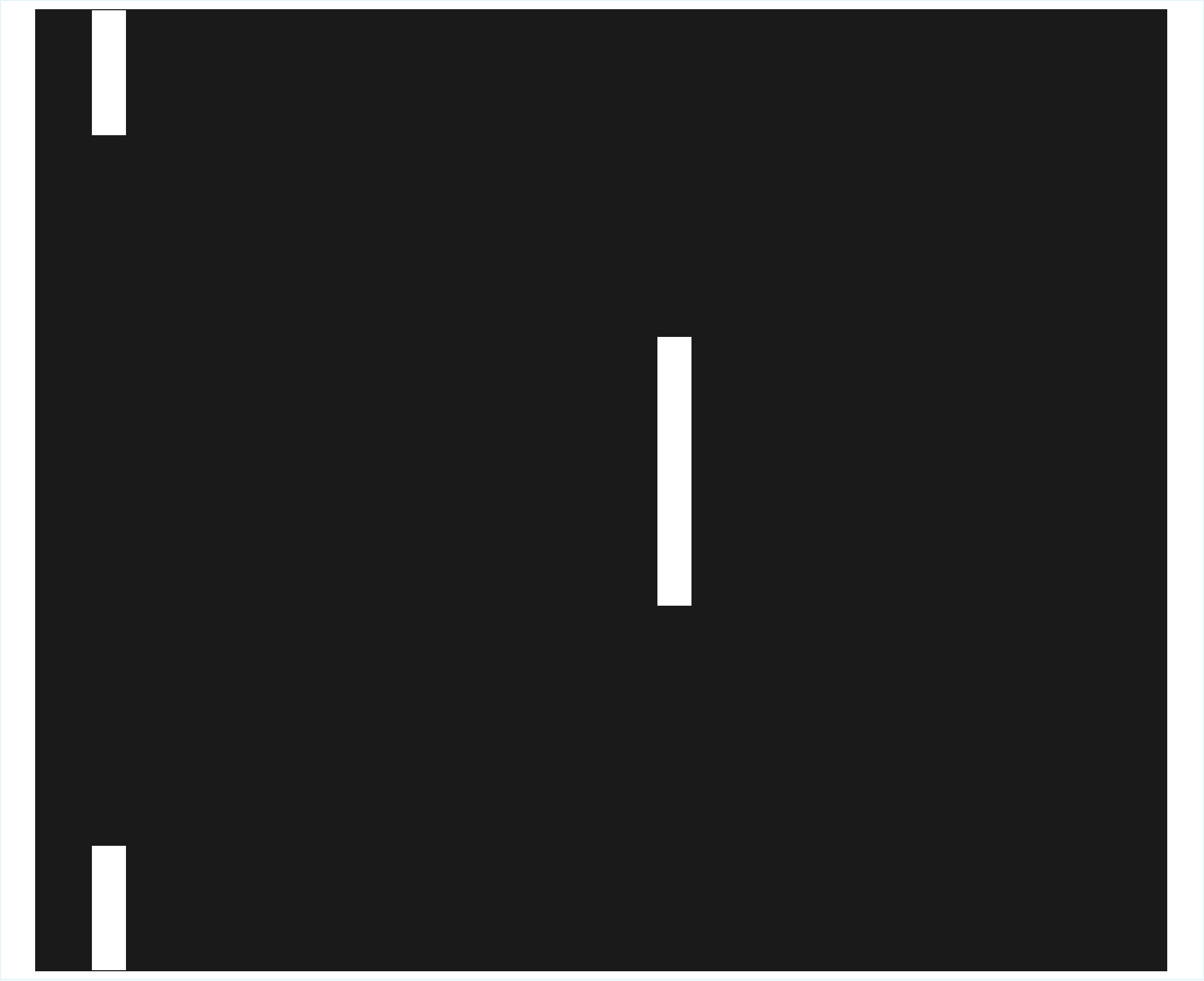}
\includegraphics[width=0.12\linewidth]
{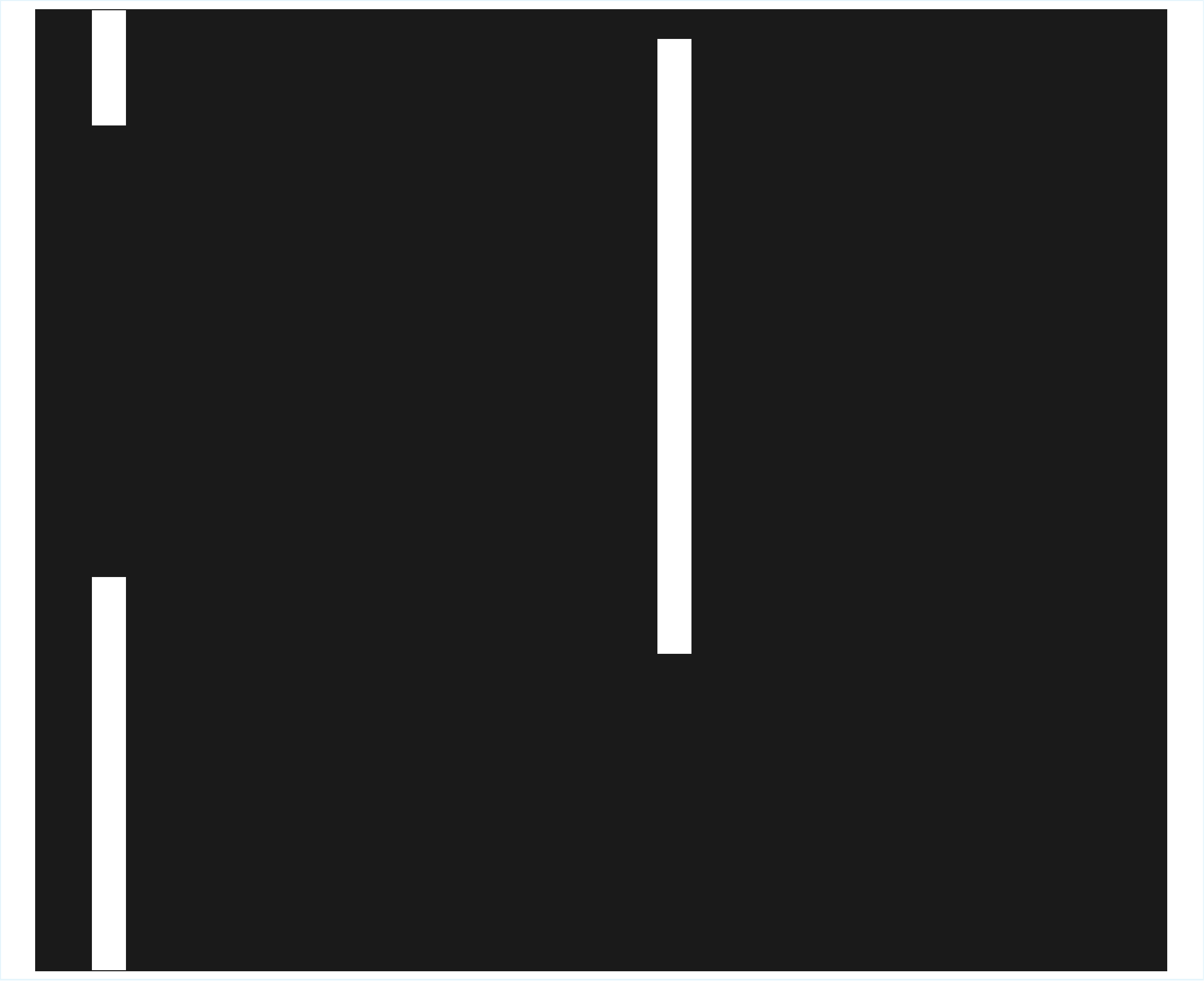}
\includegraphics[width=0.12\linewidth]
{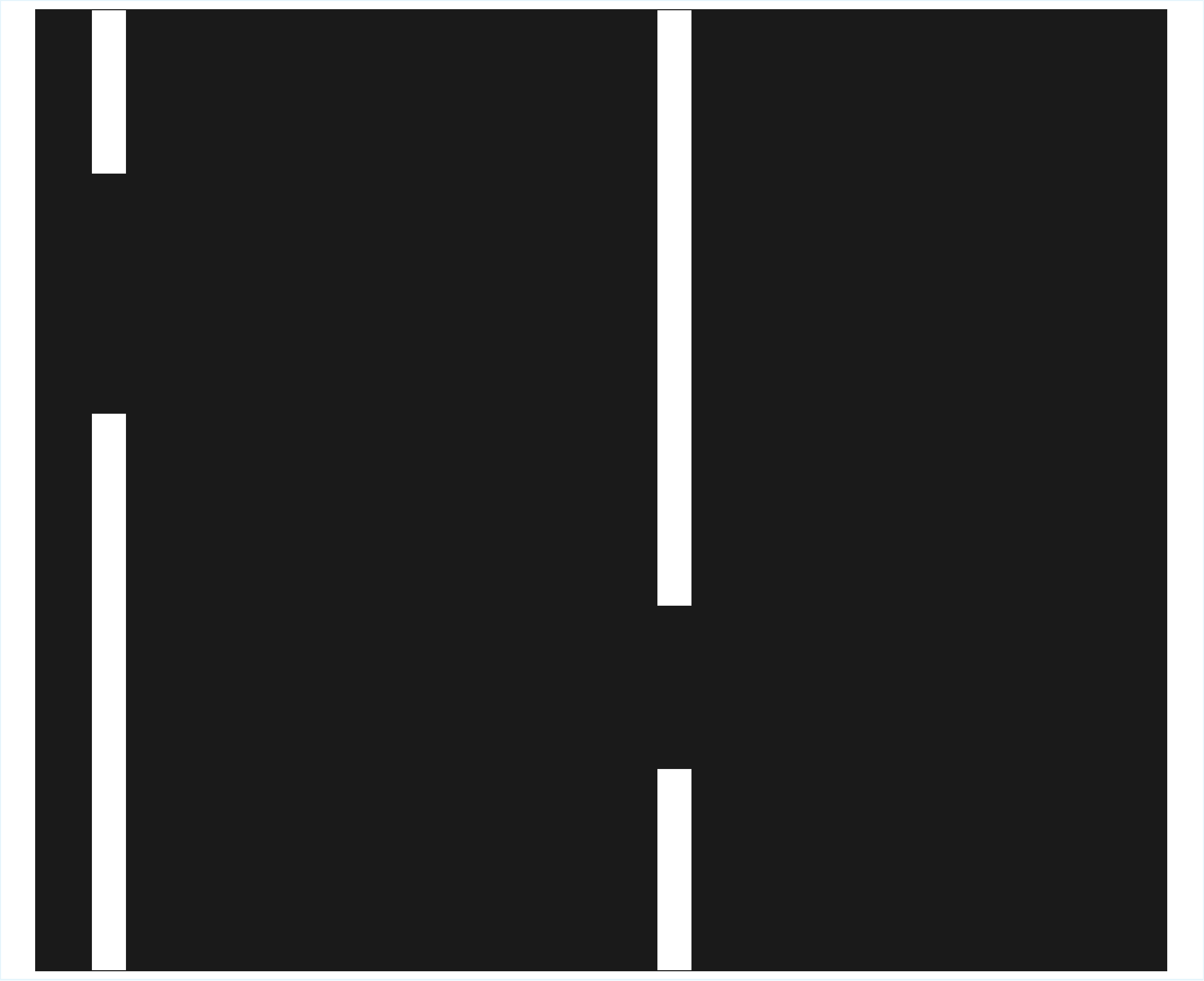}
\\
\vspace{1mm}
\includegraphics[width=0.12\linewidth]
{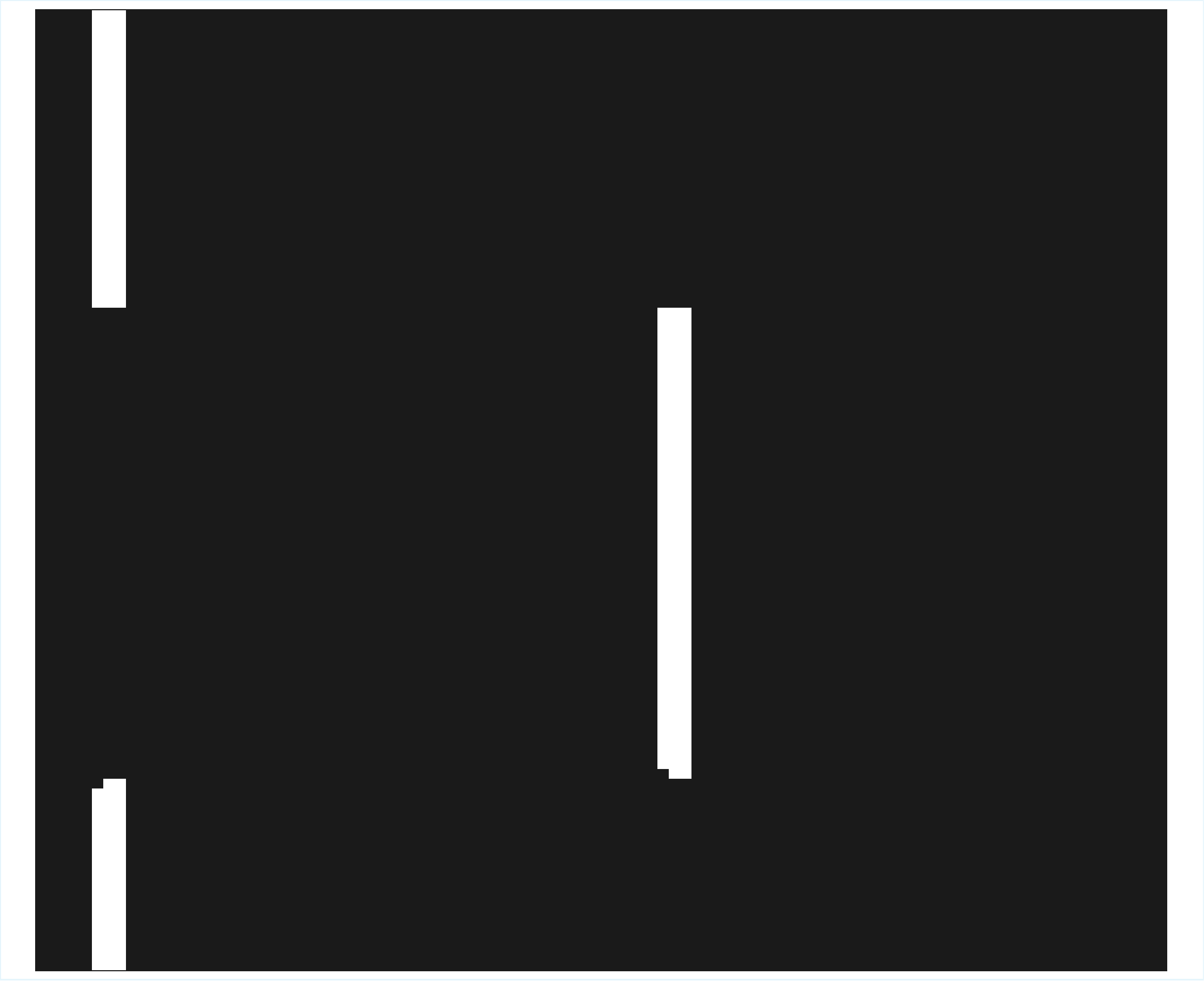}
\includegraphics[width=0.12\linewidth]
{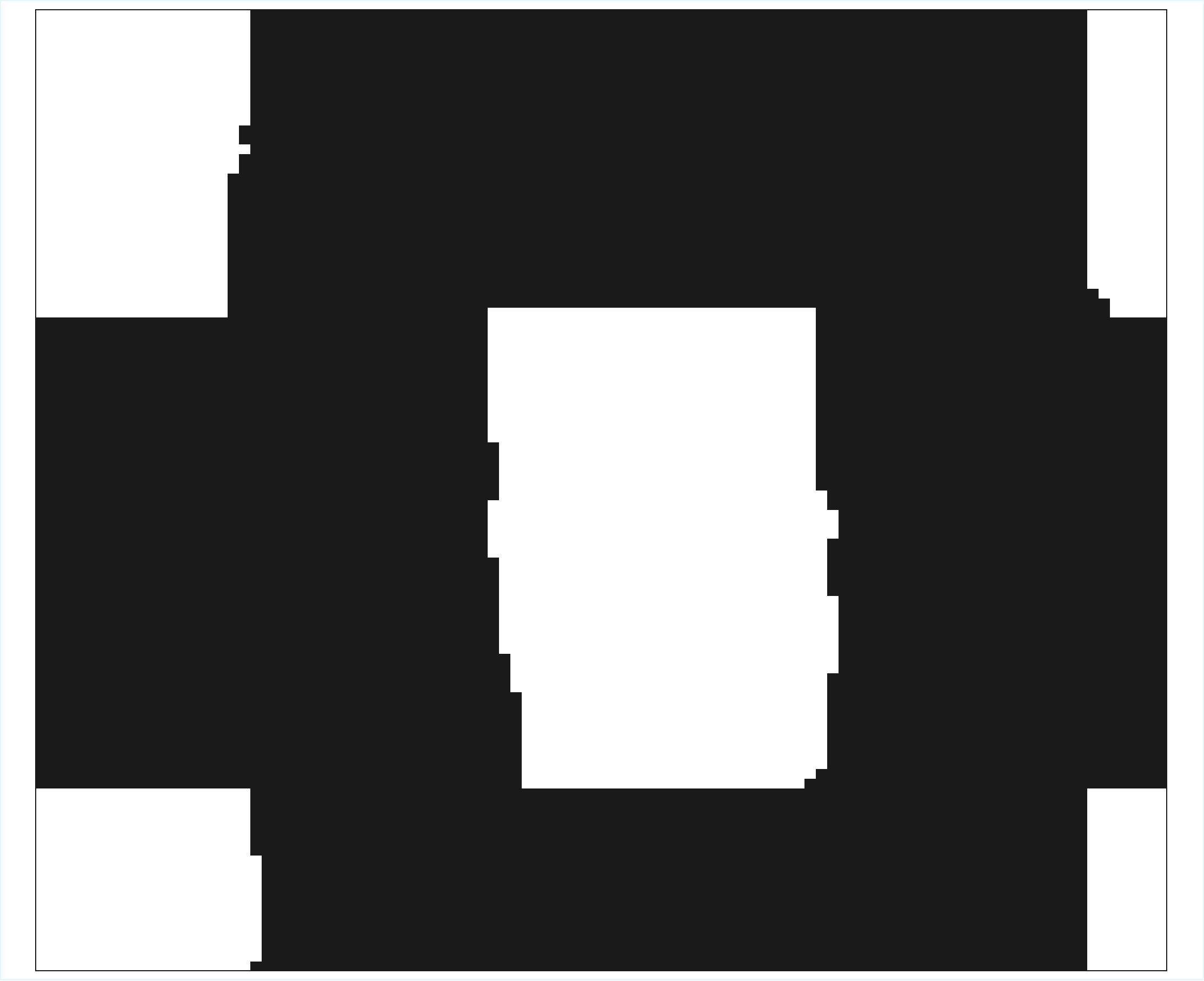}
\includegraphics[width=0.12\linewidth]
{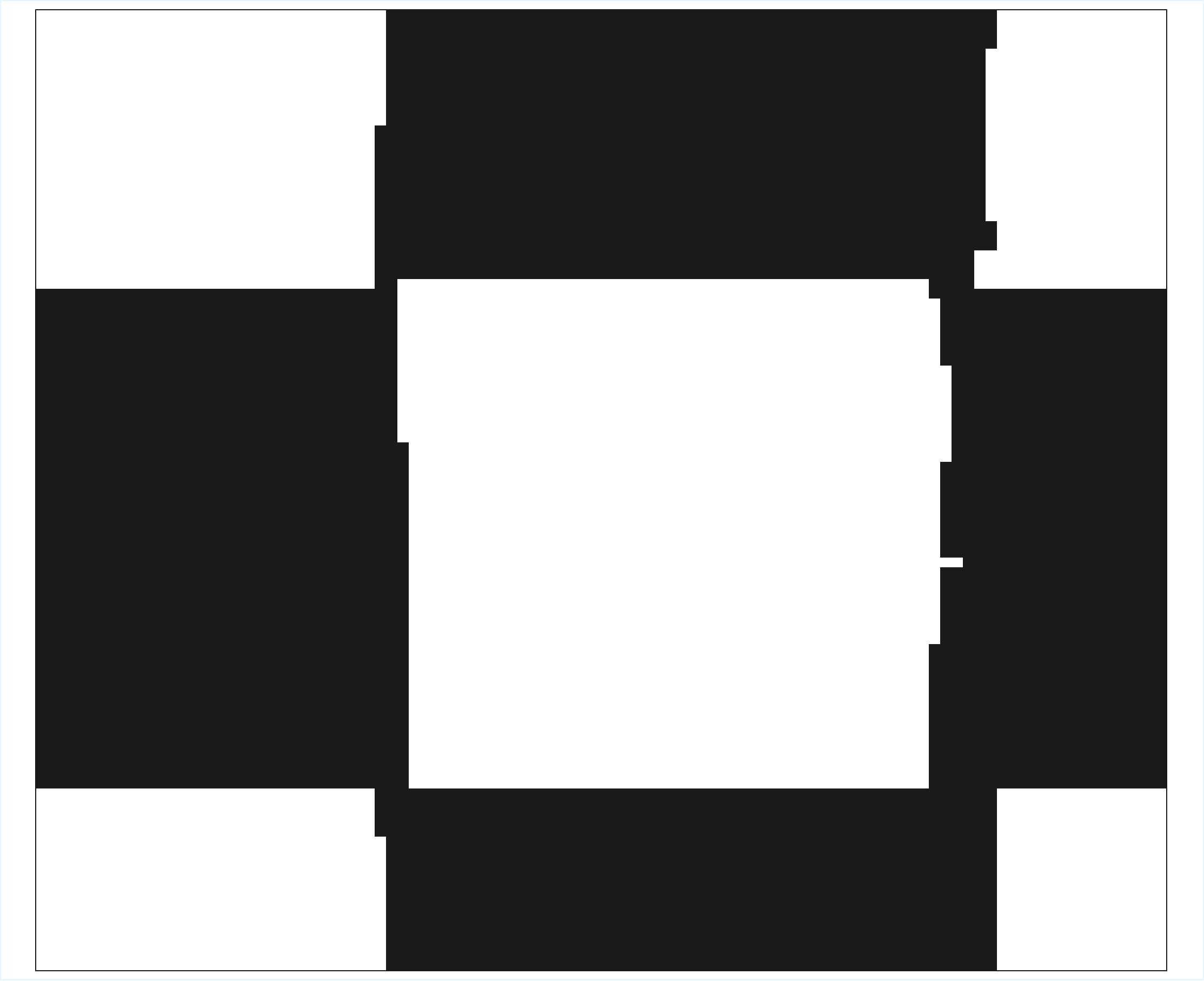}
\hspace{1.cm}
\includegraphics[width=0.12\linewidth]
{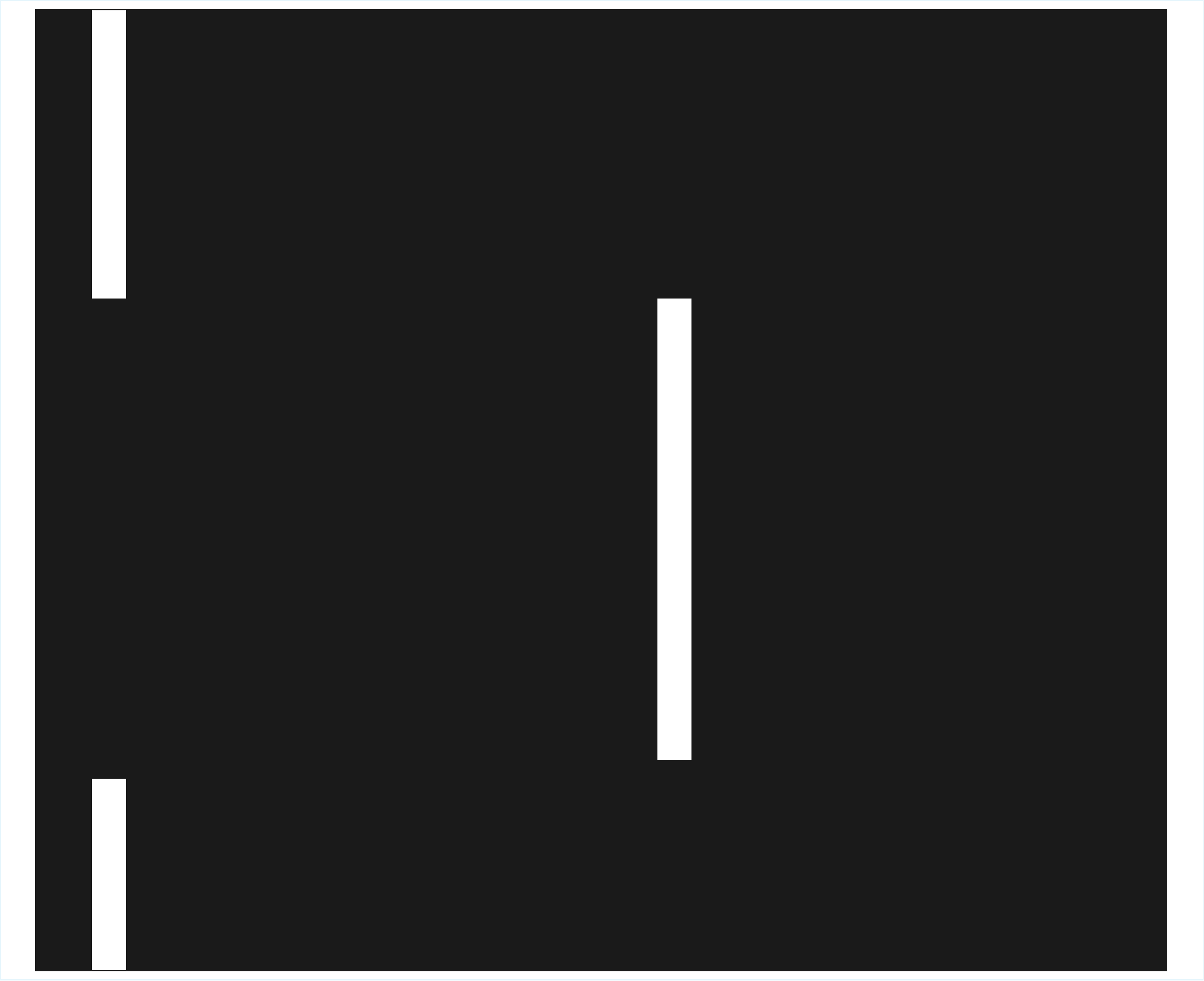}
\includegraphics[width=0.12\linewidth]
{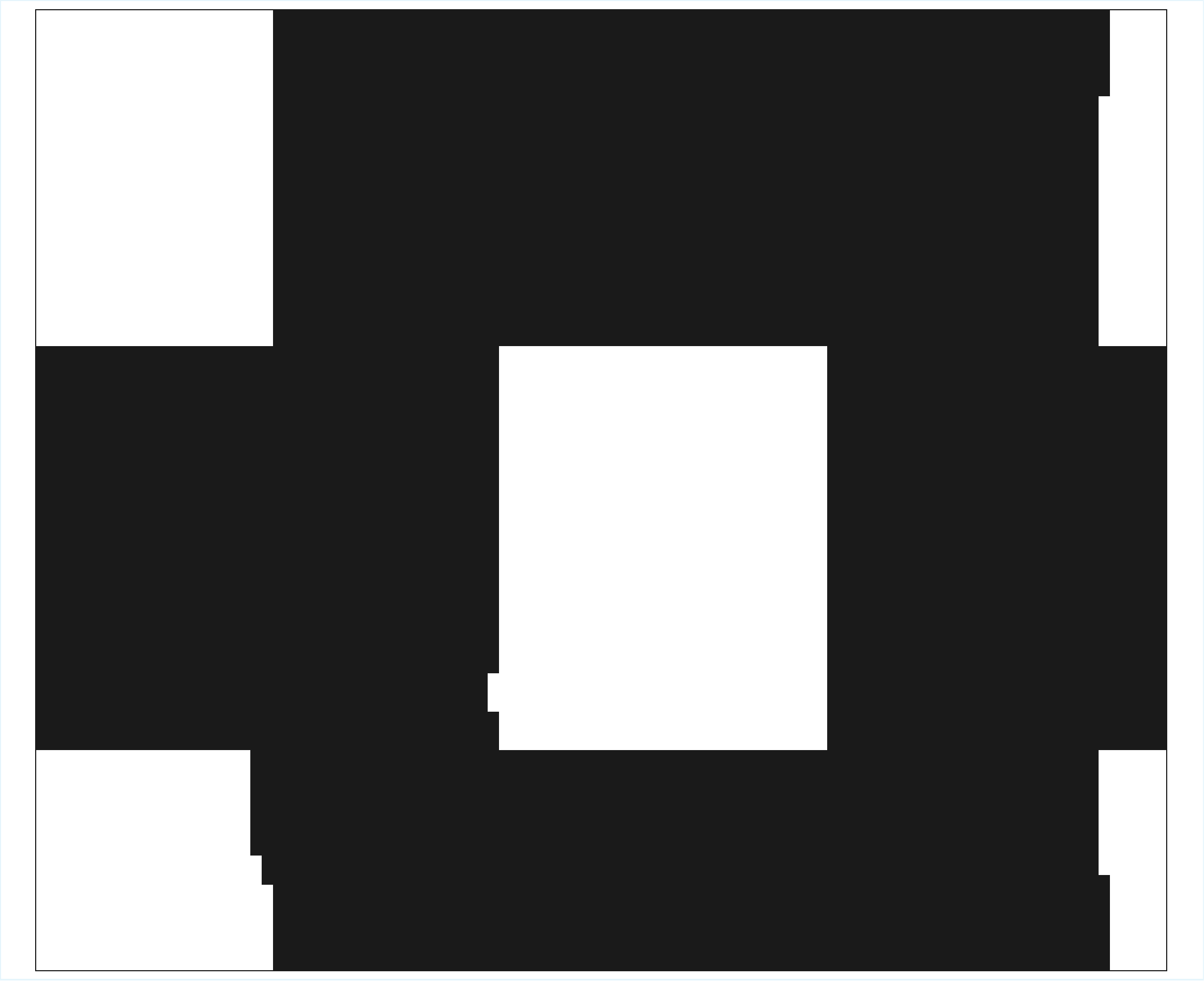}
\includegraphics[width=0.12\linewidth]
{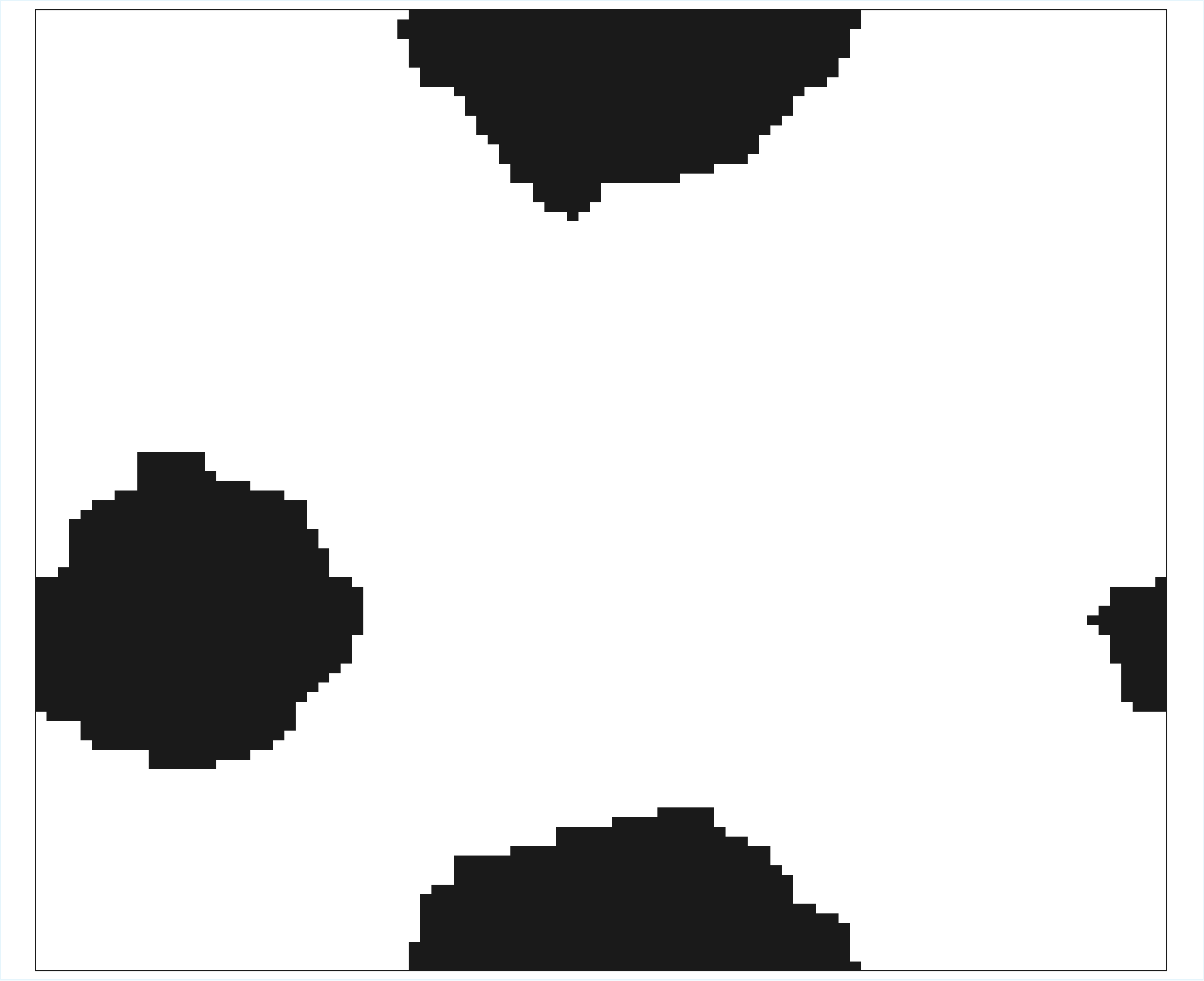}
\\
\vspace{1mm}
\includegraphics[width=0.12\linewidth]
{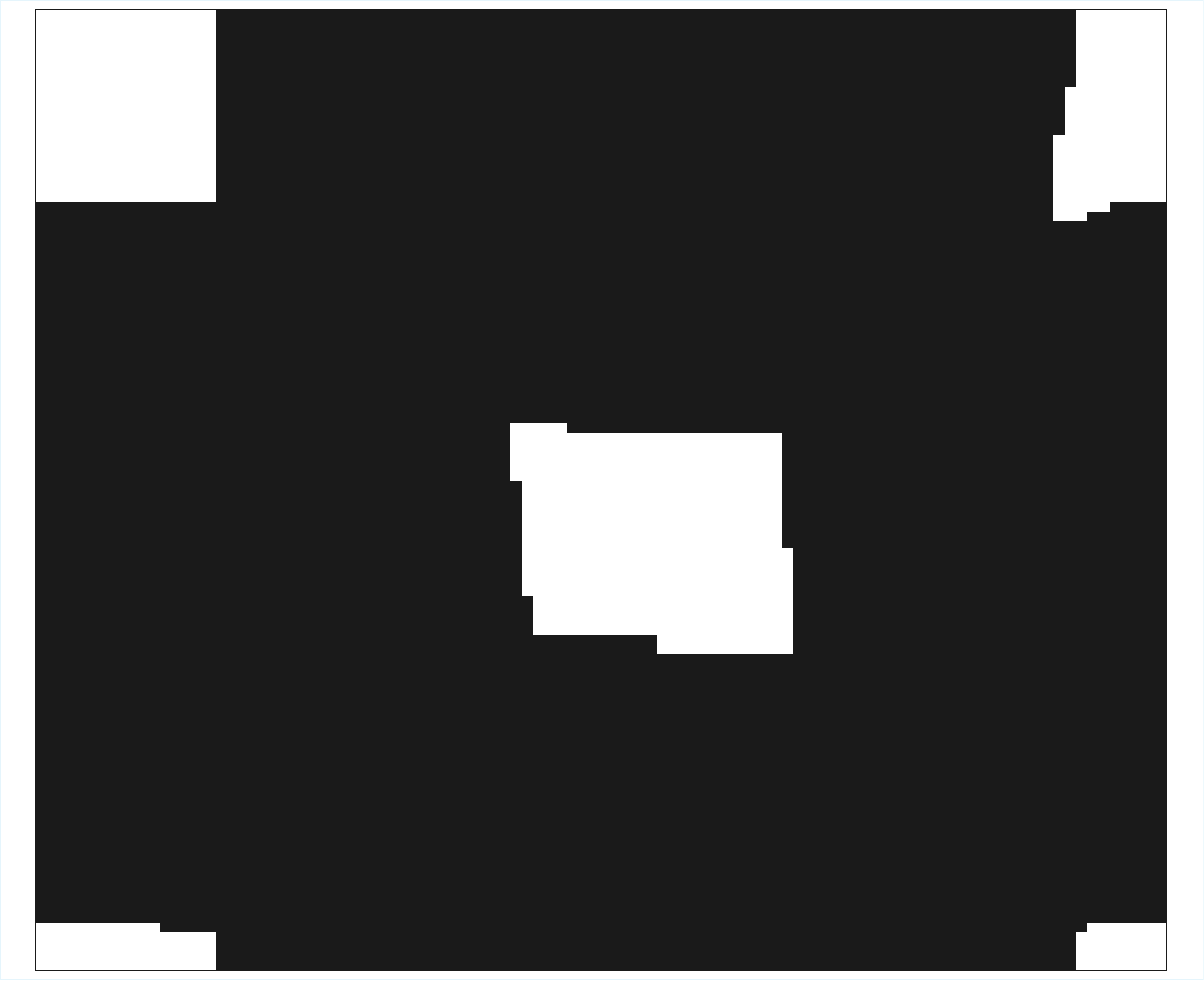}
\includegraphics[width=0.12\linewidth]
{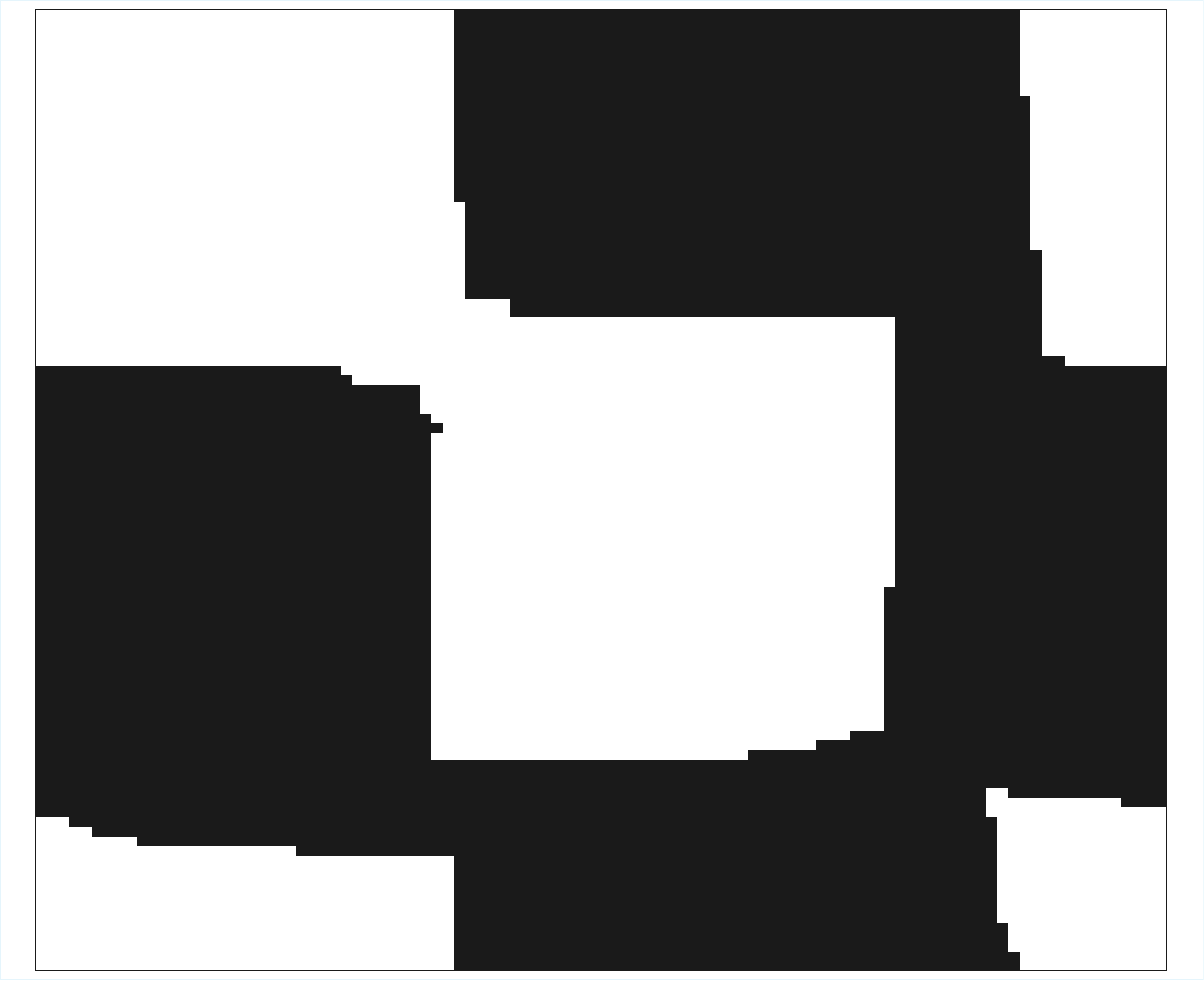}
\includegraphics[width=0.12\linewidth]
{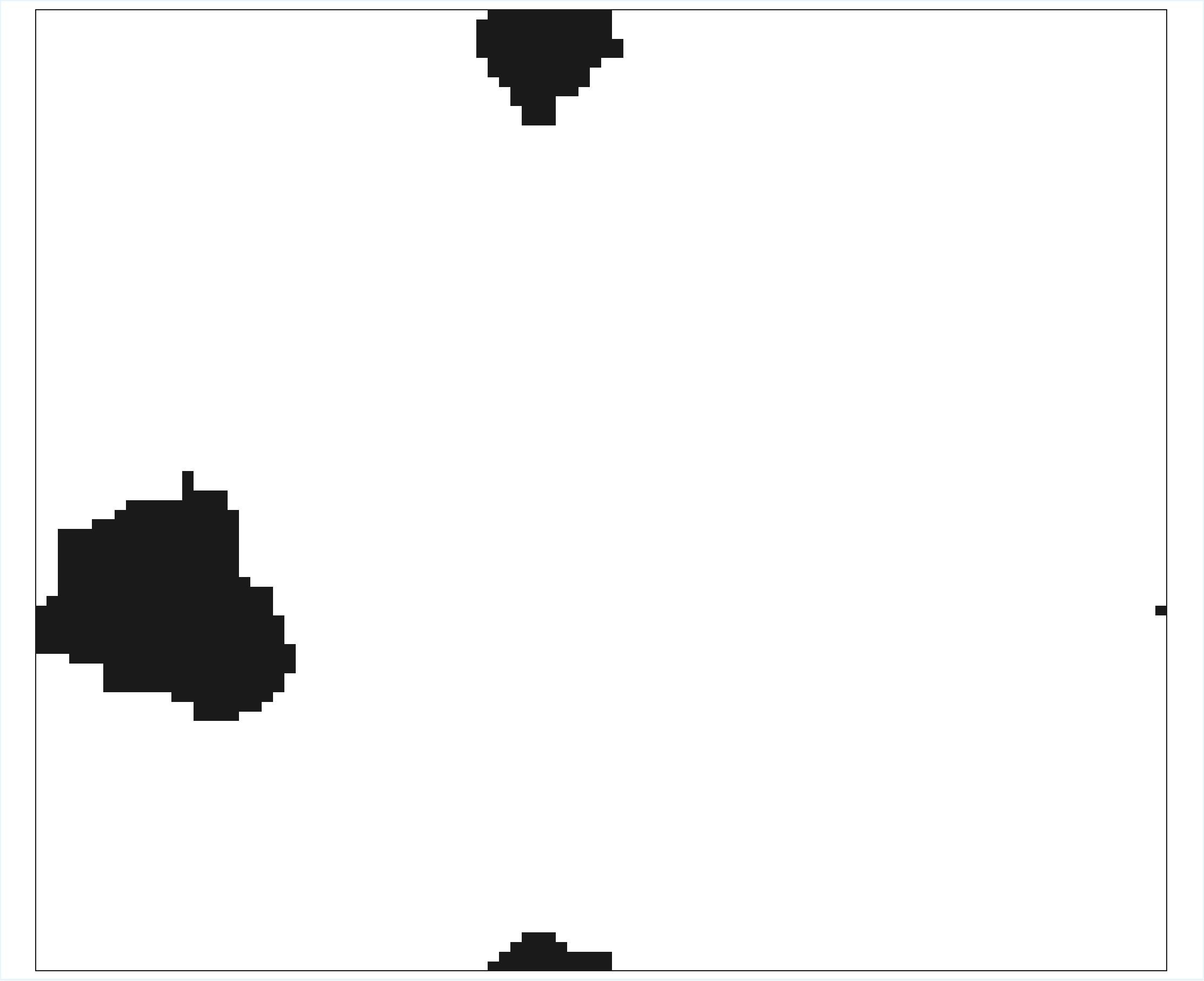}
\hspace{1.cm}
\includegraphics[width=0.12\linewidth]
{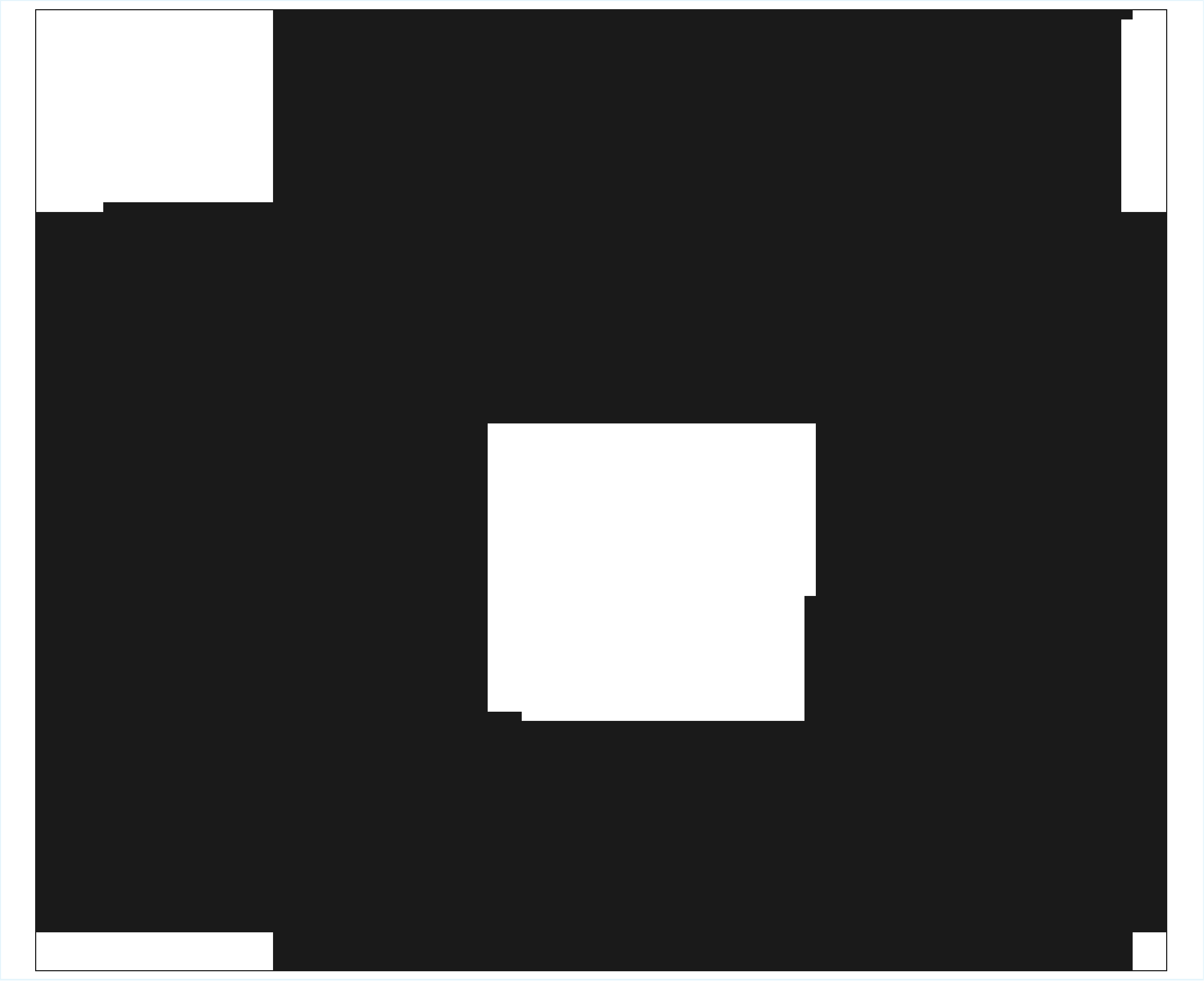}
\includegraphics[width=0.12\linewidth]
{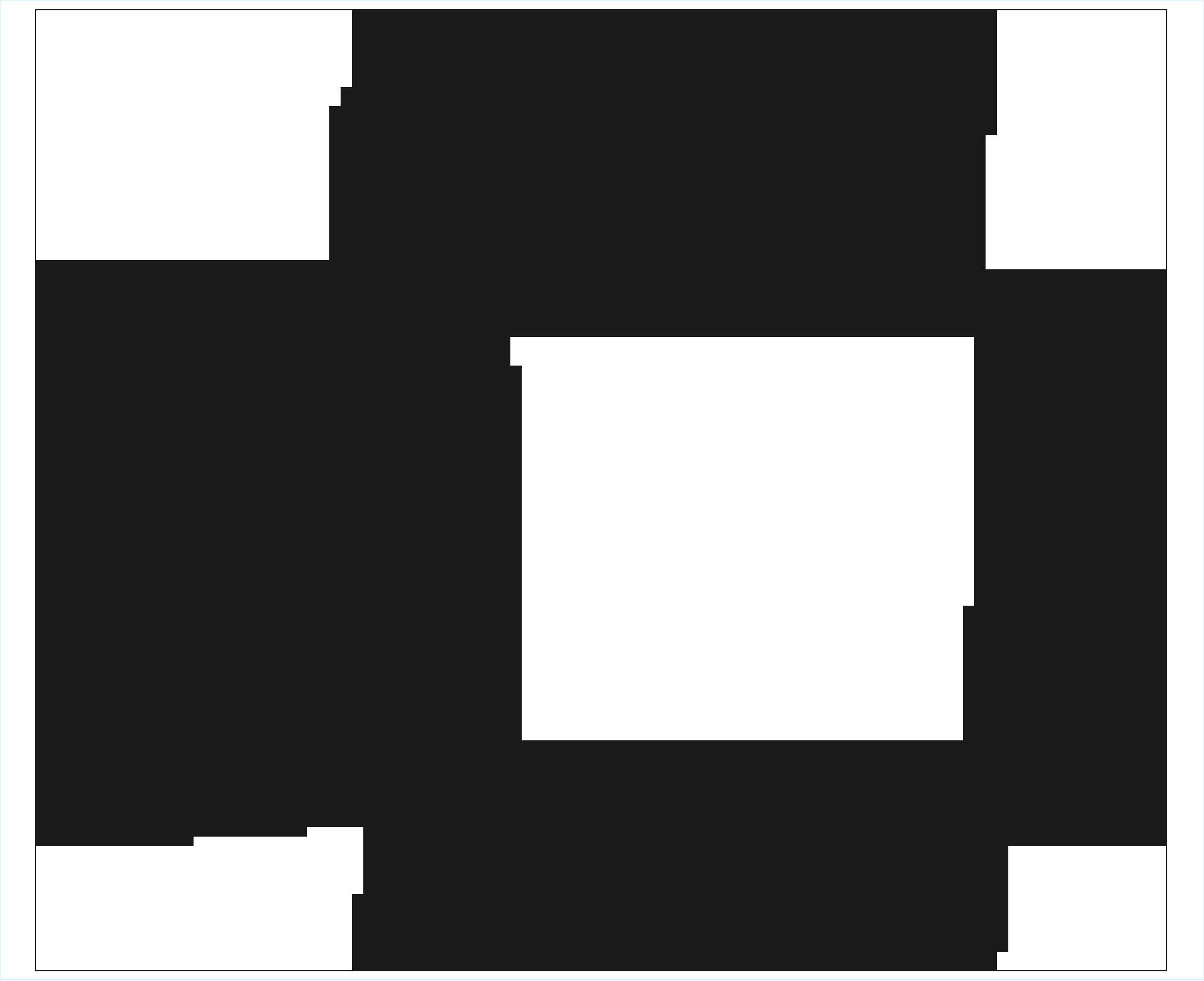}
\includegraphics[width=0.12\linewidth]
{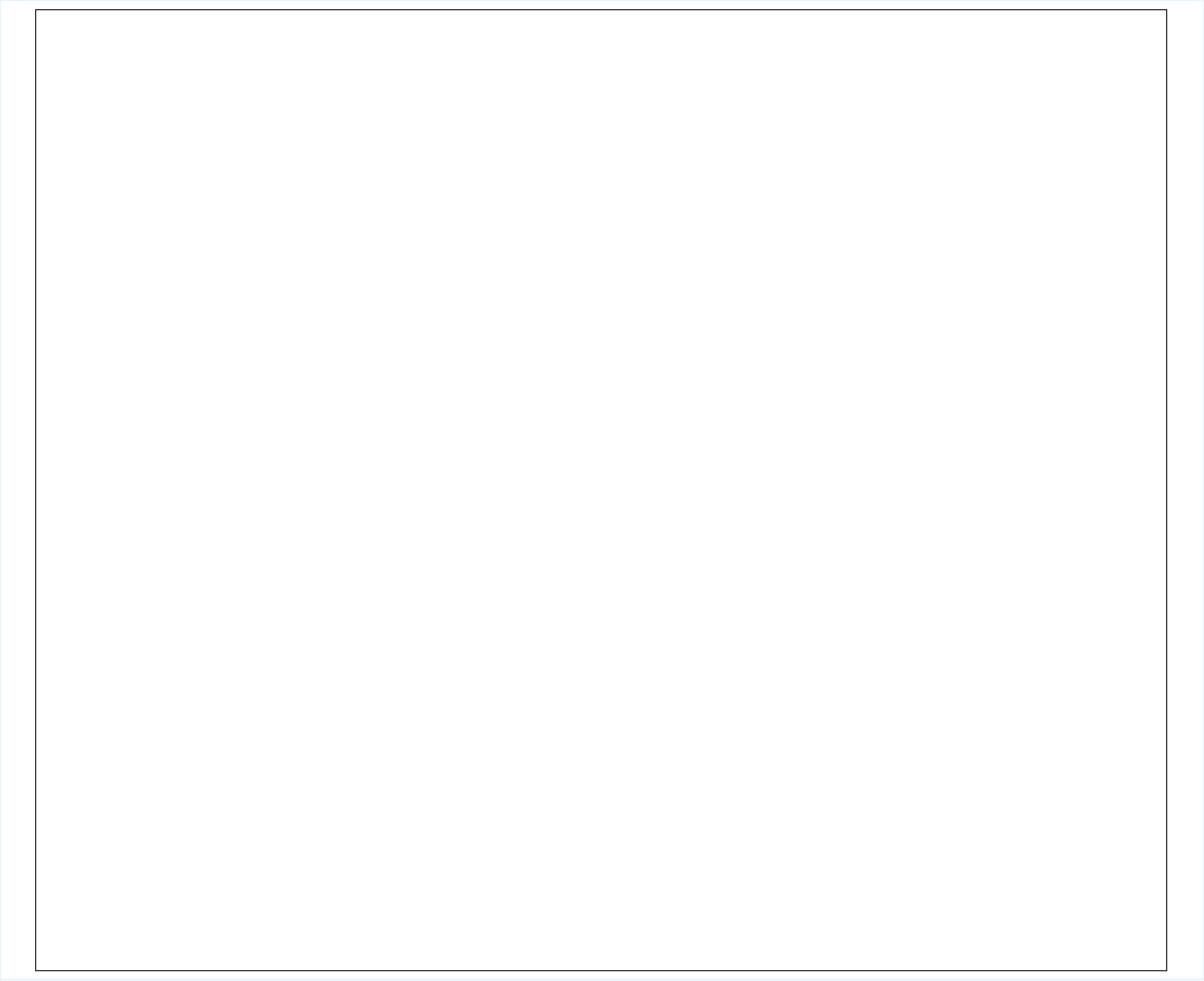}
\caption{Illustration of policies $\pi_1$ (first row), $\pi_2$ (second row) and $\pi_3$ (third row) for $N=100$ with initial seeds two 3x3 droplets at horizontal and vertical distances $47$.
Left group: $\kappa = 5,000$ 
at times $t = 250,500, 750$ (from left to right).
Right group:
$\kappa = 20,000$ at times 
$t = 150,300, 450$ (from left to right).
}
\label{Drop_drop_simulations}
\end{figure}

\begin{figure}
\centering
\includegraphics[width=0.6\linewidth]
{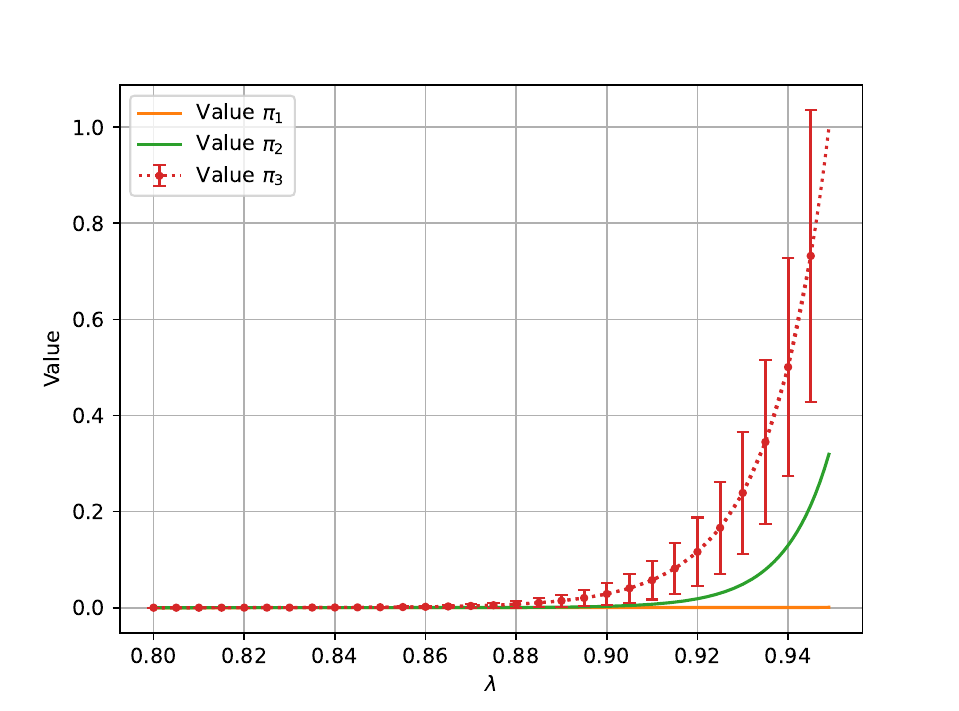}
\caption{Average values of policies $\pi_1$, $\pi_2$ and $\pi_3$ for $N = 32$ and $\kappa = 100,000$, starting from a configuration with two 3x3 droplets at distance 13 in horizontal and vertical directions.}
\label{Drop_drop_policy_comparison}
\end{figure}

We also use the simulation data to compute the average first hitting time to the all-plus configuration. The results, together with $95\%$-confidence intervals, are shown in Table \ref{Hitting_time_results_droplet_droplet}. 

\begin{table}[h!]
\caption{Average values of first hitting times with 95\%-confidence intervals for policies $\pi_1$, $\pi_2$ and $\pi_3$ for $N = 32$ and $\kappa = 100,000$.}
\centering
\begin{tabular}{l | l | l}
\hline
\textbf{Policy} & \textbf{Mean first hitting time} & \textbf{95\%-confidence interval} \\
\hline
$\pi_1$ & 199.567 & (198.768, 200.366) \\
$\pi_2$ & 79.758 & (79.470, 80.047) \\
$\pi_3$ & 58.318 & (57.988, 58.648) \\
\hline
\end{tabular}

\label{Hitting_time_results_droplet_droplet}
\end{table}

Among the candidate policies, $\pi_3$ appears to reach the all-plus configuration substantially faster than the others, rendering it the optimal choice with respect to both the expected total discounted reward and the expected first hitting time.

\section{Rigorous study of the two-stripe case}
\label{s:rigorous}
In \cite{deJongh2025}, the optimal policy in this MDP was derived analytically for robust configurations in which the sites with spin $+1$ form a single droplet by constructing an auxiliary MDP based on the geometric characterization of such configurations. In this section, we extend the analysis to configurations in the set $U^{2,x}$, that is, configurations in which there are two stripes with spin $+1$. 

\subsection{Optimal policy in case of two stripes}
\label{Optimal policy in case of two strips}
%To formalize the connection between the original low-temperature Ising MDP and the auxiliary MDP, we define a mapping $\mathcal{I}_x: U_x^{(1)} \cup U_x^{(2)} \rightarrow S^x$ as $\mathcal{I}_x(\sigma) = (i,j)$ if the distances between the two stripes in configuration $\sigma$ equal $i$ and $j$. 
%Given a configuration $\sigma \in  U_x^{(1)} \cup U_x^{(2)}$, let $\mathcal{J}_{\sigma}: A(\sigma) \rightarrow A^x(\mathcal{I}_x(\sigma))$ denote the mapping that gives for each $a \in A(\sigma)$ the corresponding element of $A^x(\mathcal{I}_x(\sigma))$. Spins that are neither at distance 1 nor 2 from the stripes are mapped to the action $0$. \mcdj{Do you think it is necessary to define these mappings? We're not using them anywhere.}

Recall the definition of the auxiliary MDP for the two-stripe case provided in Section \ref{The auxiliary MDP_strip_strip}. We now formally specify the transition probability kernel $P^x$ and the reward function $r^x$. The transition probability kernel can be computed by means of the method outlined in \cite[Lemmas 5.4--5.6]{deJongh2025}. The arguments can be easily extended to our setting. Lemma \ref{Aux_trans_probs_two_strips} provides the analogue of \cite[Lemma 5.6]{deJongh2025}.

\begin{lem}\label{Aux_trans_probs_two_strips}
    The transition probability kernel $P^x: S^x \times A^x \times S^x \rightarrow [0,1]$ for states $(i,j) \in S^x$, $i \geq j$, is given by
    \begin{align}
    \label{trans_probs_two_strips_1}
        P^x((i',j')|(i,j), a_{\ell 1}) &= \begin{cases}
            1/3, &\text{if } i' = i, \quad j' = j, \\
            2/3, &\text{if } i' = i-1, \quad j' = j, \\
            0, &\text{otherwise,}
        \end{cases}\quad \text{if }
        i > 2, \quad 1 < j \leq i, \\
    \label{trans_probs_two_strips_2}
        P^x((i',j')|(i,j), a_{\ell 2}) &= \begin{cases}
            5/9, &\text{if } i' = i, \quad j' = j, \\
            7/27, &\text{if } i' = i-1, \quad j' = j, \\
            5/27, &\text{if } i' = i-2, \quad j' = j, \\
            0, &\text{otherwise,}
        \end{cases} \quad \text{if }
        i > 3, \quad 1 < j \leq i,  \\
    \label{trans_probs_two_strips_3}
        P^x((i',j')|(i,j), a_{s1}) &= \begin{cases}
            1/3, &\text{if } i' = i, \quad j' = j, \\
            2/3, &\text{if } i' = i, \quad j' = j-1, \\
            0, &\text{otherwise,}
        \end{cases} \quad \text{if } i > j, \quad j > 2, \\
    \label{trans_probs_two_strips_4}
        P^x((i', j')|(i,j), a_{s2}) &= \begin{cases}
            5/9, &\text{if } i' = i, \quad j' = j, \\
            7/27, &\text{if } i' = i, \quad j' = j-1, \\
            5/27, &\text{if } i' = i, \quad j' = j-2, \\
            0, &\text{otherwise,}
        \end{cases} \quad \text{if } i > j, \quad j > 3, \\
    \label{trans_probs_two_strips_5}
        P^x((i',j')|(2,2), a_{\ell 1}) &= \begin{cases}
            1/4, &\text{if } i' = 2, \quad j' = 2, \\
            3/4, &\text{if } i' = 0, \quad j' = 2, \\
            0, &\text{otherwise,}
        \end{cases} \\
    \label{trans_probs_two_strips_6}
        P^x((i',j')|(3,j),a_{\ell 2}) &= \begin{cases}
            7/18, &\text{if } i' = 3, \quad j' = j, \\
            31/144, &\text{if } i' = 2, \quad j' = j, \\
            19/48, &\text{if } i' = 0, \quad j' = j, \\
            0, &\text{otherwise,}
        \end{cases}  \quad \text{if } j = 2, 3, \\
    \label{trans_probs_two_strips_7}
        P^x((i',j')|(i,2), a_{s1}) &= \begin{cases}
            1/4, &\text{if } i' = i, \quad j' = 2, \\
            3/4, &\text{if } i' = i, \quad j' = 0, \\
            0, &\text{otherwise,}
        \end{cases} \quad \text{if } i > 2, \\
    \label{trans_probs_two_strips_8}
        P^x((i',j')|(i,3), a_{s2}) &= \begin{cases}
            7/18, &\text{if } i' = i, \quad j' = j, \\
            31/144, &\text{if } i' = i, \quad j' = 2, \\
            19/48, &\text{if } i' = i, \quad j' = 0, \\
            0, &\text{otherwise,}
        \end{cases} \quad \text{if } i > 3.
    \end{align}
    The expressions for states $(i,j) \in S^x$, $i < j$ follow immediately from symmetry. 
\end{lem}
\begin{proof}
    The expressions can be obtained by a similar argument as that presented in \cite[Lemma 5.6]{deJongh2025}.
    Expressions (\ref{trans_probs_two_strips_1}) and (\ref{trans_probs_two_strips_3}) follow immediately from \cite[Figure 12, Table 2]{deJongh2025}. Expressions (\ref{trans_probs_two_strips_2}) and (\ref{trans_probs_two_strips_4}) follow from \cite[Figure 11, Table 1]{deJongh2025}. Expressions (\ref{trans_probs_two_strips_5}) and (\ref{trans_probs_two_strips_7}) follow from \cite[Figure 16, Table 6]{deJongh2025}. Finally, expressions (\ref{trans_probs_two_strips_6}) and (\ref{trans_probs_two_strips_8}) follow from \cite[Figure 17, Table 7]{deJongh2025}.
\end{proof}
Furthermore, we define the reward function of the auxiliary process as
\begin{equation*}
    r^x(s) = \begin{cases}
        1, &\text{if } s = (0, 0), \\
        0, &\text{otherwise.}
    \end{cases}
\end{equation*}

\subsection{The optimal policy}
\label{The optimal policy_strip_strip}
The optimal policy of the auxiliary MDP $(S^x, A^x, P^x, r^x)$ is provided in Theorem \ref{opt_pol_two_strips}. 
\begin{thm}\label{opt_pol_two_strips}
A stationary, deterministic policy $\pi^* = (d^*)^{\infty}$ is optimal in the auxiliary MDP $(S^x, A^x, P^x, r^x)$ if and only if
\begin{equation*}
    d^*(i,j) \in \begin{cases}
        A_1(i,j), &\text{if } \lambda \in (\lambda_c, 1), \\
        A_1(i,j) \cup A_2(i,j), &\text{if } \lambda = \lambda_c, \\
        A_2(i,j), &\text{if } \lambda \in (0, \lambda_c), 
    \end{cases}
\end{equation*}
for all $(i,j) \in S^x$, where $\lambda_c = 15/17$ and the functions $A_k: S^x \rightarrow P(A^x)$, $k = 1,2$, are as defined in expressions (\ref{eq:action-ssc}) and (\ref{eq:action-ssnc}). 
  
\end{thm}
\begin{proof}
By \cite[p.~152]{Puterman} it suffices to show that the value function $v^{\pi^*}_{\lambda}: S^x \rightarrow \mathbb{R}$ of a stationary, deterministic policy $\pi^* = (d^*)^{\infty}$ satisfies the Bellman equations, given in expression (\ref{Bellman_equations}), if and only if it has the form specified above. We prove the result for states $(i,j) \in S^x$ of the form $i \geq j$. Analogous results for the remaining states follow immediately from symmetry.

First, we define $\Pi_k$, $k = 1, 2$, as the sets of stationary, deterministic policies $\pi_k = (d_k)^{\infty}$ such that $d_k(i,j) \in A_k(i,j)$ for all $(i,j) \in S^x$. We prove the statement for the regime $\lambda \in (\lambda_c, 1)$. The case $\lambda \in (0, \lambda_c]$ can be established in a similar way. For $\lambda \in (\lambda_c, 1)$, we show that the value function $v_{\lambda}^{\pi}: S^x \rightarrow \mathbb{R}$ of policy $\pi \in \Pi$ satisfies the Bellman equations if and only if $\pi \in \Pi_1$. 

Let $\pi_1 = (d_1)^{\infty}$ denote a policy in the set $\Pi_1$. Using the transition probabilities presented in Lemma \ref{Aux_trans_probs_two_strips}, we obtain the following expressions for the value function  $v^{\pi_1}_{\lambda}: S^x \rightarrow \mathbb{R}$, $k = 1, 2$, for states $(i,j) \in S^x$, $i \geq j$.
\begin{align}
    \label{rec_1}&v^{\pi_1}_{\lambda}(i,j) = \dfrac{2\lambda}{3-\lambda}v^{\pi_1}_{\lambda}(i-1, j), \quad i > 2,  \\
    %\label{rec_2}&\tilde{v}^{\pi_k}_{\lambda}((i,j), a_{\ell 2}) = \dfrac{7\lambda}{3(9-5\lambda)}v^{\pi_k}_{\lambda}(i-1, j) + \dfrac{5\lambda}{3(9-5\lambda)}v^{\pi_k}_{\lambda}(i-2, j), \quad (i,j) \in S^x, \quad i > 3\\
    \label{rec_3}&v^{\pi_1}_{\lambda}(i,j) = \dfrac{2\lambda}{3-\lambda}v^{\pi_1}_{\lambda}(i, j-1), \quad j > 2, \\
    \label{rec_4}&v^{\pi_1}_{\lambda}(0, 0) = \dfrac{1}{1-\lambda}, \\
    %\label{rec_4}&\tilde{v}^{\pi_k}_{\lambda}((i,j), a_{s2}) = \dfrac{7\lambda}{3(9-5\lambda)}v^{\pi_k}_{\lambda}(i, j-1) + \dfrac{5\lambda}{3(9-5\lambda)}v^{\pi_k}_{\lambda}(i, j-2), \quad (i,j) \in S^x, \quad j>3, \quad i \neq j \\
    \label{rec_5}&v^{\pi_1}_{\lambda}(2,j) = \dfrac{3\lambda}{4-\lambda}v^{\pi_1}_{\lambda}(0, j),  \quad j = 0, 2\\
    \label{rec_6}&\tilde{v}^{\pi_1}_{\lambda}(3,j) = \dfrac{31\lambda}{8(18-7\lambda)}v^{\pi_1}_{\lambda}(2,j) + \dfrac{57\lambda}{8(18-7\lambda)}v^{\pi_1}_{\lambda}(0,j), \quad j =0,  2, 3\\
    \label{rec_7}&v^{\pi_1}_{\lambda}(i,2) = \dfrac{3\lambda}{4-\lambda}v^{\pi_1}_{\lambda}(i,0), \quad i > 2, \\
    \label{rec_8}&\tilde{v}^{\pi_1}_{\lambda}(i,3) = \dfrac{31\lambda}{8(18-7\lambda)}v^{\pi_1}_{\lambda}(i, 2) + \dfrac{57\lambda}{8(18-7\lambda)}v^{\pi_1}_{\lambda}(i, 0), \quad i > 3.
\end{align}
The expressions for states $(i,j) \in S^x$, $i < j$, follow from symmetry. 

For $\lambda \in (\lambda_c, 1)$, we show that for each $(i,j) \in S$, we have
\begin{align}
    \label{eqs_1}& r(i,j) + \lambda \sum\limits_{(i',j') \in S^x} P((i',j')|(i,j), a)v^{\pi_1}_{\lambda}(i,j) = r(i,j) + \lambda \sum\limits_{(i',j') \in S^x} P((i',j')|(i,j), a')v^{\pi_1}_{\lambda}(i,j),
\end{align}
for all $a, a' \in A_1(i,j)$ and
\begin{align}
    \label{ineqs_1}& r(i,j) + \lambda \sum\limits_{(i',j') \in S^x} P((i',j')|(i,j), a)v^{\pi_1}_{\lambda}(i,j) > r(i,j) + \lambda \sum\limits_{(i',j') \in S^x} P((i',j')|(i,j), a')v^{\pi_1}_{\lambda}(i,j),
\end{align}
for all $a \in A_1(i,j), \quad a' \notin A_1(i,j)$.

Using expressions (\ref{trans_probs_two_strips_6}) and (\ref{trans_probs_two_strips_7}), expression (\ref{eqs_1}) for state $(3,2)$ becomes
\begin{equation*}
    \dfrac{7\lambda}{18}v^{\pi_1}_{\lambda}(3,2) + \dfrac{31\lambda}{144}v^{\pi_1}_{\lambda}(2,2) + \dfrac{19\lambda}{48}v^{\pi_1}_{\lambda}(0, 2) = \dfrac{\lambda}{4}v^{\pi_1}_{\lambda}(3,2) + \dfrac{3\lambda}{4}v^{\pi_1}_{\lambda}(3,0),
\end{equation*}
which can be simplified to
\begin{equation}\label{eq1_1}
    20v^{\pi_1}_{\lambda}(3, 2) + 31v^{\pi_1}_{\lambda}(2, 2) + 57v^{\pi_1}_{\lambda}(0,2) - 108v^{\pi_1}_{\lambda}(3, 0) = 0,
\end{equation}
Using a similar approach, equation (\ref{eqs_1}) for the remaining states reduces to
\begin{align}
    \label{eq1_2}&v^{\pi_1}_{\lambda}(i, 2) + 8v^{\pi_1}_{\lambda}(i-1,2) - 9v^{\pi_1}_{\lambda}(i, 0) = 0, \quad i \geq 4,\\
    \label{eq1_3}&-8v^{\pi_1}_{\lambda}(i, 3) + 96v^{\pi_1}_{\lambda}(i-1,3)-31v^{\pi_1}_{\lambda}(i,2) -57v^{\pi_1}_{\lambda}(i, 0) = 0, \quad i \geq 4 \\
    \label{eq1_4}&v^{\pi_1}_{\lambda}(i-1,j) = v^{\pi_1}_{\lambda}(i, j-1) = 0, \quad i,j \geq 4.
\end{align}
Again using the recursive expressions (\ref{trans_probs_two_strips_1}--\ref{trans_probs_two_strips_8}) as well as equation (\ref{eqs_1}), we write expression (\ref{ineqs_1}) as:
\begin{align}
    \label{ineq1_1}&8v^{\pi_1}_{\lambda}(3,j) - 65v^{\pi_1}_{\lambda}(2, j) + 57v^{\pi_1}_{\lambda}(0, j) > 0, \quad j = 2, 3, \\
    \label{ineq1_2}&8v^{\pi_1}_{\lambda}(i,3) - 65v^{\pi_1}_{\lambda}(i,2) +57v^{\pi_1}_{\lambda}(i, 0) > 0, \quad i \geq 3,\\
    \label{ineq1_3}&-6v^{\pi_1}_{\lambda}(i,j) + 11v^{\pi_1}_{\lambda}(i-1,j) - 5v^{\pi_1}_{\lambda}(i-2,j) > 0, \quad i \geq 4, \\
    \label{ineq1_4}&-6v^{\pi_1}_{\lambda}(i,j) + 11v^{\pi_1}_{\lambda}(i,j-1) - 5v^{\pi_1}_{\lambda}(i,j-2) > 0, \quad i \geq j \geq 4.
\end{align}
Note that the statement for states of the form $(i, 0)$, $i \geq 0$, has been shown in \cite{deJongh2025}. 
%Some necessary explicit expressions for the values of such states, taken from their paper, can be found in Table \ref{expl_exp}. 

%The structure of the proof is as follows. We first prove the validity of expressions (\ref{eq1_1}) and (\ref{ineq1_1}). Then, we proceed to show the correctness of expressions (\ref{eq1_2}, \ref{eq1_3}, \ref{ineq1_2}) by induction over $i$. Finally, we prove that expressions (\ref{eq1_4}, \ref{ineq1_3}, \ref{ineq1_4}) hold by means of an induction argument of the following structure:

%\begin{enumerate}
%    \item \textbf{Induction base:} first we show that the concerned expressions hold for states $(4, j)$, where $j \leq 4$. 
%    \item \textbf{Induction hypothesis:} we proceed to assume that the expressions hold for states $(k-1, j)$, where $k \geq 5$ and $j \leq k-1$. We refer to this assumption as the \textit{global induction hypothesis}.
%    \item \textbf{Induction step:} we show that the global induction hypothesis implies the validity of the expressions for states $(k, j)$, $j \geq k$ by means of an embedded induction argument over $j$.  
%\end{enumerate}

\paragraph{Proof of expression (\ref{eq1_1}):}
We start by proving expression (\ref{eq1_1}). Using recursive expressions (\ref{rec_4}), (\ref{rec_6}) and (\ref{rec_7}), we obtain
\begin{align*}
    v^{\pi_1}_{\lambda}(2, 0) &= \dfrac{3\lambda}{(1-\lambda)(4-\lambda)}, \\
    v^{\pi_1}_{\lambda}(3, 0) &= \dfrac{3\lambda(19+3\lambda)}{2(4-\lambda)(1-\lambda)(18-7\lambda)}, \\
    v^{\pi_1}_{\lambda}(2, 2) &= \dfrac{9\lambda^2}{(1-\lambda)(4-\lambda)^2}, \\
    v^{\pi_1}_{\lambda}(3, 2) &= \dfrac{9\lambda^2(19+3\lambda)}{2(18-7\lambda)(1-\lambda)(4-\lambda)^2}.
\end{align*}
%These are also provided in Table \ref{expl_exp}. 
Inserting these in expression (\ref{eq1_1}) yields the desired result.

\paragraph{Proof of expression (\ref{eq1_2}):}
We proceed to show the validity of expression (\ref{eq1_2}) for all $i \geq 4$, by means of induction over $i$. First, we show that it holds for $i = 4$. Expression (\ref{rec_1}) yields
\begin{align*}
    v^{\pi_1}_{\lambda}(4,0) &= \dfrac{3\lambda^2(19+3\lambda)}{(4-\lambda)(3-\lambda)(1-\lambda)(18-7\lambda)}, \\
    v^{\pi_1}_{\lambda}(4,2) &= \dfrac{9\lambda^3(19+3\lambda)}{(18-7\lambda)(1-\lambda)(3-\lambda)(4-\lambda)^2}.
\end{align*}
%which is also given in Table \ref{expl_exp}.
Inserting these and the explicit expression for states $(3, 2)$ into expression (\ref{eq1_2}) yields the desired result for $i = 4$. Now, assume that expression (\ref{eq1_2}) holds for $i = k-1$ for some $k \geq 5$. This implies for $i = k$, using expression (\ref{rec_1}),
\begin{equation*}
    v^{\pi_1}_{\lambda}(k, 2) + 8v^{\pi_1}_{\lambda}(k-1,2) - 9v^{\pi_1}_{\lambda}(k, 0) = \dfrac{2\lambda}{3-\lambda}(v^{\pi_1}_{\lambda}(k-1, 2) + 8v^{\pi_1}_{\lambda}(k-2,2) - 9v^{\pi_1}_{\lambda}(k-1, 0)) = 0.
\end{equation*}
Hence, expression (\ref{eq1_2}) holds for all $i \geq 4$. 

\paragraph{Proof of expression (\ref{eq1_3}):}
We now show the validity of expression (\ref{eq1_3}) for all $i \geq 4$ in a similar way. Using expression (\ref{rec_8}), we obtain
\begin{align*}
    v^{\pi_1}_{\lambda}(4, 3) &= \dfrac{9\lambda^3(19+3\lambda)^2}{2(18-7\lambda)^2(4-\lambda)^2(3-\lambda)(1-\lambda)}.
\end{align*}
%which can also be found in Table \ref{expl_exp}. 
Inserting this and the explicit expressions for states $(4, 2)$ and $(4, 0)$ in expression (\ref{eq1_3}) yields the desired result for $i = 4$. Now, assume that expression (\ref{eq1_3}) is true for $i = k-1$ for some $k \geq 5$. Using this hypothesis and expression (\ref{rec_1}) now yields
\begin{align*}
    &-8v^{\pi_1}_{\lambda}(k, 3) + 96v^{\pi_1}_{\lambda}(k-1, 3) - 31v^{\pi_1}_{\lambda}(k, 2) - 57v^{\pi_1}_{\lambda}(k, 0) \\
    &= \dfrac{2\lambda}{3-\lambda}(-8v^{\pi_1}_{\lambda}(k-1, 3) + 96v^{\pi_1}_{\lambda}(k-2, 3) - 31v^{\pi_1}_{\lambda}(k-1, 2) - 57v^{\pi_1}_{\lambda}(k-1, 0)) = 0.
\end{align*}
This establishes the validity of expression (\ref{eq1_3}) for all $i \geq 4$. 

\paragraph{Proof of expression (\ref{eq1_4}):}
Invoking expression (\ref{rec_1}), the validity of expression (\ref{eq1_4}) follows easily from
\begin{equation*}
    \dfrac{2\lambda}{3-\lambda} v^{\pi_1}_{\lambda}(i-1,j) = v^{\pi_1}_{\lambda}(i,j) = \dfrac{2\lambda}{3-\lambda}v^{\pi_1}_{\lambda}(i, j-1). 
\end{equation*}

\paragraph{Proof of expression (\ref{ineq1_1}):}
Now, consider expression (\ref{ineq1_1}). Using expression (\ref{rec_8}), we obtain
\begin{equation*}
    v^{\pi_1}_{\lambda}(3, 3) = \dfrac{9\lambda^2(19+3\lambda)^2}{4(18-7\lambda)^2(4-\lambda)^2(1-\lambda)}.
\end{equation*}
%which is also given in Table \ref{expl_exp}. 
Inserting this and the explicit expressions for states $(3,2)$, $(2, 2)$, $(2, 0)$ and $(3, 0)$ yields the validity of expression (\ref{ineq1_1}) for $j = 2, 3$.

\paragraph{Proof of expression (\ref{ineq1_2}):}
We proceed to consider expression (\ref{ineq1_2}), which we again prove by means of induction over $i$. Inserting the explicit expressions for states $(3, 3)$, $(3, 2)$ and $(3, 0)$ yields the validity of expression (\ref{ineq1_2}) for $i = 3$. Assume now that the inequality holds for $i = k-1$ for some $k \geq 4$. This, in combination with expression (\ref{rec_1}), implies
\begin{equation*}
    8v^{\pi_1}_{\lambda}(k, 3) - 65v^{\pi_1}_{\lambda}(k, 2) + 57v^{\pi_1}_{\lambda}(k, 0) = \dfrac{2\lambda}{3-\lambda}(8v^{\pi_1}_{\lambda}(k-1, 3) - 65v^{\pi_1}_{\lambda}(k-1, 2) + 57v^{\pi_1}_{\lambda}(k-1, 0)) > 0.
\end{equation*}
Thus, expression (\ref{ineq1_2}) holds for all $i \geq 3$.

\paragraph{Proof of expression (\ref{ineq1_3}):}
To prove expression (\ref{ineq1_3}), we first compute, using expression (\ref{rec_1}),
\begin{equation*}
    v^{\pi_1}_{\lambda}(4, 4) = \dfrac{9\lambda^4(19+3\lambda)^2}{(18-7\lambda)^2(4-\lambda)^2(3-\lambda)^2(1-\lambda)}.
\end{equation*}
%which is also provided in Table \ref{expl_exp}.
Now, we use the explicit expressions for states $(4, 2)$, $(3, 2)$, $(2, 2)$, $(4, 3)$, $(3, 3)$ and $(4, 4)$ to verify expression (\ref{ineq1_3}) for states $(4, j)$, $j \leq 4$. We proceed to assume that expression (\ref{ineq1_3}) holds for all states $(i, j)$, $i \leq k-1$ and $j \leq i$, for some $k \geq 5$. We show that this induction hypothesis implies the correctness of expression (\ref{ineq1_3}) for states $(k, j)$, $j \leq k$. We distinguish between the cases $j \leq k-1$ and $j = k$. For $j \leq k-1$, we obtain, using expression (\ref{rec_1}),
    \begin{align*}
        &-6v^{\pi_1}_{\lambda}(k,j) + 11v^{\pi_1}_{\lambda}(k-1, j) - 5v^{\pi_1}_{\lambda}(k-2, j)\\
        &= \dfrac{2\lambda}{3-\lambda}(-6v^{\pi_1}_{\lambda}(k-1,j) + 11v^{\pi_1}_{\lambda}(k-2,j) - 5v^{\pi_1}_{\lambda}(k-3, j)) > 0, 
    \end{align*}
    by the induction hypothesis. 

    For $j = k$, on the other hand, applying expressions (\ref{rec_1}) and (\ref{rec_3}) yields
    \begin{align*}
        &-6v^{\pi_1}_{\lambda}(k, k) + 11v^{\pi_1}_{\lambda}(k-1, k) - 5v^{\pi_1}_{\lambda}(k-2, k) \\
        &= \dfrac{4\lambda^2}{(3-\lambda)^2}(-6v^{\pi_1}_{\lambda}(k-1, k-1) + 11v^{\pi_1}_{\lambda}(k-2, k-1) - 5v^{\pi_1}_{\lambda}(k-3, k-1)) > 0,
    \end{align*}
    by the induction hypothesis. Thus, we established the correctness of expression (\ref{ineq1_3}) for all $i \geq 4$, $j \geq 2$.

\paragraph{Proof of expression (\ref{ineq1_4}):}
    We prove the correctness of expression (\ref{ineq1_4}) by means of a similar argument to that used for expression (\ref{ineq1_3}). First, we verify the correctness of the expression for state $(4, 4)$, using the explicit expressions for states $(4, 4)$, $(4, 3)$ and $(4, 2)$ 
    provided above.
    %in Table \ref{expl_exp}. 
    Now, we assume that expression (\ref{ineq1_4}) is valid for all $(i,j)$, $4 \leq i \leq k-1$, $4 \leq j \leq i$, for some $k \geq 5$. We show that this hypothesis implies that the expression holds for all states $(k,j)$, $4 \leq j \leq k$. We again distinguish between the cases $j \leq k-1$ and $j = k$. For $j \leq k-1$, we obtain, using expression (\ref{rec_1}),
    \begin{align*}
        &-6v^{\pi_1}_{\lambda}(k,j) + 11v^{\pi_1}_{\lambda}(k, j-1) - 5v^{\pi_1}_{\lambda}(k, j-2)\\
        &= \dfrac{2\lambda}{3-\lambda}(-6v^{\pi_1}_{\lambda}(k-1, j) + 11v^{\pi_1}_{\lambda}(k-1, j-1) - 5v^{\pi_1}_{\lambda}(k-1, j-2)) > 0, 
    \end{align*}
    by the induction hypothesis.
    For $j = k$, using expressions (\ref{rec_1}) and (\ref{rec_3}) yields
    \begin{align*}
       &-6v^{\pi_1}_{\lambda}(k,j) + 11v^{\pi_1}_{\lambda}(k, j-1) - 5v^{\pi_1}_{\lambda}(k, j-2) \\
       &= \dfrac{4\lambda^2}{(3-\lambda)^2}(-6v^{\pi_1}_{\lambda}(k-1, j-1) + 11v^{\pi_1}_{\lambda}(k-1, j-2) - 5v^{\pi_1}_{\lambda}(k-1, j-3)) > 0, 
    \end{align*}
    by the induction hypothesis. It follows that expression (\ref{ineq1_4}) holds for all $i, j \geq 4$. 
    
    This concludes the proof for the range $\lambda \in (\lambda_c, 1)$.
The case
        $\lambda \in (0, \lambda_c]$
        can be treated similarly.
    \color{black}
\end{proof}

%\section{Optimal growth for non-striped seeds}
%\label{s:nostripes}

%\subsection{Optimal growth mechanism in case of a strip and a rectangle}

%In this section, we analyze the structure of the optimal policy for robust configurations consisting of two droplets of $+$-spins, of which one forms a strip that wraps around the torus and the other takes the shape of a rectangle. We denote the set of such configurations by $U^{(2)}_y$. 

%\subsubsection{The auxiliary MDP}

%We again construct an auxiliary MDP for this case, which is denoted by $(S^y, A^y, P^y, r^y)$. Here, the state space $S^y$ is defined as 
%\begin{equation*}
%    S^y = \{(i,j,k)|i,j,k = 0, 2, 3, \ldots, N\}.
%\end{equation*}
%A state $(i, j, k) \in S^y$ is a representation of the set of configurations in which the distances between the strip and the rectangle equal $i$ and $j$, and the distance between the two boundaries of the rectangle equals $k$. Here, a state $(i,j,0) \in S^y$ is equivalent to state $(i,j) \in S^x$, in the sense that they represent the same set of configurations. 

\section{Conclusions}
\label{s:conclusion}
We examined the optimization of lattice growth 
under spatial constraints by formulating the two--seed Ising 
dynamics as a Markov decision process. Using the zero--temperature 
Metropolis dynamics as the underlying evolution, we showed 
how carefully timed external actions can steer the system 
efficiently toward the absorbing all--plus state. Our analysis 
of the stripe--stripe, stripe--droplet, and droplet--droplet 
regimes revealed that optimal policies depend sensitively on 
both geometry and the discount factor. In the stripe--stripe 
case, we identified a sharp transition at the critical value 
$\lambda_{c}=15/17$, marking a switch from next-to-nearest-neighbor to nearest-neighbor preferred growth.

The stripe--droplet and droplet--droplet regimes exhibited 
similar qualitative behaviors, although the competition among 
growth geometries is richer. Simulations show that rapid 
front expansion generally accelerates absorption, while 
diagonal growth between separated droplets emerges as the 
most efficient coalescence mechanism. Policies acting on 
wider regions tend to be less efficient, emphasizing the 
importance of carefully targeting interventions based on 
both spatial configuration and temporal priorities. These 
results underline how MDPs provide a systematic framework 
to evaluate and select optimal strategies in stochastic 
spatial systems.

Our findings highlight the versatility of the MDP approach and 
suggest several directions for future work. These include 
extending the analysis to higher dimensions, studying the 
interaction of multiple droplets, and incorporating partial 
observability or explicit control costs. Moreover, 
reinforcement--learning algorithms may offer approximate 
optimal policies for larger and more complex state spaces, 
connecting naturally with metastability theory and providing 
insights into sequential intervention strategies in materials 
science, microbial biofilms, and other spatially extended 
stochastic processes.
Future studies could further 
exploit the connection with bootstrap percolation to 
characterize critical thresholds and universal growth patterns 
under minimal control interventions.

\bigskip
\par\noindent
\textbf{Acknowledgments}
\par\noindent
ENMC thanks the PRIN 2022 project
``Mathematical Modelling of Heterogeneous Systems (MMHS)",
financed by the European Union - Next Generation EU,
CUP B53D23009360006, Project Code 2022MKB7MM, PNRR M4.C2.1.1.
This work was carried out under the auspices of the
Italian National Group of Mathematical Physics (GNFM).
ENMC also thanks the Department of Mathematics of the Utrecht 
University.\\
MCJ thanks ENMC and the Dipartimento SBAI of Sapienza Università di Roma for their kind hospitality while conducting this research. MCJ is also grateful for financial support from the Erasmus+ programme for her stay in Rome.

\color{black}

\printbibliography

\end{document}